\documentclass{amsart}

\usepackage{amssymb}
\usepackage{url}
\usepackage{calrsfs}
\newcommand{\F}{{\mathbb{F}}}
\newcommand{\G}{{\mathbb{G}}}

\newcommand{\Q}{{\mathbb{Q}}}
\newcommand{\K}{{\mathbb{K}}}
\newcommand{\bO}{{\mathbb{O}}}
\newcommand{\R}{{\mathbb{R}}}
\newcommand{\Z}{{\mathbb{Z}}}
\newcommand{\C}{{\mathbb{C}}}

\newcommand{\cA}{{\mathcal{A}}}
\newcommand{\cC}{{\mathcal{C}}}

\newcommand{\cL}{{\mathcal{L}}}

\newcommand{\cP}{{\mathcal{P}}}
\newcommand{\cR}{{\mathcal{R}}}
\newcommand{\cS}{{\mathcal{S}}}

\newcommand{\SG}{{\mathfrak{S}}}

\newcommand{\bk}{{\boldsymbol{k}}}
\newcommand{\bkp}{{\boldsymbol{k}^+}}
\newcommand{\bkm}{{\boldsymbol{k}^\times}}
\newcommand{\bkn}{{\boldsymbol{k}^n}}
\newcommand{\bB}{{\boldsymbol{B}}}
\newcommand{\bG}{{\boldsymbol{G}}}
\newcommand{\bH}{{\boldsymbol{H}}}

\newcommand{\bK}{{\boldsymbol{K}}}

\newcommand{\bN}{{\boldsymbol{N}}}

\newcommand{\bS}{{\boldsymbol{S}}}
\newcommand{\bT}{{\boldsymbol{T}}}
\newcommand{\bW}{{\boldsymbol{W}}}
\newcommand{\bU}{{\boldsymbol{U}}}
\newcommand{\bV}{{\boldsymbol{V}}}
\newcommand{\bX}{{\boldsymbol{X}}}
\newcommand{\bY}{{\boldsymbol{Y}}}
\newcommand{\bZ}{{\boldsymbol{Z}}}

\newcommand{\Hom}{{\operatorname{Hom}}}
\newcommand{\Aut}{{\operatorname{Aut}}}

\newcommand{\Irr}{{\operatorname{Irr}}}

\newcommand{\cRad}{\mathcal{R}_{\operatorname{ad}}}
\newcommand{\cRsc}{\mathcal{R}_{\operatorname{sc}}}
\newcommand{\bGad}{\boldsymbol{G}_{\operatorname{ad}}}
\newcommand{\bGsc}{\boldsymbol{G}_{\operatorname{sc}}}
\newcommand{\bGss}{\boldsymbol{G}_{\operatorname{ss}}}
\newcommand{\bGder}{\boldsymbol{G}_{\operatorname{der}}}
\newcommand{\id}{{\operatorname{id}}}

\newcommand{\trp}{{\operatorname{tr}}}

\newcommand{\GL}{{\operatorname{GL}}}
\newcommand{\GU}{{\operatorname{GU}}}
\newcommand{\PGL}{{\operatorname{PGL}}}
\newcommand{\SL}{{\operatorname{SL}}}
\newcommand{\SU}{{\operatorname{SU}}}
\newcommand{\SO}{{\operatorname{SO}}}
\newcommand{\Sp}{{\operatorname{Sp}}}
\newcommand{\Spin}{{\operatorname{Spin}}}

\renewcommand{\leq}{\leqslant}
\renewcommand{\geq}{\geqslant}

\newcommand{\nm}[1]{{\it{#1\index{#1}}}}
\newcommand{\nmi}[2]{{\it{#1\index{#2}}}}
\newcommand{\nms}[2]{{\it{#1\index{#2@#1}}}}

\def\dddots{\mathinner{\mkern1mu\raise1pt
    \vbox{\kern7pt\hbox{.}}\mkern2mu
    \raise4pt\hbox{.}\mkern2mu\raise7pt\hbox{.}\mkern1mu}}

\makeatletter
\renewenvironment{proof}[1][\proofname]{\par
  \pushQED{\qed}%
  \normalfont \topsep6\p@\@plus6\p@\relax
  \trivlist
  \itemindent\normalparindent
  \item[\hskip\labelsep
        \scshape
    #1\@addpunct{.}]\ignorespaces
}{%
  \popQED\endtrivlist\@endpefalse
}
\makeatother

\makeindex

\def\daytime{
  \count100=\time             
  \divide\count100 by 60      
  \count101=\count100
  \multiply\count101 by -60
  \count102=\time
  \advance\count102  by \count101
  \number\count100:
  \ifcase                     
    \count102 00\or 01\or 02\or 03\or 04\or 05\or 06\or 07\or 08\or 09
  \else
     \number\count102
  \fi
}
\def\today{\ifcase\month\or
        January\or February\or March\or April\or May\or June\or
        July\or August\or September\or October\or November\or December\fi
        \space\number\day, \number\year}

\newtheorem{thm}{Theorem}[section]
\newtheorem{prop}[thm]{Proposition}
\newtheorem{lem}[thm]{Lemma}
\newtheorem{cor}[thm]{Corollary}

\theoremstyle{definition}
\newtheorem{defn}[thm]{Definition}
\newtheorem{exmp}[thm]{Example}

\newtheorem{abs}[thm]{}

\theoremstyle{remark}
\newtheorem{rem}[thm]{Remark}

\begin{document}
\title{Reductive groups and Steinberg maps}

\author{Meinolf Geck}

\address{IAZ -- Lehrstuhl f\"ur Algebra, Universit\"at Stuttgart,
Pfaffenwaldring 57, 70569 Stuttgart, Germany}
\email{meinolf.geck@mathematik.uni-stuttgart.de}

\author{Gunter Malle}
\address{FB Mathematik, TU Kaiserslautern, Postfach 3049,
         67653 Kaisers\-lautern, Germany.}
\email{malle@mathematik.uni-kl.de}

\date{\today}

\begin{abstract}
This is a preliminary version of the first chapter of a book project
on the character theory of finite groups of Lie type. It provides the
foundations from the general theory of reductive algebraic groups over
a finite field. 
\end{abstract}

\maketitle

\tableofcontents



This first chapter is of a preparatory nature; its purpose is to collect 
some basic results about algebraic groups (with proofs where appropriate) 
which will be needed for the discussion of characters and applications in 
later chapters. In particular, one of our aims is to arrive at the point 
where we can give a precise definition of a ``series of finite groups of 
Lie type'' $\{\bG(q)\}$, indexed by a parameter $q$. We also introduce a 
number of tools which will be helpful in the discussion of examples. 

For a reader who is familiar with the basic notions about algebraic groups, 
root data and Frobenius maps it may just be sufficient to browse through this
chapter on a first reading, in order to see some of our notation. There are,
however, a few topics and results which are frequently used in the literature
on algebraic groups and finite groups of Lie type, but for which we have 
found the coverage in standard reference texts (like \cite{Bor}, 
\cite{Ca2}, \cite{DiMi2}, \cite{Hum}, \cite{Spr}) not to be sufficient; 
these will be treated here in a fairly self-contained manner.

Section~\ref{sec00} is purely expository: it introduces affine varieties, 
linear algebraic groups in general, and the first definitions
concerning reductive algebraic groups. 

In Section~\ref{sec:rootdata}, we consider in some detail (abstract) root 
data, the basic underlying combinatorial structure of the theory of reductive 
algebraic groups. In particular, we present the approach of \cite{BruLu}, in 
which root data simply appear as factorisations of the Cartan matrix of a 
root system. This provides, first of all, an efficient procedure for 
constructing root data from Cartan matrices; secondly, it will be extremely 
useful for computational purposes and the discussion of examples. 

Section~\ref{sec:chevalley} contains the fundamental existence and 
isomorphism theorems of Chevalley \cite{Chev}, \cite{Ch05} concerning 
connected reductive algebraic groups. We also state the more general 
``isogeny theorem'' and present some of its basic applications. (There
is now quite a short proof available, due to Steinberg \cite{St2}.) An 
important class of homomorphisms of algebraic groups to which this more 
general theorem applies are the Steinberg maps, to be discussed in detail
in Section~\ref{sec:steinberg}. 

Following \cite{St68}, one might just define a Steinberg map of a connected 
reductive algebraic group $\bG$ to be an endomorphism whose fixed point set 
is finite. But it will be important and convenient to single out a certain 
subclass of such morphisms to which one can naturally attach a positive real
number~$q$ (some power of which is a prime power) and such that one can 
speak of the corresponding finite group $\bG(q)$. The known results on
Frobenius and Steinberg maps are somewhat scattered in the literature so we 
treat this in some detail here, with complete proofs. 

In Section~\ref{sec:semisimple}, we illustrate the material developed so 
far by a number of further basic constructions and examples. In 
Section~\ref{sec:generic}, we show how all this leads to the notion of 
``generic'' reductive groups, in which $q$ will appear as a formal 
parameter. Finally, Section~\ref{sec:regemb} discusses in some detail the 
first applications to the character theory of finite groups of Lie type:
the ``Multiplicity--Freeness'' Theorem~\ref{multfree}. 

\section{Affine varieties and algebraic groups} \label{sec00}

In this section, we introduce some basic notions concerning affine
varieties and algebraic groups. We will do this in a somewhat informal
way, assuming that the reader is willing to fill in some details from
text books like \cite{Bor}, \cite{Ca2}, \cite{mybook}, \cite{Hum}, 
\cite{MaTe}, \cite{Spr}.

\begin{abs} {\bf Affine varieties.} \label{subsec11}
Let $k$ be a field and let $\bX$ be a set. Let $A$ be a subalgebra of 
the $k$-algebra $\cA(X,k)$ of all functions $f\colon \bX\rightarrow 
k$. Using $A$ we can try to define a topology on $\bX$: a subset $\bX'
\subseteq \bX$ is called closed if there is a subset $S\subseteq A$ such 
that $\bX'=\{x\in \bX\mid f(x)=0\mbox{ for all $f\in S$}\}$. This works 
well, and gives rise to the \nm{Zariski topology} on $\bX$, if $A$ is 
neither too small nor too big. The precise requiremente are 
(see \cite{cart}):
\begin{itemize}
\item[(1)] $A$ is finitely generated as a $k$-algebra and contains 
the identity of $\cA(X,k)$;
\item[(2)] $A$ separates points (that is, given $x\neq x'$ in $\bX$, there
exist some $f\in A$ such that $f(x)\neq f(x')$); 
\item[(3)] any $k$-algebra homomorphism $\lambda \colon A \rightarrow k$ 
is given by evaluation at a point (that is, there exists some $x\in \bX$ 
such that $\lambda(f)=f(x)$ for all $f\in A$). 
\end{itemize}
A pair $(\bX,A)$ satisfying the above conditions will be called an 
\nm{affine variety} {\em over~$k$}; the functions in $A$ are called the 
\nm{regular functions} on $\bX$. We define $\dim X$ to be the supremum of all 
$r \geq 0$ such that there exist $r$ algebraically independent elements in 
$A$. Since $A$ is finitely generated, $\dim \bX <\infty$. (See 
\cite[1.2.18]{mybook}.) If $A$ is an integral domain, then $\bX$ is
called  {\em irreducible}.

There is now also a natural notion of morphisms. Let $(\bX,A)$ and $(\bY,
B)$ be affine varieties over $k$. A map $\varphi \colon \bX\rightarrow\bY$ 
will be called a {\em morphism} if composition with $\varphi$ maps $B$ 
into $A$ (that is, for all $g\in B$, we have $\varphi^*(g):=g\circ \varphi 
\in A$); in this case, $\varphi^*\colon B\rightarrow A$ is an algebra
homomorphism, and every algebra homomorphism $B\rightarrow A$ arises in 
this way. The morphism $\varphi$ is an {\em isomorphism} if there is a 
morphism $\psi \colon \bY \rightarrow \bX$ such that $\psi\circ \varphi=
\id_{\bX}$. (Equivalently: the induced algebra homomorphism $\varphi^*
\colon B \rightarrow A$ is an isomorphism.) 

Starting with these definitions, the basics of (affine) algebraic geometry 
are developed in \cite{St74}, and this is also the approach taken in 
\cite{mybook}. The link with the more traditional approach via closed 
subsets in affine space (which, when considered as an algebraic set with
the Zariski topology, we denote by $\bkn$) is obtained as follows. Let 
$(\bX,A)$ be an affine variety over $k$. Choose a set $\{a_1,\ldots,a_n\}$ 
of algebra generators of $A$ and consider the polynomial ring $k[t_1,
\ldots,t_n]$ in $n$ independent indeterminates $t_1,\ldots,t_n$. There is 
a unique algebra homomorphism $\pi\colon k[t_1,\ldots,t_n] \rightarrow A$ 
such that $\pi(t_i)=a_i$ for $1\leq i \leq n$. Then we have a morphism
\[ \varphi\colon \bX \rightarrow \bkn,\qquad x \mapsto (a_1(x),
\ldots,a_n(x)),\]
such that $\varphi^*=\pi$. The image of $\varphi$ is the ``Zariski closed''
set of $\bkn$ consisting of all $(x_1,\ldots,x_n)\in \bkn$ such that
$f(x_1,\ldots,x_n)=0$ for all $f\in \ker(\pi)$.

To develop these matters any further, it is then essential to assume that 
$k$ is algebraically closed, which we will do from now on. One can go a 
long way towards those parts of the theory which are relevant 
for algebraic groups, once the following basic result about morphisms is 
available (see \cite[\S 1.13]{St74}, \cite[\S 2.2]{mybook}):

{\em Let $\varphi\colon \bX\rightarrow \bY$ be a morphism between 
irreducible affine varieties such that $\varphi(\bX)$ is dense in $\bY$. Then 
there is a non-empty open subset $\bV \subseteq \bY$ such that $\bV\subseteq 
\varphi(\bX)$ and, for all $y\in\bV$, we have $\dim\varphi^{-1}(y)= 
\dim \bX-\dim \bY$.}
\end{abs}

\begin{abs} {\bf Algebraic groups.} \label{subsec12}
In order to define algebraic groups, we need to know that direct products
of affine varieties are again affine varieties. So let $(\bX,A)$ and 
$(\bY,B)$ be affine varieties over $k$. Given $f \in A$ and $g\in B$, we 
define the function $f\otimes g\colon \bX\times \bY\rightarrow k$, 
$(x,y) \mapsto f(x)g(y)$. Let $A\otimes B$ be the subspace of 
$\cA(\bX\times \bY,k)$ spanned by all $f\otimes g$ where $f\in A$
and $g\in B$. Then $A\otimes B$ is a subalgebra of $\cA(\bX\times
\bY,k)$ (isomorphic to the tensor product of $A$, $B$ over $k$) and the pair 
$(\bX \times\bY, A\otimes B)$ is easily seen be an affine variety 
over~$k$. Now let $(\bG, A)$ be an affine variety and 
assume that $\bG$ is an abstract group where multiplication and inversion 
are defined by maps $\mu\colon \bG\times \bG \rightarrow \bG$ and 
$\iota\colon \bG\rightarrow\bG$. Then we say that $\bG$ is an \nm{affine 
algebraic group} if $\mu$ and $\iota$ are morphisms. The first example is 
the additive group of $k$ which, when considered as an algebraic group, 
we denote by~$\bkp$ (with algebra of regular functions given by the 
polynomial functions $k\rightarrow k$). 

Most importantly, the group $\GL_n(k)$ ($n\geq 1$), is an affine algebraic 
group, with algebra of regular functions given as follows. For $1\leq i,j
\leq n$ let $f_{ij}\colon \GL_n(k)\rightarrow k$ be the function which sends 
a matrix $g\in \GL_n(k)$ to its $(i,j)$-entry; furthermore, let $\delta
\colon \GL_n(k) \rightarrow k$, $g\mapsto \det(g)^{-1}$. Then the algebra of 
regular functions on $\GL_n(k)$ is the subalgebra of $\cA(\GL_n(k),k)$ 
generated by $\delta$ and all $f_{ij}$ ($1\leq i,j \leq n$). In particular, 
the muliplicative group $\bkm:=\GL_1(k)$ is an affine algebraic group. 

It is a basic fact that any affine algebraic group $\bG$ over $k$ is
isomorphic to a closed subgroup of $\GL_n(k)$, for some $n\geq 1$; see 
\cite[2.4.4]{mybook}. For this reason, an affine algebraic group is also 
called a \nm{linear algebraic group}. When we just write ``algebraic 
group'', we always mean an affine (linear) algebraic group.
\end{abs}

\begin{abs} {\bf Connected algebraic groups.} \label{subsec13}
A topological space is \nm{connected} if it cannot be written 
as a disjoint union of two non-empty open subsets. A linear algebraic
group $\bG$ can always be written as the disjoint union of finitely
many connected components, where the component containing the identity
element is a closed connected normal subgroup of $\bG$, denoted by
$\bG^\circ$; see \cite[1.3.13]{mybook}. Thus, $\bG$ is connected if and
only if $\bG=\bG^\circ$. (Equivalently: $\bG$ is irreducible as
an affine variety; see \cite[1.1.12, 1.3.1]{mybook}.) 

What is the significance of this fundamental notion? Every finite group 
$\bG$ can be regarded as a linear algebraic group, with algebra of regular 
functions given by all of $\cA(\bG,k)$. Thus, the study of {\em all} linear 
algebraic groups is necessarily more complicated than the study of all
finite groups. But, as Vogan \cite{VoE8} writes, ``a~miracle happens'' when 
we consider {\em connected} algebraic groups: things actually become much 
less complicated. One reason is that a connected algebraic group is almost 
completely determined by its Lie algebra (see \ref{subsec1tang} and also 
\ref{subsec1weights} below), and the latter can be studied using linear 
algebra methods. 

Combined with our assumption that $k$ is algebraically closed, this gives 
us some powerful tools. For example, matrices over algebraically closed
fields can be put in triangular form. An analogue of this fact for an 
arbitrary connected algebraic group is the statement that every 
element is contained in a Borel subgroup (that is, a maximal closed 
connected normal solvable subgroup); see \cite[3.4.9]{mybook}. 

A useful criterion for showing the connectedness of a subgroup of $\bG$ is
as follows. 

\smallskip
{\em Let $\{\bH_i\}_{i\in I}$ be a family of closed connected subgroups in 
$\bG$. Then the (abstract) subgroup $H=\langle \bH_i\mid i\in I\rangle
\subseteq \bG$ generated by this family is closed and connected; 
furthermore, we have $H=\bH_{i_1} \cdots \bH_{i_n}$ for some $n$ and 
$i_1,\ldots,i_n\in I$.}

\smallskip
The proof uses the result on morphisms mentioned at the end of 
\ref{subsec11}; see, e.g., \cite[2.4.6]{mybook}. Note that, if $\bU,\bV$ 
are any closed subgroups of $\bG$, then the abstract subgroup $\langle 
\bU, \bV\rangle \subseteq \bG$ need not even be closed. For example,
if $\bG=\SL_2(\C)$, then it is well-known that the subgroup $\SL_2(\Z)$
is generated by two elements of order $4$ and $6$, but this subgroup is 
certainly not closed in $\bG$. However, if $\bV$ is normalised by $\bU$,
then $\langle \bU,\bV\rangle=\bU.\bV$ is closed; see \cite[\S 3.3, 
Corollaire]{Ch05}.

We will use without further special mention some standard facts (whose 
proofs also rely on the above-mentioned result on morphisms). For example, 
if $f\colon \bG \rightarrow \bG'$ is a homomorphism of linear algebraic 
groups, then the image $f(\bG)$ is a closed subgroup of $\bG'$ (connected 
if $\bG$ is connected), the kernel of $f$ is a closed subgroup of $\bG$ 
and we have $\dim \bG=\dim \ker(f)+\dim f(\bG)$. (See, e.g., 
\cite[2.2.14]{mybook}.) 
\end{abs}

\begin{abs} {\bf Classical groups.} \label{subsecclassic} These form
an important class of examples of linear algebraic groups. They are
closed subgroups of $\GL_n(k)$ defined by certain quadratic polynomials 
corresponding to a bilinear or quadratic form on the underlying 
vector space $k^n$. There is an extensive literature on these groups; see,
e.g., \cite{bour9}, \cite{dieu}, \cite{Grove2}, \cite{dtay}. Since our
base field $k$ is algebraically closed, the general theory simplifies
considerably and we only need to consider three classes of groups, 
leading to the Dynkin types $B,C,D$. First, and quite generally, for 
any invertible matrix $Q_n\in M_n(k)$, we obtain a linear algebraic group 
\[\Gamma(Q_n,k) :=\{A\in M_n(k) \mid A^\trp Q_nA=Q_n\}; \]
note that $\det(A)=\pm 1$ for all $A\in \Gamma(Q_n,k)$. Let us now take 
$Q_n$ of the form 
\[Q_n=\left[\begin{array}{cccc} 0 & \cdots & 0 & \pm 1 \\ \vdots &\dddots &
\dddots & 0 \\ 0 &\pm  1 &\dddots & \vdots \\ \pm 1 & 0 & \cdots & 0 
\end{array} \right]\in M_n(k)\qquad (n\geq 2)\]
where the signs are such that $Q_n^\trp=\pm Q_n$. Then $Q_n$ is the 
matrix of a non-degenerate symmetric or alternating bilinear form on 
$k^n$; furthermore, $Q_n^{-1}=Q_n^\trp$ and $\Gamma(Q_n,k)$ will be 
invariant under transposing matrices. 

If $Q_n^\trp=-Q_n$ and $n$ is even, then $\Gamma(Q_n,k)$ will be denoted 
$\mbox{Sp}_n(k)$ and called the \nm{symplectic group}. This group is always 
connected; see \cite[1.7.4]{mybook}. 

Now assume that $Q_n^\trp=Q_n$ and that all signs in $Q_n$ are $+$. Then
we also consider the quadratic from on $k^n$ defined by the polynomial
\[ f_n:=\left\{\begin{array}{cl} t_1t_{2m+1}+t_2t_{2m}+\ldots +
t_mt_{m+2}+t_{m+1}^2 & \quad \mbox{if $n=2m+1$ is odd},\\
t_1t_{2m}+t_2t_{2m-2}+\ldots +t_mt_{m+1} & \quad \mbox{if $n=2m$ is even},
\end{array}\right.\]
(where $t_1,\ldots,t_n$ are indeterminates). This defines a function 
$\dot{f}_n\colon k^n\rightarrow k$, where we regard the elements of $k^n$ as 
column vectors. Then, using the notation in \cite[\S 1.2]{MaTe}, the 
\nm{general orthogonal group} is defined as 
\[ \mbox{GO}_{n}(k):=\{A\in M_n(k)\mid \dot{f}_n(Av)=\dot{f}_n(v) 
\mbox{ for all $v\in k^n$}\};\]
furthermore, $\mbox{SO}_{n}(k):=\mbox{GO}_{n}(k)^\circ$ will be called the
\nm{special orthogonal group}. In each case, we have 
$[\mbox{GO}_n(k): \mbox{SO}_n(k)]\leq 2$; see \cite[\S 1.7]{mybook}, 
\cite{Grove2} for further details. Note also that, if $\mbox{char}(k)
\neq 2$, then $\mbox{GO}_n(k)=\Gamma(Q_n,k)$; otherwise, $\mbox{GO}_n(k)$ 
will be strictly contained in $\Gamma(Q_n,k)$. (See also Example~\ref{isogBC}
for the case where $n$ is odd and $\mbox{char}(k)=2$.)

The particular choices of $Q_n$ and $f_n$ lead to simple descriptions of 
a $BN$-pair in $\mbox{Sp}_n(k)$ and $\mbox{SO}_n(k)$; see, e.g., 
\cite[\S 1.7]{mybook} (and also \ref{subsec16} below). The Dynkin types 
and dimensions are given as follows.
\[\renewcommand{\arraystretch}{1.2} \begin{array}{ccc} \hline \mbox{Group} 
& \mbox{Type} & \mbox{Dimension}\\\hline
\mbox{SO}_{2m{+}1}(k) & B_m & 2m^2+m\\
\mbox{Sp}_{2m}(k) & C_m & 2m^2+m\\
\mbox{SO}_{2m}(k) & D_m & 2m^2-m\\\hline \end{array}\]
\end{abs}

\begin{abs} {\bf Tangent spaces and the Lie algebra.} \label{subsec1tang} 
Let $(\bX,A)$ be an affine variety over $k$. Then the \nm{tangent space} 
$T_x(\bX)$ of $\bX$ at a point $x\in \bX$ is the set of all $k$-linear maps 
$D\colon A\rightarrow k$ such that $D(fg)=f(x)D(g)+g(x)D(f)$. (Such linear
maps are called {\em derivations}.) Clearly, $T_x(\bX)$ is a subspace of the
vector space of all linear maps from $A$ to $k$. Any $D\in T_x(\bX)$ is 
uniquely determined by 
its values on a set of algebra generators of $A$. Hence, since $A$ is 
finitely generated, we have $\dim T_x(\bX)<\infty$. If $\bX'\subseteq \bX$ 
is a closed subvariety, we have a natural inclusion $T_x(\bX')\subseteq 
T_x(\bX)$ for any $x\in \bX'$. For example, we can identify $T_x(\bkn)$
with $k^n$ for all $x\in \bkn$ and so, if $\bX\subseteq \bkn$ is a Zariski 
closed subset, we have $T_x(\bX)\subseteq k^n$ for all $x\in \bX$ (see 
\cite[1.4.10]{mybook}). 
More generally, any morphism $\varphi \colon \bX \rightarrow \bY$ between
affine varieties over $k$ naturally induces a linear map $d_x\varphi \colon 
T_x(\bX) \rightarrow T_{\varphi(x)}(\bY)$ for any $x \in \bX$, called the 
{\em differential} of $\varphi$ at~$x$.  (See \cite[\S 1.4]{mybook}.)

Now let $\bG$ be a linear algebraic group and denote $L(\bG):=T_1(\bG)$,
the tangent space at the identity element of $\bG$. Then
\[ L(\bG)=L(\bG^\circ) \qquad \mbox{and} \qquad \dim \bG=\dim L(\bG);\]
see \cite[1.5.2]{mybook}. Furthermore, there is a Lie product $[\;,\;]$ on 
$L(\bG)$ which can be defined as follows. Consider a realisation of $\bG$ 
as a closed subgroup of $\GL_n(k)$ for some $n\geq 1$. We have a natural
isomorphism of $L(\GL_n(k))$ onto $M_n(k)$, the vector space of all 
$n \times n$-matrices over $k$; see \cite[1.4.14]{mybook}. Hence we obtain
an embedding $L(\bG) \subseteq M_n(k)$ where $M_n(k)$ is endowed with
the usual Lie product $[A,B]=AB-BA$ for $A,B\in M_n(k)$. Then one shows that
$[L(\bG),L(\bG)] \subseteq L(\bG)$ and so $[\;,\;]$ restricts to a Lie 
product on $L(\bG)$; see \cite[1.5.3]{mybook}. (Of course, there is also an 
intrinsic description of $L(\bG)$ in terms of the algebra of regular 
functions on $\bG$ which shows, in particular, that the product does not 
depend on the choice of the realisation of $\bG$; see \cite[1.5.4]{mybook}.)
\end{abs}

\begin{abs} {\bf Quotients.}\label{subsec1quot} Let $\bG$ be a linear
algebraic group and $\bH$ be a closed normal subgroup. We have the abstract 
factor group $\bG/\bH$ and we would certainly like to know if this can also
be viewed as an algebraic group. More generally, let $\bX$ be an affine
variety and $\bH$ be a linear algebraic group such that we have a morphism
$\bH\times \bX \rightarrow \bX$ which defines an action of $\bH$ on $\bX$.
The question of whether we can view the set of orbits $\bX/\bH$ as an 
algebraic variety leads to ``geometric invariant theory''; in general, these 
are quite delicate matters. Let us begin by noting  that there is a natural 
candidate for the algebra of functions on the orbit set $\bX/\bH$: If $A$ 
is the algebra of regular functions on $\bX$, then 
\[A^\bH:=\{f\in A\mid f(h.x)=f(x) \mbox{ for all $h\in \bH$ and all
$x\in \bX$}\}\]
can naturally  be regarded as an algebra of $k$-valued functions on $\bX/\bH$.
However, the three properties in \ref{subsec11} will not be satisfied in
general. There are two particular situations in which this is the
case, and these will be sufficient for most parts of this book; these two
situations are:
\begin{itemize}
\item $\bH$ is a finite group, or 
\item $\bX=\bG$ is an algebraic group and $\bH$ is a closed normal
subgroup (acting by left multiplication).
\end{itemize}
(For the proofs, see \cite[5.25]{Fog} or \cite[2.5.12]{mybook} in the first
case, and \cite[2.26]{Fog} or \cite[\S 5.5]{Spr} in the second case.) Now
let us assume that ($\bX/\bH,A^\bH)$ is an affine variety. Then, first of 
all, the natural map $\bX\rightarrow \bX/\bH$ is a morphism of affine 
varieties; furthermore, we have the following universal property:

{\em If $\varphi \colon \bX \rightarrow \bY$ is any morphism of affine
varieties which is constant on the orbits of $\bH$ on $\bX$, then there
is a unique morphism $\bar{\varphi} \colon \bX/\bH\rightarrow \bY$
such that $\varphi$ is the composition of $\bar{\varphi}$ and the natural
map $\bX\rightarrow \bX/\bH$.}

(Indeed, if $B$ is the algebra of regular functions on $\bY$, then the
induced algebra homomorphism $\varphi^*\colon B\rightarrow A$ has image
in $A^\bH$, hence it factors through an algebra homomorphism
$\bar{\varphi}^*\colon B \rightarrow A^\bH$ for a unique morphism
$\bar{\varphi}\colon \bX/\bH\rightarrow \bY$.)

For example, if we are in the second of the above two cases, then the
universal property shows that the induced multiplication and inversion 
maps on $\bG/\bH$ are morphisms of affine varieties. Thus, $\bG/\bH$ is 
an affine algebraic group. 
\end{abs}

\begin{abs} {\bf Algebraic groups in positive characteristic.}
\label{subsec14}
The finite groups that we shall study in this book are obtained as 
\[ \bG^F:=\{g\in \bG \mid F(g)=g\}\]
where $F\colon \bG \rightarrow \bG$ are certain bijective endomorphisms
with finitely many fixed points, the so-called {\em Steinberg maps}. (This
will be discussed in detail in Section~\ref{sec:steinberg}.) Such maps $F$ 
will only exist if $k$ has prime characteristic. So we
will usually assume that $p$ is a prime number and $k=\overline{\F}_p$ is 
an algebraic closure of the field $\F_p=\Z/p\Z$. Now, algebraic geometry
over fields with positive characteristic is, in some respects, more tricky 
than algebraic geometry over $\C$, say (because of the inseparability 
of certain field extensions; see also \ref{subsec1wrong} below). However, 
some things are actually easier. For example, using an embedding of $\bG$ 
into some $\GL_n(k)$ as in \ref{subsec12}, we see that every element
$g\in\bG$ has finite order. Thus, we can define $g$ to be \nm{semisimple} 
if the order of $g$ is prime to $p$; we define $g$ to be \nm{unipotent}
if the order of $g$ is a power of $p$. Then, clearly, any $g\in \bG$ has 
a unique decomposition
\[ g=us=su \qquad \mbox{where $s\in \bG$ is semisimple and
$u\in \bG$ is unipotent},\]
called the \nm{Jordan decomposition of elements}. (The proof in
characteristic $0$ certainly requires more work; see \cite[\S 2.4]{Spr}.) 
Another example: An algebraic group $\bG$ is called a \nm{torus} if 
$\bG$ is isomorphic to a direct product of a finite number of copies of 
the multiplicative group $\bkm$. Then $\bG$ is a torus if and only if $\bG$ 
is connected, abelian and consists entirely of elements of order prime 
to~$p$; see \cite[3.1.9]{mybook}. (To formulate this in characteristic 
$0$, one would need the general definition of semisimple elements.)
\end{abs}

\begin{abs} {\bf Some things that go wrong in positive characteristic.}
\label{subsec1wrong} Here we collect a few items which show that, when 
working over a field $k=\overline{\F}_p$ as above, things may not work
as one might hope or expect. The first item is:
\begin{itemize}
\item {\em A bijective homomorphism of algebraic groups $\varphi
\colon \bG_1 \rightarrow \bG_2$ need not be an isomorphism.}
\end{itemize}
The standard example is the Frobenius map $\overline{\F}_p \rightarrow 
\overline{\F}_p$, $x \mapsto x^p$. (Note that, over $\C$, a bijective 
homomorphism between connected algebraic groups is an isomorphism; see 
\cite[11.1.16]{GoWa}.) A useful criterion is given as follows (see 
\cite[2.3.15]{mybook}):
\begin{itemize}
\item {\em A bijective homomorphism of algebraic groups $\varphi\colon \bG_1
\rightarrow \bG_2$ is an isomorphism if $\bG_1,\bG_2$ are connected and if
the differential $d_1\varphi \colon T_1(\bG_1) \rightarrow T_1(\bG_2)$ 
between the tangent spaces is an isomorphism}.
\end{itemize}
The next item concerns the Lie algebra of an algebraic group. Let 
$\bG$ be a linear algebraic group and $\bU,\bH$ be closed subgroups of
$\bG$. As already noted in \ref{subsec1tang}, we have natural inclusions 
of $L(\bU)$, $L(\bH)$ and $L(\bU\cap \bH)$ into $L(\bG)$. It is 
always true that $L(\bU\cap \bH) \subseteq L(\bU)  \cap L(\bH)$. 
\begin{itemize}
\item {\em When considering the intersection of closed subgroups $\bU,\bH$
of an algebraic group $\bG$, it is not always true that $L(\bU\cap \bH)=
L(\bU) \cap L(\bH)$.}
\end{itemize}
A good example to keep in mind is as follows. Let $\bG=\GL_n
(\overline{\F}_p)$, $\bH=\SL_n(\overline{\F}_p)$ and $\bZ$ be the center of 
$\bG$ (the scalar matrices in $\bG$). Then $\bZ,\bH$ are closed subgroups of 
$\bG$. As in \ref{subsec1tang}, we can identify $L(\bG)=M_n(k)$; then 
$L(\bH)$ consists of all matrices of trace $0$ and $L(\bZ)$ consists of all 
scalar matrices. (For these facts see, for example, \cite[\S 1.5]{mybook}). 
Assume now that $p$ divides $n$. Then, clearly, $L(\bZ) \subseteq L(\bH)$,
whereas $\bZ\cap \bH$ is finite and so $L(\bZ\cap \bH)=L((\bZ\cap 
\bH)^\circ)=\{0\}$. (This phenomenon can not happen in characteristic~$0$; 
see \cite[6.12]{Bor} or \cite[12.5]{Hum}.) Closely related to the above 
item is the next item: semidirect products. Let $\bG$ be an algebraic group 
and $\bU,\bH$ be closed subgroups such that $\bU$ is normal, $\bG=\bU.\bH$ 
and $\bU\cap\bH =\{1\}$. Following \cite[1.11]{Bor}, we say that $\bG$ is the 
\nm{semidirect product (of algebraic groups)} of $\bU,\bH$ if the natural 
map $\bU \times \bH\rightarrow \bG$ given by multiplication is an
isomorphism of affine varieties. If this holds, we have an inverse
isomorphism $\bG \rightarrow \bU \times \bH$ and the second projection
will induce an isomorphism of algebraic groups $\bG/\bU\cong \bH$.
\begin{itemize}
\item {\em In the definition of semidirect products of algebraic
groups, the assumption that $\bU \times \bH\rightarrow \bG$ is an 
isomorphism of affine varieties can not be omitted.}
\end{itemize}
Take again the above example where $\bG=\GL_n(\overline{\F}_p)$, $\bH=
\SL_n(\overline{\F}_p)$ and $\bZ$ is the center of $\bG$ (the scalar 
matrices in $\bG$). Assume now that $n=p$. Then $\bZ,\bH$ are closed 
connected normal subgroups such that $\bG=\bZ.\bH$ and $\bZ\cap\bH=\{1\}$. 
However, this is not a semidirect product of algebraic groups! For, if it 
were, then we would have an induced isomorphism $\SL_p(\overline{\F}_p)=
\bH\cong \bG/\bZ=\PGL_p(\overline{\F}_p)$ which does not exist, as we
will see later in Example~\ref{rootdatPGL}. 
\end{abs}

\begin{abs} {\bf The unipotent radical.} \label{subsec14a} 
Let $\bG$ be a linear algebraic group over $k=\overline{\F}_p$, where $p$
is a prime number. We can now also define the \nm{unipotent radical} 
$R_u(\bG)\subseteq \bG$, as follows. An abstract subgroup of $\bG$ is 
called {\em unipotent} if all of its elements are unipotent. Since every 
element in $\bG$ has finite order, one easily sees that the product of two 
normal unipotent subgroups is again a normal unipotent subgroup of $\bG$. 
If $\bG$ is finite, then this immediately shows that there is a unique 
maximal normal unipotent subgroup in $\bG$. (In the theory of finite groups, 
this is denoted $O_p(\bG)$.) In the general case, we define 
\[ R_u(\bG)\,:=\,\mbox{subgroup of $\bG$ generated by all $\bU \in 
\cS_{\text{unip}}(\bG)$},\]
where $\cS_{\text{unip}}(\bG)$ denotes the set of all closed connected 
normal unipotent subgroups of $\bG$. It is clear that $R_u(\bG)$ is an 
abstract normal subgroup. By the criterion in \ref{subsec13}, $R_u(\bG)$ is 
a closed connected subgroup of $\bG$; furthermore, $R_u(\bG)= \bU_1 \ldots 
\bU_n$ for some $n\geq 1$ and $\bU_1, \ldots, \bU_n\in 
\cS_{\text{unip}}(\bG)$. As already remarked before, this product will 
consist of unipotent elements. Thus, $R_u(\bG)$ is the unique maximal 
closed connected normal unipotent subgroup of $\bG$. (The analogous 
definition also works when $k$ is an arbitrary algebraically closed field, 
using the slightly more complicated characterisation of unipotents 
elements in that case.)
\begin{center}
{\em We say that $\bG$ is \nm{reductive} if $R_u(\bG)=\{1\}$.}
\end{center}
(Thus, connected reductive groups can be regarded as analogues of finite 
groups $G$ with $O_p(G)=\{1\}$.) These are the groups that we will be 
primarily concerned with. In an arbitrary algebraic group $\bG$, we always 
have the closed connected normal subgroups $R_u(\bG)\subseteq \bG^\circ
\subseteq \bG$, and $\bG/R_u(\bG)$ will be reductive. Note also that, 
clearly, we have the implication 
\[\mbox{$\bG$ simple} \quad \Rightarrow \quad \mbox{$\bG$ reductive (and
connected)}.\]
Here, we say that $\bG$ is a \nm{simple algebraic group} if it is 
connected, non-abelian and if it has no closed connected normal
subgroups other than $\{1\}$ and $\bG$ itself. (So, for example, $\SL_n(k)$
is a simple algebraic group, although in general it is not simple as an 
abstract group; $\GL_n(k)$ is reductive, but not simple.) 

Note that, even if one is mainly interested in studying a simple group 
$\bG$, one will also have to look at subgroups with a geometric origin, 
like Levi subgroups or centralisers of semisimple elements. These subgroups 
tend to be reductive, not just simple. For example, if $\bG$ is connected, 
reductive and $s\in \bG$ is a semisimple element, then the centralizer 
$C_{\bG}(s)$ will be a closed reductive (not necessarily connected or 
simple) subgroup; see \cite[3.5.4]{Ca2}.
\end{abs}

\begin{abs} {\bf Characters and co-characters of tori.}\label{subsec17} The
simplest examples of connected reductive algebraic groups are tori, and it
will be essential to understand some basic constructions with them. First,
a general definition. A homomorphism of algebraic groups $\lambda\colon 
\bG\rightarrow \bkm$ will be called a {\em character} of $\bG$. The set 
$X=X(\bG)$ of all characters of $\bG$ is an abelian group (which we write 
additively), called the \nm{character group} of $\bG$. Similarly, a 
homomorphism of algebraic groups $\nu\colon \bkm\rightarrow \bG$ will be 
called a {\em co-character} of $\bG$. If $\bG$ is abelian, then the set 
$Y=Y(\bG)$ of all co-characters of $\bG$ also is an abelian group (written 
additively), called the \nm{co-character group} of $\bG$. 
Now let $\bG=\bT$ be a torus over $k$; recall that this means that $\bT$ is 
isomorphic to a direct product of a finite number of copies of $\bkm$.
It is an easy exercise to show that every homomorphism of algebraic groups 
of the multiplicative group $\bkm$ into itself is given by $\xi \mapsto 
\xi^n$ for a well-defined $n\in \Z$. Thus, we have $X(\bkm)=
Y(\bkm) \cong \Z$ and this yields
\[X(\bT)\cong Y(\bT) \cong \Z^r\qquad\mbox{where} \qquad
\bT \cong \bkm \times \ldots \times \bkm \quad 
\mbox{($r$ factors)}.\]
Hence, $X(\bT)$ and $Y(\bT)$ are free abelian groups of the same
finite rank. Furthermore, we obtain a natural bilinear pairing
\[\langle \;,\;\rangle \colon X(\bT) \times Y(\bT) \rightarrow \Z,\]
defined by the condition that $\lambda(\nu(\xi))=\xi^{\langle \lambda,
\nu\rangle}$ for all $\lambda\in X(\bT)$, $\nu \in Y(\bT)$
and $\xi\in \bkm$. This pairing is a \nm{perfect pairing}, that is, 
it induces group isomorphisms 
\begin{alignat*}{2}
X(\bT)&\;\stackrel{\sim}{\longrightarrow}\;\Hom(Y(\bT),\Z), \qquad 
&\lambda \mapsto\bigl(\nu&\mapsto \langle \lambda,\nu\rangle\bigr),\\ 
Y(\bT)&\;\stackrel{\sim}{\longrightarrow}\;\Hom(X(\bT),\Z), \qquad &
\nu\mapsto\bigl(\lambda&\mapsto \langle \lambda,\nu\rangle\bigr),
\end{alignat*}
(see \cite[3.6]{MaTe}). The pair $(X(\bT),Y(\bT))$, together with the 
above pairing, is the simplest example of a so-called ``root datum'', which 
will be considered in more detail in Section~\ref{sec:rootdata}. The 
assignment $\bT \leadsto X(\bT)$ has the following fundamental property: if 
$\bT'$ is another torus over $k$, then we have a natural bijection 
\[\{\mbox{homomorphisms of algebraic groups $\bT\rightarrow 
\bT'$}\}\:\; \stackrel{1{-}1}{\longleftrightarrow} \;\;\Hom(X(\bT'),
X(\bT))\]
where, on the right hand side, $\Hom$ just stands for homomorphisms
of abstract abelian groups. The correspondence is defined by sending
a homomorphism of algebraic groups $f\colon \bT\rightarrow \bT'$ to the 
map $\varphi\colon X(\bT') \rightarrow X(\bT)$, $\chi'\mapsto \chi'\circ f$. 
For future reference, we state the following basic properties of this
correspondence:
\begin{itemize}
\item[(a)] $f\colon \bT\rightarrow \bT'$ is a closed embedding (that is, 
an isomorphism onto a closed subgroup of $\bT'$) if and only if 
$\varphi\colon X(\bT')\rightarrow X(\bT)$ is surjective; in this case, we 
have a canonical isomorphism $\ker(\varphi)\cong X(\bT'/f(\bT))$. 
\item[(b)] $f\colon \bT\rightarrow \bT'$ is surjective if and only
if $\varphi\colon X(\bT')\rightarrow X(\bT)$ is injective; in this case, we
have a canonical isomorphism $X(\bT)/\varphi(X(\bT'))\cong X(\ker(f))$ 
(induced by restriction of characters from $\bT$ to $\ker(f)$).
\end{itemize}
See \cite[Chap.~III, \S 8]{Bor} and \cite[\S 2.6]{St74} for proofs and
further details. Furthermore, by \cite[\S 3.1]{Ca2}, $\bT$ can be 
recovered from $X(\bT)$ through the isomorphism
\begin{itemize}
\item[(c)] $\;\;\bT \;\stackrel{\sim}{\longrightarrow} \Hom(X(\bT),
k^\times), \qquad t\mapsto (\lambda \mapsto \lambda(t))$. 
\end{itemize}
(Here again, $\Hom$ just stands for abstract homomorphisms of abelian 
groups.) 
\end{abs}


\begin{abs} {\bf Weight spaces.} \label{subsec1weights} Characters of tori 
play a major role in the following context. Let $\bG$ be a linear algebraic
group and $V$ be a finite-dimensional vector space over $k$. Note that $V$ 
is an affine variety with algebra of regular functions given by the 
subalgebra generated by the dual space $V^*=\Hom(V,k)\subseteq\cA(V,k)$. 
Assume that we have a \nm{representation} of $\bG$ on $V$, that is, we are 
given a morphism of affine varieties $\bG\times V \rightarrow V$ which 
defines a linear action of $\bG$ on $V$. Let $\bT \subseteq \bG$ be a maximal 
torus. (Any torus of maximum dimension is maximal.) For each character 
$\lambda\in X(\bT)$ we define the subspace 
\[ V_\lambda:=\{ v\in V\mid t.v=\lambda(t)v \mbox{ for all $t\in \bT$}\}.\]
Let $\Psi(\bT,V)$ be the set of all $\lambda\in X(\bT)$ such that $V_\lambda
\neq \{0\}$. Since $\bT$ consists of pairwise commuting semisimple elements, 
we have
\[ V=\bigoplus_{\lambda \in \Psi(\bT,V)} V_\lambda\]
(see \cite[3.1.5]{mybook}); in particular, this shows that $\Psi(\bT,V)$
is finite. The characters in $\Psi(\bT,V)$ are called \nm{weights} and the 
corresponding subspaces $V_\lambda$ called \nm{weight spaces} (relative
to $\bT$). Now, we always have the \nm{adjoint representation} of $\bG$ on its 
Lie algebra $L(\bG)$, defined as follows. For $g \in \bG$, consider the inner
automorphism $\gamma_g$ of $\bG$ defined by $\gamma_g(x)=gxg^{-1}$. Taking 
the differential, we obtain a linear map $d_1\gamma_g \colon L(\bG) 
\rightarrow L(\bG)$, which is a vector space isomorphism. Hence, we obtain 
a linear action of $\bG$ on $L(\bG)$ such that $g.v=d_1(\gamma_g)(v)$ for
all $g\in \bG$ and $v\in L(\bG)$. (The corresponding map $\bG \times L(\bG) 
\rightarrow L(\bG)$ indeed is a representation; see, for example, 
\cite[10.3]{Hum}.) Then the finite set 
\[R:=\Psi(\bT, L(\bG))\setminus \{0\}\quad \subseteq \;X(\bT)\]
is called the set of \nm{roots} of $\bG$ {\em relative to} $\bT$; we have
the \nm{root space decomposition}
\[L(\bG)=L(\bG)_0\oplus \bigoplus_{\alpha\in R} L(\bG)_\alpha.\]
This works in complete generality, for any algebraic group $\bG$. If $\bG$ 
is connected and reductive, then it is possible to obtain much more precise 
information about the root space decomposition. It turns out that then
\[L(\bG)_0=L(\bT), \quad R=-R \quad \mbox{and} \quad 
\dim L(\bG)_\alpha=1 \quad \mbox{for all $\alpha\in R$}.\] 
So, in this case, the picture is analogous to that in the theory of complex 
semisimple Lie algebras and, quite surprisingly, it shows that some crucial 
aspects of the theory do not depend on the underlying field! This fundamental
result, first proved in the {\em S\'eminaire Chevalley} \cite{Ch05}, will be 
discussed in more detail in Section~\ref{sec:chevalley}.
\end{abs}

\begin{abs} {\bf General structure of connected reductive algebraic groups.} 
\label{subsec15} Let $\bG$ be a connected linear algebraic group. Denote by 
$\bZ= \bZ(\bG)$ the center of $\bG$. Then we have $\bZ^\circ=R_u(\bZ) 
\times \bS$ where $\bS$ is a torus; see \cite[3.5.3]{mybook}. Since 
$R_u(\bZ)$ is a characteristic subgroup of $\bZ$ and $\bZ$ is a 
characteristic subgroup of $\bG$, we see that $R_u(\bZ)$ is normal in $\bG$. 
Hence, if $\bG$ is reductive, then $\bZ^\circ$ is a torus. In this case, the 
above-mentioned results about the root space decomposition lead to the 
following product decomposition of $\bG$ (see \cite[\S 8.4]{MaTe}, 
\cite[\S 8.1]{Spr}):
\[ \bG=\bZ^\circ.\bG_1 \ldots \bG_n\quad \mbox{where $\bG_1,\ldots,\bG_n$ 
are closed normal simple subgroups}\]
and $\bG_i$, $\bG_j$ pairwise commute with each other for $i \neq j$;
furthermore, this decomposition of $\bG$ has the following properties.
\begin{itemize}
\item The subgroups $\{\bG_1,\ldots,\bG_n\}$ are uniquely determined in the
sense that every closed normal simple subgroup of $\bG$ is equal to some 
$\bG_i$. 
\item We have $\bG_1\ldots \bG_n\,=\,\bGder\,:=\,$ commutator (or derived)
subgroup of $\bG$.
\end{itemize}
(Recall from \ref{subsec14a} that simple algebraic groups are assumed
to be connected and non-trivial; note also that the commutator subgroup of 
a connected algebraic group always is a closed connected normal subgroup; 
see \cite[2.4.7]{mybook}.) A connected
reductive algebraic group $\bG$ will be called \nm{semisimple} if 
$\bZ^\circ=\{1\}$ (or, equivalently, if the center of $\bG$ is finite). 
Thus, in the above setting, $\bGder$ is semisimple.

The above product decomposition can be used to prove general statements
about connected reductive algebraic groups by a reduction to simple
algebraic groups; see, for example, Lemma~\ref{Mreductord},
Theorem~\ref{multfree}.
\end{abs}

\begin{abs} {\bf Algebraic $BN$-pairs (or Tits systems).} \label{subsec16}
The concept of $BN$-pairs has been introduced by Tits \cite{Tits}, and it 
has turned out to be extremely useful. It applies to connected algebraic
groups and to finite groups, and it allows to give uniform proofs of 
many results, instead of going through a large number of case--by--case 
proofs. Recall that two subgroups $B,N$ in an arbitrary (abstract) group $G$ 
form a \nms{$BN$-pair}{BN-pair} (or a \nm{Tits system}) if the following 
conditions are satisfied.
\begin{itemize}
\item[(BN1)] $G$ is generated by $B$ and $N$.
\item[(BN2)] $H:=B \cap N$ is normal in $N$ and the quotient $W:=N/H$ is a
finite group generated by a set $S$ of elements of order~$2$.
\item[(BN3)] $n_sB n_s\neq B$ if $s\in S$ and $n_s$ is a
representative of $s$ in $N$.
\item[(BN4)] $n_sBn\subseteq Bn_snB\cup BnB$ for any $s\in S$
and $n\in N$.
\end{itemize}
The group $W$ is called the corresponding \nm{Weyl group}. We have a 
\nm{length function} on $W$, as follows. We set $l(1)=0$. If $w\neq 1$, we 
define $l(w)$ to be the length of a shortest possible expression of $w$ as a 
product of generators in~$S$. (Note that we don't have to take into 
account inverses, since $s^2=1$ for all $s\in S$.) Thus, any $w\in W$ can
be written in the form $w=s_1\cdots s_p$ where $p=l(w)$ and $s_i\in S$ 
for all~$i$. Such an expression (which is by no means unique) will be 
called a \nm{reduced expression} for~$w$. 

Furthermore, for any $w\in W$, we set $C(w):=Bn_wB$ where $n_w\in N$ is a 
representative of $w$ in $N$. Since any two representatives of $w$ lie in 
the same coset of $H\subseteq B$, we see that $C(w)$ does not depend on the 
choice of the representative. The double cosets $C(w)$ are called \nm{Bruhat 
cells} of $G$. Then the above axioms imply the fundamental \nm{Bruhat 
decomposition} (see \cite[Chap.~IV, n$^\circ$ 2.3]{bour}):
\[ G=\coprod_{w\in W} Bn_wB.\]
As Lusztig \cite{Lu10} writes, by allowing one to reduce many questions 
about $G$ to questions about the Weyl group $W$, the Bruhat decomposition 
is indispensible for the understanding of both the structure and 
representations of $G$. A key role in this context will be played by the 
Iwahori-Hecke algebra (introduced in \cite{Iw64}); this is 
a deformation of the group algebra of $W$ whose definition is 
based on the Bruhat decomposition. (We will come back to this in
a later section on Hecke algebras.)

Now let $\bG$ be a linear algebraic group over $k$ and let $\bB,\bN$ be closed 
subgroups of $\bG$ which form a $BN$-pair. Following \cite[\S 2.5]{Ca2}, we 
shall say that this is an \nm{algebraic $BN$-pair} if $\bH=\bB\cap \bN$ is 
abelian and consists entirely of semisimple elements, and we have an 
abstract semidirect product decomposition $\bB=\bU.\bH$ where $\bU$ is a 
closed normal unipotent subgroup of $\bB$ such that $\bU \cap \bH=\{1\}$. 
(If $\bB$ is connected, then this is automatically a semidirect product 
of algebraic groups as in \ref{subsec1wrong}; see \cite[6.3.5]{Spr}.) We 
do not assume that $\bG$ is connected, so the definition can apply in 
particular to finite algebraic groups. We now have:
\end{abs}

\begin{prop} \label{algbnpair} Let $\bG$ be a linear algebraic group and 
$\bB, \bN$ be subgroups which form an algebraic $BN$-pair in $\bG$, where 
$\bB= \bU.\bH$ as above. Assume that $\bH$, $\bU$ are connected and that 
$C_{\bG}(\bH)=\bH$. Then the following hold. 
\begin{itemize}
\item[(a)] $\bG$ is connected and reductive. 
\item[(b)] $\bB$ is a Borel subgroup (that is, a maximal closed connected
solvable subgroup of $\bG$); we have $\bB=N_{\bG}(\bU)$ and $[\bB,\bB]=
\bU=R_u(\bB)$.
\item[(c)] $\bH$ is a maximal torus of $\bG$ and we have $\bN=N_{\bG}(\bH)$.
\end{itemize}
\end{prop}

(See \cite[\S 2.5]{Ca2} and \cite[3.4.6, 3.4.7]{mybook}.) As in 
\cite[3.4.5]{mybook}, a $BN$-pair satisfying the conditions in 
Proposition~\ref{algbnpair} will be called a \nms{reductive 
$BN$-pair}{reductive BN-pair}. 

Much more difficult is the converse of the above result, which comes about 
as the culmination of a long series of arguments. Namely, if $\bG$ is a 
connected reductive algebraic group, then $\bG$ has a reductive $BN$-pair 
in which $\bB$ is a Borel subgroup and $\bN$ is the normaliser of a maximal 
torus contained in $\bB$. (We will discuss this in more detail in 
Section~\ref{sec:chevalley}.) For our purposes here, the realisation of 
connected reductive algebraic groups in terms of algebraic $BN$-pairs as 
above is sufficient for many purposes.  For example, if $\bG$ is a 
``classical group'' as in \ref{subsecclassic}, then 
algebraic $BN$-pairs as above are explicitly described in 
\cite[\S 1.7]{mybook}. In these cases, one can always find an algebraic 
$BN$-pair in which $\bB$ consists of upper triangular matrices and $\bH$ 
consists of diagonal matrices. See also the relevant chapters in \cite{Go}, 
\cite{Go2}, \cite{Go3}.

\section{Root data} \label{sec:rootdata}

We now introduce abstract root data  and prove some basic properties of them.
As we shall see in later sections, these form the combinatorial skeleton of
 connected reductive algebraic groups, that is, they capture those features 
which do not depend on the underlying field $k$. (A reader who wishes to see 
a much more systematic discussion of root data is referred to 
\cite[Expos\'e~XXI]{sga33}.)

\begin{abs} \label{Mabs21} Let $X,Y$ be free abelian groups of
the same finite rank; assume that there is a bilinear pairing 
$\langle \;,\;\rangle \colon X \times Y \rightarrow \Z$ which is
perfect, that is, it induces group isomorphisms $Y\cong
\Hom(X,\Z)$ and $X\cong\Hom(Y,\Z)$ (as in \ref{subsec17}). Furthermore,
let $R\subseteq X$ and $R^\vee \subseteq Y$ be finite subsets. Then 
the quadruple $\cR= (X,R, Y, R^\vee)$ is called a \nm{root datum} if
the following conditions are satisfied. 
\begin{itemize}
\item[(R1)] There is a bijection $R \rightarrow R^\vee$, $\alpha\mapsto 
\alpha^\vee$, such that $\langle \alpha,\alpha^\vee\rangle=2$ for all 
$\alpha\in R$.
\item[(R2)] For every $\alpha \in R$, we have $2\alpha \not\in R$.
\item[(R3)] For $\alpha\in R$, we define endomorphisms $w_\alpha 
\colon X \rightarrow X$ and $w_\alpha^\vee \colon Y \rightarrow Y$ by 
\[w_\alpha(\lambda)=\lambda-\langle \lambda,\alpha^\vee\rangle \alpha 
\qquad \mbox{and}\qquad w_\alpha^\vee(\nu)=\nu-\langle \alpha,\nu
\rangle \alpha^\vee\]
for all $\lambda \in X$ and $\nu \in Y$. Then we require that
$w_\alpha(R)=R$ and $w_\alpha^\vee(R^\vee)=R^\vee$ for all 
$\alpha\in R$.
\end{itemize}
We shall see in \ref{Mcartan0a} that the concept of root data is, in a very
precise sense, an enhancement of the more traditional concept of root 
systems (related to finite reflection groups; see \cite{bour}). First, we
need some preparations.

The defining formula immediately shows that $w_\alpha^2=\id_X$ and 
$(w_\alpha^\vee)^2=\id_{Y}$. Hence, we have $w_\alpha\in \Aut(X)$ 
and $w_\alpha^\vee \in \Aut(Y)$ for all $\alpha\in R$. We set
\[\bW:=\langle w_\alpha \mid \alpha\in R\rangle \subseteq \Aut(X) \qquad 
\mbox{and} \qquad \bW^\vee:=\langle w_\alpha^\vee \mid \alpha\in R
\rangle \subseteq \Aut(Y);\]
these groups are called the \nmi{Weyl groups}{Weyl group} of $R$ and 
$R^\vee$, respectively\footnote{For the time being, we keep a separate 
notation for these two Weyl groups; in Remark~\ref{MidentW}, we will 
identify them using the isomorphism in Lemma~\ref{Mweylfinite1}(a).}. 
By (R3), we have an action of $\bW$ on $R$ and an action of $\bW^\vee$ 
on $R^\vee$. 
\end{abs}

%

\begin{abs} \label{Mhomrootdata} Let $\cR=(X,R,Y, R^\vee)$ and $\cR'=
(X',R',Y', R'^\vee)$ be root data. Let $\varphi\colon X' \rightarrow 
X$ be a group homomorphism. The corresponding {\em transpose map}
$\varphi^\trp \colon Y \rightarrow Y'$ is uniquely defined by the condition 
that 
\[ \langle \varphi(\lambda'), \nu\rangle=\langle \lambda', \varphi^\trp(\nu)
\rangle'\qquad\mbox{for all $\lambda' \in X'$ and $\nu\in Y$},\]
where $\langle\;, \;\rangle$ is the bilinear pairing for $\cR$ and 
$\langle \;,\;\rangle'$ is the bilinear pairing for $\cR'$. We say that 
$\varphi$ is a \nm{homomorphism of root data} if $\varphi$ maps $R'$
bijectively onto $R$ and $\varphi^\trp$ maps $R^\vee$ bijectively onto
$R'^\vee$. It then follows automatically that $\varphi^\trp(\varphi
(\beta)^\vee)=\beta^\vee$ for all $\beta\in R'$; see \cite[XXI, 
6.1.2]{sga33}. If $\varphi$ is a bijective homomorphism of root data, 
we say that $\cR$ and $\cR'$ are {\em isomorphic}.
\end{abs}

\begin{lem} \label{Mweylfinite1} Let $\cR=(X,R,Y, R^\vee)$ be a
root datum. 
\begin{itemize}
\item[(a)] There is a unique group isomorphism $\delta\colon \bW 
\stackrel{\sim}{\rightarrow} \bW^\vee$ such that $\delta(w_\alpha)=
w_\alpha^\vee$ for all $\alpha\in R$; we have 
\[ \langle w^{-1}(\lambda),\nu\rangle=\langle \lambda,\delta(w)(\nu)
\rangle \quad \mbox{for all $w\in \bW$, $\lambda\in X$, $\nu \in Y$}.\]
\item[(b)] The quadruple $(Y, R^\vee,X,R)$ also is a root datum,
with pairing $\langle\;, \;\rangle^*\colon Y \times X \rightarrow \Z$ 
defined by $\langle \nu,\lambda\rangle^*:= \langle \lambda, \nu
\rangle$ for all $\nu \in Y$ and 
$\lambda\in X$.
\item[(c)] For any $\lambda\in X$ and $w\in \bW$, we have $\lambda-
w(\lambda)\in \Z R$.
\end{itemize}
\end{lem}
The root datum in (b) is called the \nm{dual root datum} of $\cR$.

\begin{proof} (a) For any group homomorphism $\varphi\colon X\rightarrow X$, 
consider its transpose $\varphi^\trp\colon Y\rightarrow Y$, as defined above.
Clearly, we have $\id_X^\trp=\id_{Y}$ and $(\varphi \circ \psi)^\trp=
\psi^\trp \circ \varphi^\trp$ if $\psi\colon X \rightarrow X$ is a further 
group homomorphism. Thus, $\bW^\trp:=\{w^\trp \mid w\in \bW\}$ is a subgroup 
of $\Aut(Y)$ and the map $\delta\colon \bW \rightarrow \bW^\trp$, 
$w\mapsto (w^{-1})^\trp$, is an isomorphism. Now, using the defining 
formulae in (R3), one immediately checks that
\[ \langle w_\alpha(\lambda),\nu\rangle=\langle \lambda,w_\alpha^\vee
(\nu) \rangle \quad \mbox{for all $\alpha\in R$, $\lambda \in X$, 
$\nu  \in Y$}.\]
Hence, we have $w_\alpha^\trp=w_\alpha^\vee$ for all $\alpha\in R$ and
so $\bW^\trp=\bW^\vee$. This yields (a).

(b) This is a straightforward verification.

(c) The defininig formula shows that this is true if $w=w_\alpha$ for
$\alpha\in R$. But then it also follows in general, since $\bW$ is
generated by the $w_\alpha$ ($\alpha\in R$).
\end{proof}

\begin{lem} \label{Mweylfinite2} Let $\cR=(X,R,Y,R^\vee)$ be a
root datum. We set $X_0:=\{\lambda \in X\mid \langle \lambda,
\alpha^\vee\rangle =0 \mbox{ for all $\alpha\in R$}\}$. Then 
\[X_0 \cap \Z R=\{0\}\qquad \mbox{and}\qquad |X/(X_0+\Z R)|<\infty.\]
Consequently, $\bW$ is a finite group and the action of $\bW$ on $R$ is 
faithful (that is, if $w\in \bW$ is such that $w(\alpha)=\alpha$ for all
$\alpha\in R$, then $w=1$).
\end{lem}

\begin{proof} Let us extend scalars from $\Z$ to $\Q$. We denote $X_\Q=
\Q\otimes_\Z X$ and $Y_\Q=\Q\otimes_\Z Y$. Then $\langle\;,\;
\rangle$ extends to a non-degenerate $\Q$-bilinear form on $X_\Q \times 
Y_\Q$ which we denote by the same symbol. Since $X$, $Y$ are free
$\Z$-modules, we can naturally regard $X$ as a subset of $X_\Q$ and $Y$ 
as a subset of $Y_\Q$. Similarly, we can regard $\bW$ as a subgroup of
$\GL(X_\Q)$ and $\bW^\vee$ as a subgroup of $\GL(Y_\Q)$. So, in order 
to show the statements about $X_0$ and $\Z R$, it is sufficient to show that
\[X_\Q=X_{0,\Q} \oplus \Q R\quad \mbox{where} \quad X_{0,\Q}:=\{x\in X_\Q  
\mid \langle x,\alpha^\vee\rangle=0 \mbox{ for all $\alpha \in R$}\}.\]
For this purpose, following \cite[XXI, \S 1.2]{sga33}, we consider 
the linear map 
\[f\colon X_\Q\rightarrow Y_\Q, \qquad x\mapsto \sum_{\alpha\in R} 
\langle x,\alpha^\vee\rangle\,\alpha^\vee.\]
Let $\beta\in R$. Using (R3), Lemma~\ref{Mweylfinite1}(a) and the fact
that $(w_\beta^\vee)^2=\id_Y$, we obtain 
\[ \bigl(f\circ w_\beta\bigr)(x)=\sum_{\alpha\in R} \langle w_\beta(x),
\alpha^\vee \rangle\,\alpha^\vee=\sum_{\alpha\in R} \langle x,w_\beta^\vee
(\alpha^\vee)\rangle\,\alpha^\vee=(w_\beta^\vee\circ f)(x)\]
for all $x\in X_\Q$. This identity in turn implies that, for any $\beta
\in R$, we have:
\begin{align*}
f(\beta) &= -f(w_\beta(\beta))=-w_\beta^\vee(f(\beta))=
-\sum_{\alpha\in R} \langle \beta,\alpha^\vee\rangle \,w_\beta^\vee
(\alpha^\vee)\\
&=-\sum_{\alpha\in R} \langle \beta,\alpha^\vee\rangle \,
\bigl(\alpha^\vee-\langle \beta, \alpha^\vee\rangle \,\beta^\vee
\bigr)=-f(\beta)+\Bigl(\sum_{\alpha \in R} \langle \beta, \alpha^\vee
\rangle^2\Bigr)\,\beta^\vee.
\end{align*}
Noting that $\langle \beta,f(\beta)\rangle=\sum_{\alpha \in R} \langle 
\beta, \alpha^\vee \rangle^2$, we deduce that 
\[ 2f(\beta)=\langle \beta,f(\beta)\rangle\, \beta^\vee
\quad \mbox{and} \quad \langle \beta,f(\beta)\rangle>0
\qquad \mbox{for all $\beta\in R$}.\]
This shows that $f(\Q R)=\Q R^\vee$ and so $\dim \Q R\geq \dim 
\Q R^\vee$. By the symmetry expressed in Lemma~\ref{Mweylfinite1}(b),
the reverse inequality also holds and so $\dim \Q R=\dim \Q R^\vee$.
Thus, $f$ restricts to an isomorphism $f\colon \Q R
\stackrel{\sim}{\rightarrow} \Q R^\vee$. Now, we clearly have $X_{0,
\Q}\subseteq \ker(f)$, whence $X_{0,\Q} \cap \Q R=\{0\}$. Since 
$\langle\;,\; \rangle$ extends to a non-degenerate bilinear form on 
$X_\Q \times Y_\Q$, we have $\dim X_\Q=\dim X_{0,\Q} +\dim \Q R^\vee$. 
Since also $\dim \Q R^\vee=\dim \Q R$, we conclude that $X_\Q=
X_{0,\Q}\oplus\Q R$, as desired. 

Now we show that the action of $\bW$ on $R$ is faithful. Let $w\in \bW$
be such that $w(\alpha)=\alpha$ for all $\alpha\in R$. Then $w$ acts
as the identity on the subspace $\Q R\subseteq X_\Q$. On the other hand,
the defining equation shows that all $w_\alpha$, $\alpha\in R$, act as
the identity on $X_{0,\Q}$, so $\bW$ is trivial on $X_{0,\Q}$. Hence, 
$w=1$ since $X_\Q=X_{0,\Q}+\Q R$. Since $R$ is finite, it follows 
that $\bW$ must be finite, too.
\end{proof}

\begin{abs} \label{Mcartan0a} Let $\cR=(X,R,Y, R^\vee)$ be a
root datum. As in the above proof, we extend scalars from $\Z$ to $\Q$
and set $X_\Q=\Q\otimes_\Z X$. Following \cite[Chap.~VI, \S 1, 
Prop.~3]{bour}, we define a symmetric bilinear form $(\;,\;)\colon X_\Q 
\times X_\Q \rightarrow \Q$ by
\[ (x,y):=\sum_{\alpha \in R} \langle x, \alpha^\vee\rangle \langle y,
\alpha^\vee\rangle  \qquad \mbox{for all $x,y\in X_\Q$}.\]
Using (R3) and Lemma~\ref{Mweylfinite1}(a), we see that 
$(\;,\;)$ is $\bW$-invariant, that is, we have $(w(x),w(y))=(x,y)$ for 
all $w\in \bW$ and all $x,y\in X_\Q$. Clearly, we have $(x,x)\geq 0$ for 
all $x\in X_\Q$; furthermore, $(\beta,\beta)>0$ for all $\beta \in R$ 
(since $\langle \beta, \beta^\vee\rangle=2>0$). By a standard argument
(see \cite[Chap.~VI, \S 1, Lemme~2]{bour}), this yields that 
\[2\frac{(\alpha,\beta)}{(\beta,\beta)}=\langle \alpha,\beta^\vee
\rangle\in \Z \qquad \mbox{for all $\alpha,\beta\in R$}.\]
We claim that the restriction of $(\;,\;)$ to $\Q R\times \Q R$ is 
positive-definite. Indeed, assume that $(x,x)=0$ where $x\in \Q R$.
Then $0=(x,x)=\sum_{\alpha\in R} \langle x, \alpha^\vee\rangle^2$ and 
so $\langle x,\alpha^\vee\rangle=0$ for all $\alpha\in R$. Hence, 
Lemma~\ref{Mweylfinite2} shows that $x=0$, as desired.

Thus, we see that $R$ is a crystallographic root system in the subspace 
$\Q R$ of $X_\Q$; see \cite[Chap.~VI, \S 1, D\'ef.~1]{bour}. The Weyl 
group of $R$ is $\bW$;  see Lemma~\ref{Mweylfinite2}. Furthermore, $R$ 
is {\em reduced}, in the sense that
\[ R\cap \Q \alpha=\{\pm \alpha\} \qquad \mbox{for all $\alpha\in R$}.\]
(This is an easy consequence of (R2); see \cite[14.7]{Bor}.) 
Similarly, $R^\vee$ is a reduced crystallographic root system in 
$\Q R^\vee$, by the symmetry in Lemma~\ref{Mweylfinite1}(b).
\end{abs}

\begin{abs} \label{Mcartan0} Keeping the above notation, we now recall some 
standard results on root systems (see, e.g., \cite[App.~A]{MaTe}). 
There is a subset $\Pi\subseteq R$ such that:
\begin{itemize}
\item[(a)] $\Pi$ is linearly independent in $\Q R$ and
\item[(b)] every $\alpha\in R$ can be written as $\alpha=\sum_{\beta\in\Pi}
x_\beta \,\beta$ where $x_\beta\in \Q$ and either $x_\beta\geq 0$ for all 
$\beta \in \Pi$ or $x_\beta\leq 0$ for all $\beta\in \Pi$. 
\end{itemize}
We call $\Pi$ a \nm{base} for $R$. The corresponding set of \nm{positive
roots} $R^+\subseteq R$ consists of those $\alpha\in R$ which can be 
written as $\alpha=\sum_{\beta\in \Pi} x_\beta\, \beta$ where $x_\beta
\in \Q$ and $x_\beta\geq 0$ for all $\beta\in \Pi$. The roots in $R^-:=
-R^+$ are called the corresponding {\em negative roots}. Furthermore, if 
(a) and (b) hold, then we also have:
\begin{itemize}
\item[(c)] For every $\alpha\in R$, there exists some $w\in\bW$ such that 
$w(\alpha)\in \Pi$.
\item[(d)] Every $\alpha\in R$ is a $\Z$-linear combination of the
roots in the base $\Pi$. (That is, the coefficients $x_\beta$ in (b) are 
always integers.)
\item[(e)] $\bW$ is a Coxeter group, with generators $\{w_\beta \mid 
\beta\in\Pi\}$ and defining relations $(w_{\beta}w_{\gamma})^{m_{\beta
\gamma}}=1$ for all $\beta,\gamma\in \Pi$, where $m_{\beta\gamma}\geq 1$ 
is the order of  $w_{\beta}w_{\gamma}\in\bW$; furthermore, we have
$4\cos^2(\pi/m_{\beta\gamma})=\langle \gamma,\beta^\vee\rangle
\langle \beta,\gamma^\vee\rangle$ for all $\beta,\gamma\in \Pi$.
\end{itemize}
Finally, any two bases of $R$ can be transformed into each other by a 
unique element of $\bW$. In particular, $r:=|\Pi|$ is well-defined and called 
the \nm{rank} of $R$; furthermore, writing $\Pi=\{\beta_1,\ldots,\beta_r\}$,
the matrix
\[ C:=\bigl(\langle \beta_j,\beta_i^\vee\rangle\bigr)_{1\leq i,j\leq r}\]
is uniquely determined by $\cR$ (up to reordering the rows and columns); 
it is called the \nm{Cartan matrix} of $\cR$. We say that two root data 
$\cR$, $\cR'$ have the same \nm{Cartan type} if the corresponding Cartan 
matrices are the same (up to choosing a bijection between the associated 
bases $\Pi$, $\Pi'$). Thus, $\cR$ and $\cR'$ have the same Cartan type if 
and only if $R\subseteq X_{\Q}$ and $R'\subseteq X_{\Q}'$ are isomorphic
root systems (see \cite[Chap.~VI, n$^\circ$ 1.5]{bour}).

We associate with $C$ a \nm{Dynkin diagram}, defined as follows. It has 
vertices labelled by the elements in $\Pi=\{\beta_1,\ldots,\beta_r\}$. If 
$i\neq j$ and $|\langle \beta_j,\beta_i^\vee\rangle|\geq |\langle \beta_i,
\beta_j^\vee \rangle|$, then the corresponding vertices are joined by 
$|\langle \beta_j, \beta_i^\vee\rangle|$ lines; furthermore, these lines are 
equipped with an arrow pointing towards the vertex labeled by $\beta_i$ if 
$|\langle \beta_j, \beta_i^\vee \rangle|>1$. (Note that, in this case, we
automatically have $\langle \beta_i,\beta_j^\vee\rangle=-1$ by (e).) 

We say that $C$ is an \nm{indecomposable Cartan matrix} if the associated
Dynkin diagram is a connected graph; otherwise, we say that $C$ is
{\em decomposable}. Clearly, any Cartan matrix can be expressed as a block 
diagonal matrix with diagonal blocks given by indecomposable Cartan 
matrices. The classification of indecomposable Cartan matrices is 
well-known (see \cite[Chap.~VI, \S 4]{bour}); the corresponding Dynkin 
diagrams are listed in Table~\ref{Mdynkintbl}. (The Cartan matrices of
type $A_n$, $B_n$, $C_n$, $G_2$, $F_4$ are printed explicitly in
Examples~\ref{MsuzukiB}, \ref{rootdatGL}, \ref{isogBC} below.) See Kac 
\cite[Chap.~4]{Kac} for a somewhat different approach to this classification. 
\end{abs}

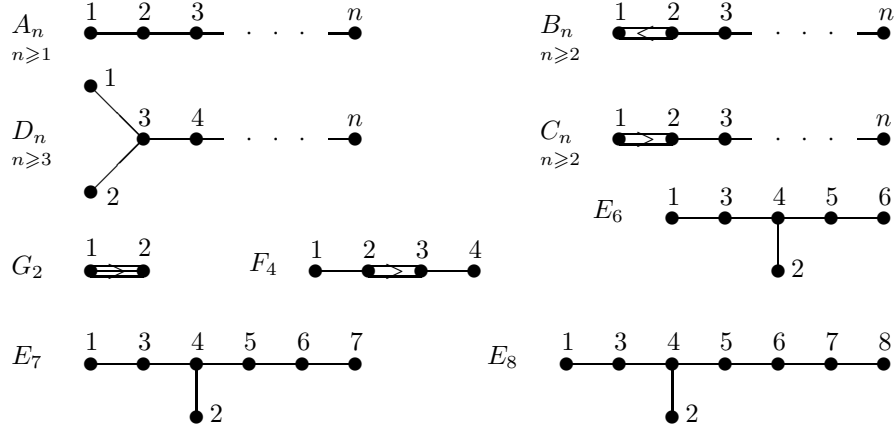
\begin{table}[htbp] 
\caption{Dynkin diagrams of indecomposable Cartan matrices} \label{Mdynkintbl} 
\begin{center}
\begin{picture}(345,160)
\put( 10, 25){$E_7$}
\put( 40, 25){\circle*{5}}
\put( 38, 30){$1$}
\put( 40, 25){\line(1,0){20}}
\put( 60, 25){\circle*{5}}
\put( 58, 30){$3$}
\put( 60, 25){\line(1,0){20}}
\put( 80, 25){\circle*{5}}
\put( 78, 30){$4$}
\put( 80, 25){\line(0,-1){20}}
\put( 80, 05){\circle*{5}}
\put( 85, 03){$2$}
\put( 80, 25){\line(1,0){20}}
\put(100, 25){\circle*{5}}
\put( 98, 30){$5$}
\put(100, 25){\line(1,0){20}}
\put(120, 25){\circle*{5}}
\put(118, 30){$6$}
\put(120, 25){\line(1,0){20}}
\put(140, 25){\circle*{5}}
\put(138, 30){$7$}

\put(190, 25){$E_8$}
\put(220, 25){\circle*{5}}
\put(218, 30){$1$}
\put(220, 25){\line(1,0){20}}
\put(240, 25){\circle*{5}}
\put(238, 30){$3$}
\put(240, 25){\line(1,0){20}}
\put(260, 25){\circle*{5}}
\put(258, 30){$4$}
\put(260, 25){\line(0,-1){20}}
\put(260, 05){\circle*{5}}
\put(265, 03){$2$}
\put(260, 25){\line(1,0){20}}
\put(280, 25){\circle*{5}}
\put(278, 30){$5$}
\put(280, 25){\line(1,0){20}}
\put(300, 25){\circle*{5}}
\put(298, 30){$6$}
\put(300, 25){\line(1,0){20}}
\put(320, 25){\circle*{5}}
\put(318, 30){$7$}
\put(320, 25){\line(1,0){20}}
\put(340, 25){\circle*{5}}
\put(338, 30){$8$}

\put( 10, 59){$G_2$}
\put( 40, 60){\circle*{5}}
\put( 38, 66){$1$}
\put( 40, 58){\line(1,0){20}}
\put( 40, 60){\line(1,0){20}}
\put( 40, 62){\line(1,0){20}}
\put( 46, 57.5){$>$}
\put( 60, 60){\circle*{5}}
\put( 58, 66){$2$}

\put(100, 60){$F_4$}
\put(125, 60){\circle*{5}}
\put(123, 65){$1$}
\put(125, 60){\line(1,0){20}}
\put(145, 60){\circle*{5}}
\put(143, 65){$2$}
\put(145, 58){\line(1,0){20}}
\put(145, 62){\line(1,0){20}}
\put(151, 57.5){$>$}
\put(165, 60){\circle*{5}}
\put(163, 65){$3$}
\put(165, 60){\line(1,0){20}}
\put(185, 60){\circle*{5}}
\put(183, 65){$4$}

\put(230, 80){$E_6$}
\put(260, 80){\circle*{5}}
\put(258, 85){$1$}
\put(260, 80){\line(1,0){20}}
\put(280, 80){\circle*{5}}
\put(278, 85){$3$}
\put(280, 80){\line(1,0){20}}
\put(300, 80){\circle*{5}}
\put(298, 85){$4$}
\put(300, 80){\line(0,-1){20}}
\put(300, 60){\circle*{5}}
\put(305, 58){$2$}
\put(300, 80){\line(1,0){20}}
\put(320, 80){\circle*{5}}
\put(318, 85){$5$}
\put(320, 80){\line(1,0){20}}
\put(340, 80){\circle*{5}}
\put(338, 85){$6$}

\put( 10,110){$D_n$}
\put( 10,100){$\scriptstyle{n \geq 3}$}
\put( 40,130){\circle*{5}}
\put( 45,130){$1$}
\put( 40,130){\line(1,-1){21}}
\put( 40, 90){\circle*{5}}
\put( 46, 85){$2$}
\put( 40, 90){\line(1,1){21}}
\put( 60,110){\circle*{5}}
\put( 58,115){$3$}
\put( 60,110){\line(1,0){30}}
\put( 80,110){\circle*{5}}
\put( 78,115){$4$}
\put(100,110){\circle*{1}}
\put(110,110){\circle*{1}}
\put(120,110){\circle*{1}}
\put(130,110){\line(1,0){10}}
\put(140,110){\circle*{5}}
\put(137,115){$n$}

\put(210,110){$C_n$}
\put(210,100){$\scriptstyle{n \geq 2}$}
\put(240,110){\circle*{5}}
\put(238,115){$1$}
\put(240,108){\line(1,0){20}}
\put(240,112){\line(1,0){20}}
\put(246,107.5){$>$}
\put(260,110){\circle*{5}}
\put(258,115){$2$}
\put(260,110){\line(1,0){30}}
\put(280,110){\circle*{5}}
\put(278,115){$3$}
\put(300,110){\circle*{1}}
\put(310,110){\circle*{1}}
\put(320,110){\circle*{1}}
\put(330,110){\line(1,0){10}}
\put(340,110){\circle*{5}}
\put(337,115){$n$}

\put( 10,150){$A_n$}
\put( 10,140){$\scriptstyle{n \geq 1}$}
\put( 40,150){\circle*{5}}
\put( 38,155){$1$}
\put( 40,150){\line(1,0){20}}
\put( 60,150){\circle*{5}}
\put( 58,155){$2$}
\put( 60,150){\line(1,0){30}}
\put( 80,150){\circle*{5}}
\put( 78,155){$3$}
\put(100,150){\circle*{1}}
\put(110,150){\circle*{1}}
\put(120,150){\circle*{1}}
\put(130,150){\line(1,0){10}}
\put(140,150){\circle*{5}}
\put(137,155){$n$}

\put(210,150){$B_n$}
\put(210,140){$\scriptstyle{n \geq 2}$}
\put(240,150){\circle*{5}}
\put(238,155){$1$}
\put(240,148){\line(1,0){20}}
\put(240,152){\line(1,0){20}}
\put(246,147.5){$<$}
\put(260,150){\circle*{5}}
\put(258,155){$2$}
\put(260,150){\line(1,0){30}}
\put(280,150){\circle*{5}}
\put(278,155){$3$}
\put(300,150){\circle*{1}}
\put(310,150){\circle*{1}}
\put(320,150){\circle*{1}}
\put(330,150){\line(1,0){10}}
\put(340,150){\circle*{5}}
\put(338,155){$n$}
\end{picture}
\end{center}
\footnotesize $\;\;$ (This labeling will be used throughout
this book; it is the same as in {\sf CHEVIE} \cite{chv}, \cite{Chv}. 
Note that $B_2=C_2$ and $D_3=A_3$, up to re-labeling the 
vertices.)$\;\;\qquad\qquad\qquad$
\end{table}

We have the following general characterisation of Cartan matrices.

\begin{prop}[Cf.\ \protect{\cite[Chap.~VI, \S 4]{bour}}] \label{Mcartan} 
Let $S$ be a finite set and $C=(c_{st})_{s,t\in S}$ be a matrix with 
integer entries. Then $C$ is the \nm{Cartan matrix} of a reduced 
crystallographic root system if and only if the following conditions hold:
\begin{itemize}
\item[(C1)] We have $c_{ss}=2$ and, for $s\neq t$ we have $c_{st}\leq 0$;
furthermore, $c_{st}\neq 0$ if and only if $c_{ts}\neq 0$.
\item[(C2)] For $s,t\in S$, let $m_{st}\in \Z_{\geq 1}$ be defined by
the condition that $c_{st}c_{ts}=4\cos^2(\pi/m_{st})$. (Thus, we have
$m_{ss}=1$ and $m_{st}\in \{2,3,4,6\}$ for $s\neq t$.) Then the matrix
$(-\cos(\pi/m_{st}))_{s,t\in S}$ is positive-definite.
\end{itemize}
\end{prop}

\begin{rem} \label{Mcartanfund} Let $C=(c_{st})_{s,t\in S}$ be a Cartan 
matrix. Let $\Omega$ be the free abelian group with basis $\{\omega_s \mid 
s\in S\}$. Let $\Z C \subseteq \Omega$ be the subgroup generated by the 
columns of $C$, that is, by all vectors of the form $\sum_{s \in S} c_{st} 
\omega_s$ for $t \in S$. Then 
\[\Lambda(C):=\Omega/\Z C\]
is called the \nm{fundamental group} of $C$. (This agrees with the 
definitions in \cite[9.14]{MaTe} or \cite[8.1.11]{Spr}, for example.) If $C$ 
is indecomposable, then the groups $\Lambda(C)$ are easily computed and
listed in Table~\ref{Mfundgrp}. 
\end{rem}

\begin{table}[htbp] 
\caption{Fundamental groups of indecomposable Cartan matrices} 
\label{Mfundgrp} \begin{center}
$\begin{array}{cl} \hline \mbox{Type of $C$} & \Lambda(C) \\ \hline
A_{n-1} & \; \Z/n\Z\\
B_n,C_n & \;\Z/2\Z\\
D_{n}  & \left\{\begin{array}{cl} \Z/2\Z \oplus \Z/2\Z & \mbox{ ($n$ 
even)}\\ \Z/4\Z & \mbox{ ($n$ odd)} \end{array}\right.\\
G_2,F_4,E_8  & \;\{0\}\\
E_6  & \;\Z/3\Z\\
E_7  & \;\Z/2\Z \\ \hline
\end{array}$
\end{center}
\end{table}

In \ref{Mhomrootdata}, we have defined what it means for two root data to
be isomorphic. We shall also need the following, somewhat more general
notion. 

\begin{defn} \label{MdefisogR} Let $\cR=(X,R,Y,R^\vee)$ and $\cR'=(X',
R',Y',R'^\vee)$ be root data. We fix an integer $p$ such that either $p=1$ 
or $p$ is a prime number. Then a group homomorphism $\varphi \colon X' 
\rightarrow  X$ is called a \nms{$p$-isogeny of root data}{p-isogeny of 
root data} if there exist a bijection $R\rightarrow R'$, $\alpha \mapsto 
\alpha^\dagger$, and positive integers $q_\alpha>0$, each an integral 
power of~$p$, such that $\varphi$ and its transpose $\varphi^\trp \colon 
Y \rightarrow Y'$ satisfy the following conditions.
\begin{itemize}
\item[(I1)] $\varphi$ and $\varphi^\trp$ are injective.
\item[(I2)] We have $\varphi(\alpha^\dagger)=q_\alpha\,\alpha$ and
$\varphi^\trp(\alpha^\vee)=q_\alpha\, (\alpha^\dagger)^\vee$ for all 
$\alpha \in R$.
\end{itemize}
The conditions (I1) and (I2) appear in \cite[\S 18.2]{Ch05}; following 
Chevalley, we call the numbers $\{q_\alpha\}$ the \nm{root exponents} of 
$\varphi$. Note that $\alpha\mapsto \alpha^\dagger$ and the numbers 
$\{q_\alpha\}$ are uniquely determined by $\varphi$ (since $R$ is reduced).

Let $\bW\subseteq \Aut(X)$ be the Weyl group of $R$. Then one easily sees 
that, for any $\alpha \in R$ and $w\in \bW$, we have $q_{w(\alpha)}=
q_\alpha$; see \cite[9.6.4]{Spr}. Hence, by \ref{Mcartan0}(c), the map 
$\alpha \mapsto q_\alpha$ is determined by its values on a base of $R$. 
We also see that $\varphi$ is an isomorphism of root data if and only if 
$\varphi$ is a bijective isogeny where $q_\alpha=1$ for all $\alpha\in R$. 
Finally note that if $p=1$, then $q_\alpha=1$ for all $\alpha\in R$.
\end{defn}

A simple example of a $p$-isogeny of a root datum into itself is given by
$\varphi\colon X\rightarrow X$, $\lambda \mapsto p\lambda$ (scalar 
multiplication with~$p$); this will be continued in Example~\ref{canfrob2a}.

\begin{rem} \label{MdefisogR1} Keep the notation in the above
definition. Let $\bW\subseteq \Aut(X)$ be the Weyl group of $\cR$ 
and $\bW'\subseteq \Aut(X')$ be the Weyl group of $\cR'$. Then one easily
sees that a $p$-isogeny $\varphi \colon X'\rightarrow X$ induces a unique 
group isomorphism 
\[\sigma \colon \bW \rightarrow \bW' \quad \mbox{such that} \quad 
\varphi\circ \sigma(w)=w\circ \varphi \quad (w\in \bW).\]
We have $\sigma(w_\alpha)=w_{\alpha^\dagger}$ for all $\alpha\in R$
where $w_\alpha\in \bW$ is the reflection corresponding to $\alpha\in R$
and $w_{\alpha^\dagger}\in \bW'$ is the reflection corresponding to
$\alpha^\dagger\in R'$. (See \cite[\S 18.3]{Ch05} for further details; see
also \ref{Mdefmatrixisog} below.) In particular, if $\varphi\colon X
\rightarrow X$ is a $p$-isogeny of $\cR$ into itself, then $\varphi\circ
\bW=\bW\circ \varphi$ and so $\varphi$ normalizes $\bW\subseteq \Aut(X)$.
\end{rem}

\begin{rem} \label{Misosimple} Let $\cR=(X,R,Y,R^\vee)$ and $\cR'=(X',R',
Y',R'^\vee)$ be root data. Let us fix a base $\Pi$ of $R$ and a base
$\Pi'$ of $R'$; see \ref{Mcartan0}. Let $\varphi\colon X' \rightarrow X$
be a group homomorphism which defines a $p$-isogeny of root data. Then (I2) 
shows that $\Pi^\dagger:=\{\alpha^\dagger \mid \alpha\in \Pi\}$ also is a 
base of $R'$, where $\alpha\mapsto \alpha^\dagger$ denotes the bijection
$R\rightarrow R'$ associated with $\varphi$. As already mentioned in
\ref{Mcartan0}, there exists a unique $w$ in the Weyl group $\bW'$ of 
$\cR'$ such that $\Pi^\dagger=w(\Pi')$. Now $w\in\bW' \subseteq \Aut(X')$ 
certainly is an isomorphism of $\cR'$ into itself. Hence, the composition 
$\varphi':=\varphi\circ w\colon X'\rightarrow X$ will also be a 
$p$-isogeny of root data and the bijection $R\rightarrow R'$ associated 
with $\varphi'$ will map $\Pi$ onto $\Pi'$. This shows that, replacing 
$\varphi$ by $\varphi \circ w$ for a suitable $w\in \bW'$ if necessary, we 
can always assume that the bijection $R\rightarrow R'$ will preserve the 
given bases $\Pi\subseteq R$ and $\Pi'\subseteq R'$.
\end{rem}

\begin{rem} \label{MidentW} Let $\cR=(X,R,Y,R^\vee)$ be a root datum
and $\bW\subseteq \Aut(X)$ be the corresponding Weyl group. Let $\Pi$ be a 
base of $R$. By \ref{Mcartan0}(e), $\bW$ is a Coxeter group with 
generating set $S=\{w_{\alpha}\mid \alpha\in\Pi\}$. We denote by $l\colon
\bW \rightarrow \Z_{\geq 0}$ the corresponding length function. To 
unify the notation, we shall use $S$ as an indexing set for $\Pi$, that 
is, $\Pi=\{\alpha_s\mid s\in S\}$ where $\alpha_s$ is the root of the
reflection~$s$. 

Using the isomorphism $\delta$ in Lemma~\ref{Mweylfinite1}(a), we can 
identify $\bW^\vee=\bW=\langle S\rangle$. Under this identification, 
$\bW$ will act on both $X$ and $Y$, and we have
\[ \langle w^{-1}.\lambda,\nu\rangle=\langle \lambda, w.\nu\rangle
\qquad \mbox{for all $w\in \bW$, $\lambda \in X$, $\nu \in Y$}.\]
Recall that the action of $s\in S$ on $X$ is given by $s.\lambda=\lambda
-\langle\lambda,\alpha_s^\vee\rangle \alpha_s$ for all $\lambda \in X$;
the action of $s\in S$ on $Y$ is given by $s.\nu=\nu -\langle\alpha_s,
\nu\rangle \alpha_s^\vee$ for all $\nu\in Y$.
\end{rem}

We now turn to the question of actually constructing root data
and $p$-isogenies for a Cartan matrix $C$. The key idea is contained
in the following remark.

\begin{rem} \label{Mcorbrulu1} Let $\cR=(X,R,Y, R^\vee)$ be a root datum
and let $\Pi$ be a base of $R$. Let $C$ be the corresponding Cartan matrix; 
see \ref{Mcartan0}. Let $\{\lambda_i\mid i\in I\}$ be a $\Z$-basis of $X$, 
where $I$ is some finite indexing set. (We have $|I|\geq |\Pi|$.) Let 
$\{\nu_i\mid i\in I\}$ be the corresponding dual basis of $Y$ (with 
respect to the pairing $\langle \;,\;\rangle$). Then we have unique 
expressions
\[ \beta=\sum_{i\in I} a_{\beta,i} \lambda_i \quad \mbox{and}\quad
\beta^\vee=\sum_{i\in I} \breve{a}_{\beta,i} \nu_i \qquad \mbox{for all
$\beta \in \Pi$}, \]
where $a_{\beta,i}\in \Z$ and $\breve{a}_{\beta,i}\in\Z$ for all 
$i\in I$. Consequently, we have a factorisation
\[ C=\breve{A}\cdot A^{\text{tr}} \quad \mbox{where} \quad
A=(a_{\beta,i})_{\beta\in \Pi,i \in I} \quad \mbox{and}\quad
\breve{A}=(\breve{a}_{\beta,i})_{\beta\in \Pi,i \in I}.\]
Note that all of $R \subseteq X$ and $R^\vee \subseteq Y$ are uniquely 
determined by $C,A$ and $\breve{A}$. (This follows from \ref{Mcartan0}(c) 
and the fact that $\bW=\langle w_\beta\mid \beta\in \Pi\rangle$;
see \ref{Mcartan0}(e).) We also note that the dual root system $(Y,R^\vee,
X,R)$ (see Lemma~\ref{Mweylfinite1}) has Cartan matrix $C^{\text{tr}}$, 
with corresponding factorisation $C^{\text{tr}}=\breve{B} \cdot 
B^{\text{tr}}$ where $B=\breve{A}$ and $\breve{B}=A$.
\end{rem}

Following \cite[\S 2.1]{BruLu}, we shall now reverse this argument and
show that {\em  every} factorisation of a Cartan matrix $C$ as above leads 
to a root datum. 

\begin{abs} \label{Mcartan1} Let $S$ be a finite set and $C=
(c_{st})_{s,t\in S}$ be a Cartan matrix, that is, a matrix satisfying 
the conditions in Proposition~\ref{Mcartan}. Let $I$ be another finite 
index set, with $|I|\geq |S|$, and assume that we have a factorisation
\[ C=\breve{A}\cdot A^{\text{tr}} \quad \mbox{where} \quad
A=(a_{s,i})_{s\in S,i \in I} \quad \mbox{and}\quad
\breve{A}=(\breve{a}_{s,i})_{s\in S,i \in I};\]
here, $A,\breve{A}$ are matrices with integer coefficients and both have
size $|S| \times |I|$. Let $X$ be a free abelian group with $\Z$-basis 
indexed by $I$, say $\{\lambda_i \mid i\in I\}$; also let $Y$ be a 
free abelian group with $\Z$-basis labelled by $I$, say $\{\nu_i \mid 
i\in I\}$. We define a bilinear pairing $\langle \;,\;\rangle \colon X 
\times Y \rightarrow \Z$ such that $\{\lambda_i \mid i\in I\}$ and $\{\nu_i
\mid i\in I\}$ are dual bases to each other. Then $C=\bigl(\langle \alpha_t,
\alpha_s^\vee\rangle\bigr)_{s,t \in S}$, where we set
\[ \alpha_s:=\sum_{i\in I} a_{s,i} \lambda_i \qquad \mbox{and}\qquad
\alpha_s^\vee:=\sum_{i\in I} \breve{a}_{s,i} \nu_i \qquad \mbox{for all
$s\in S$}.\]
For $s \in S$, we define endomorphisms $w_s \colon X \rightarrow X$ 
and $w_s^\vee \colon Y \rightarrow Y$ by
\[w_s(\lambda)=\lambda-\langle \lambda,\alpha_s^\vee\rangle \alpha_s 
\qquad \mbox{and}\qquad w_s^\vee(\nu)=\nu-\langle \alpha_s,\nu
\rangle \alpha_s^\vee\]
for all $\lambda \in X$ and $\nu \in Y$. Then $w_s^2=\id_X$ and
$(w_s^\vee)^2=\id_{Y}$. Hence, we have $w_s\in \Aut(X)$ and 
$w_s^\vee \in \Aut(Y)$ for all $s\in S$. Let 
\[W:=\langle w_s \mid s\in S\rangle \subseteq \Aut(X) \qquad \mbox{and}
\qquad W^\vee:=\langle w_s^\vee \mid s\in S\rangle \subseteq \Aut(Y).\]
Finally, let $R:=\{w(\alpha_s) \mid w \in W,s\in S\}$ and
$R^\vee:=\{w^\vee(\alpha_s^\vee) \mid w^\vee \in W^\vee, s\in S\}$.
\end{abs}

\begin{lem} \label{Mcartan2} The quadruple $\cR:=(X,R,Y, R^\vee)$ 
in \ref{Mcartan1} is a root datum with Cartan matrix $C$,
where $\{\alpha_s \mid s \in S\}$ is a base of $R$ and 
$\{\alpha_s^\vee \mid s \in S\}$ is a base of $R^\vee$. Furthermore, we have
$W=\bW$ and $W^\vee=\bW^\vee$ (with the notation of \ref{Mabs21}). The 
bijection $R\rightarrow R^\vee$, $\alpha\mapsto \alpha^\vee$, is determined 
as follows. If $w\in \bW$ and $s\in S$ are such that $\alpha=w(\alpha_s)$, 
then $\alpha^\vee=\delta(w)(\alpha_s^\vee)$ (with $\delta$ as in 
Lemma~\ref{Mweylfinite1}(a)). 
\end{lem}

\begin{proof} Let $V\subseteq \Q\otimes_{\Z} X$ be the subspace spanned
by $\{\alpha_s\mid s\in S\}$. Then $w_s(\alpha_t)=\alpha_t-c_{st}\alpha_s$
for all $t \in S$. So $w_s(V) \subseteq V$ for all $s\in S$. Let $W_R
\subseteq \GL(V)$ be the group generated by the restrictions $w_s \colon
V \rightarrow V$; then $R=\{w(\alpha_s) \mid w \in W_R, s\in S\}\subseteq V$.
Thus, we are in the setting of \cite[\S 1.1]{GePf}.

The matrix $C$ is {\em symmetrizable}, that is, there exist positive
numbers $\{d_s \mid s\in S\}$ such that $(d_s c_{st})_{s,t\in S}$ is a
symmetric matrix. (This easily follows from the fact that there are no
closed paths in the Dynkin diagram of $C$; see also \cite[\S 4.6]{Kac}.) 
Then we can define a $W_R$-invariant symmetric bilinear form on $V$ by 
$(\alpha_s, \alpha_t)=d_sc_{st}/2$ for $s,t\in S$. The $W_R$-invariance
implies that
\[c_{st}=2\frac{(\alpha_s,\alpha_t)}{(\alpha_s,\alpha_s)} \qquad
\mbox{for all $s,t \in S$};\]
see \cite[1.3.2]{GePf}. By (C2), this form is positive-definite; furthermore,
each $w_s\colon V \rightarrow V$ is an orthogonal reflection with root
$\alpha_s$. So $W_R$ and $R$ are finite; see \cite[Chap.~V, \S 8]{bour} or
\cite[1.3.8]{GePf}. In fact, $R$ is a root system in $V$ with Weyl group
$W_R$, with $\{\alpha_s\mid s\in S\}$ as base and $C$ as Cartan matrix; see 
\cite[1.1.10]{GePf}. Also note that $R$ is reduced, that is, (R2) holds; 
see \cite[1.3.7]{GePf}.  Let $R^+$ be the set of positive roots in $R$ 
defined by the base $\{\alpha_s\mid s \in S\}$. 

Similarly, let $V^\vee\subseteq \Q\otimes_{\Z} Y$ be the 
subspace spanned by $\{\alpha_s^\vee\mid s\in S\}$. Then
$w_s^\vee(\alpha_t^\vee)=\alpha_t^\vee-c_{ts}\alpha_s^\vee$
for all $t \in S$. Let $W_{R^\vee}\subseteq \GL(V^\vee)$ be the group 
generated by the restrictions $w_s^\vee \colon V^\vee \rightarrow V^\vee$. 
Again, we are in the setting of \cite[\S 1.1]{GePf} (with respect to
$C^{\text{tr}}$,  the transpose of $C$) and so we can repeat the previous 
argument. Consequently, $W_{R^\vee}$ is also finite; furthermore,
$R^\vee$ is a reduced root system in $V^\vee$ with Weyl group 
$W_{R^\vee}$, with $\{\alpha_s^\vee\mid s\in S\}$ as base
and $C^{\text{tr}}$ as Cartan matrix. Let $(R^\vee)^+$ be the set of 
positive roots in $R^\vee$ defined by the base $\{\alpha_s^\vee
\mid s \in S\}$. 

Next, we define a linear map $f\colon V \rightarrow V^\vee$ by 
$f(\alpha_s)=\frac{1}{2}(\alpha_s,\alpha_s)\alpha_s^\vee$ for all 
$s \in S$. (This is analogous to the definition in the proof of 
Lemma~\ref{Mweylfinite2}.) Clearly, $f$ is bijective. One immediately
checks that $w_s^\vee\circ f=f\circ w_s$ for all $s\in S$. So the map 
$w\mapsto f\circ w\circ f^{-1}$ defines a group isomorphism $\delta\colon 
W_R\stackrel{\sim}{\rightarrow} W_{R^\vee}$ such that $\delta(w_s)=
w_s^\vee$ for all $s\in S$. Consequently, we can define a bijection 
$R \rightarrow R^\vee$, $\alpha\mapsto \alpha^\vee$, as follows. 
First, let $\alpha\in R^+$. By definition, $\alpha=w(\alpha_s)$ for some 
$w\in W$, $s\in S$. Then 
\begin{center}
$f(\alpha)=f(w(\alpha_s))=\delta(w)(f(\alpha_s))=
\frac{1}{2}(\alpha_s,\alpha_s)\delta(w)(\alpha_s^\vee)$.
\end{center}
Since $\alpha_s^\vee\in R^\vee$, we have $\delta(w)(\alpha_s^\vee)
\in R^\vee$ by definition. Thus, $f(\alpha) \in V^\vee$ is a 
positive scalar multiple of some element of $R^\vee$; since 
$R^\vee$ is reduced, there is a unique positive root with this property,
denoted $\alpha^\vee$, and the above computation shows that $\alpha^\vee=
\delta(w)(\alpha_s^\vee)$. The definition of $\alpha^\vee$ for negative 
$\alpha$ is analogous; we then have $(-\alpha)^\vee=-\alpha^\vee$
for all $\alpha \in R$. Consequently, we obtain a map $R\rightarrow 
R^\vee$, $\alpha \mapsto \alpha^\vee$, which is easily seen to be 
bijective. Once this is established, the maps $w_\alpha\colon 
X\rightarrow X$ and $w_\alpha^\vee\colon Y\rightarrow Y$ 
are defined for any $\alpha \in R$. The $W_R$-invariance of $(\;,\;)$ 
implies that
\[\langle \beta,\alpha^\vee\rangle=2\frac{(\alpha,\beta)}{(\alpha,\alpha)}
\qquad \mbox{for all $\alpha,\beta \in R$}.\]
In particular, we obtain $\langle \alpha,\alpha^\vee\rangle=2$; thus,
(R1) holds. Finally, we show that (R3) holds. For each $\alpha\in R$, 
we claim that $w_\alpha(R)=R$. First one notices that $w_\alpha(V)
\subseteq V$. But the above formula for $\langle \beta,\alpha^\vee\rangle$ 
shows that the restriction $w_\alpha \colon V \rightarrow V$ is the 
orthogonal reflection with root $\alpha$. Hence, since $W_R$ is the Weyl 
group of $R$, we have $w_\alpha(R)=R$, as desired. The argument for 
$w_\alpha^\vee$ is analogous.
\end{proof}

In view of Remark~\ref{Mcorbrulu1}, the above construction yields all root 
data up to isomorphism. Thus, given a Cartan matrix $C$, we can think of 
the various root data of Cartan type $C$ simply as factorisations $C=
\breve{A}\cdot A^{\text{tr}}$, where $A,\breve{A}$ are integer matrices 
of the same size. (This observation, in this explicit form, appears in 
\cite{BruLu}. It is also implicit in \cite[\S 1]{Lu89}, 
\cite[\S\S 1--3]{Lu09d}.)

\begin{exmp} \label{expadjsc} Let $C$ be a Cartan matrix. We have just seen
that any factorisation $C=\breve{A}\cdot A^{\text{tr}}$ as in \ref{Mcartan1}
gives rise to a root datum $\cR=(X,R,Y,R^\vee)$. Obviously, there are two 
natural choices for such a factorisation, namely, 
\begin{itemize}
\item either $A$ is the identity matrix and, hence, $\breve{A}=C$;
\item or $\breve{A}$ is the identity matrix and, hence, $A=C^{\text{tr}}$.
\end{itemize}
In the first case, we denote the corresponding root datum by
$\cR=\cR_{\text{ad}}(C)$. We have $X=\Z R$ in this case; any root datum 
satisfying $X=\Z R$ will be called a \nm{root datum of adjoint type}. 
In the second case, we denote the corresponding root datum 
by $\cR= \cR_{\text{sc}}(C)$. We have $Y=\Z R^\vee$ in this case; 
any root datum satisfying $Y=\Z R^\vee$ will be called a \nm{root datum 
of simply-connected type}. 

Thus, $\cR_{\text{ad}}(C)$ and $\cR_{\text{sc}}(C)$ may be regarded as
the standard models of root data of adjoint type and simply-connected 
type, respectively. (See also Example~\ref{Misoadsc}.) The relevance of 
these notions will be become clearer when we consider semisimple algebraic 
groups in Section~\ref{sec:semisimple}.
\end{exmp}

\begin{exmp} \label{MdirprodR} There is an obvious notion of direct product
of root data. Indeed, if $\cR_i=(X_i,R_i, Y_i,R_i^\vee)$ for $i=1,\ldots,n$ 
are root data, then we obtain a new root datum $\cR=(X,R,Y,R^\vee)$
as follows. We set 
\begin{alignat*}{2}
X&:=X_1\oplus \ldots \oplus X_n, \qquad & R&:=\dot{R}_1\cup \ldots \cup 
\dot{R}_n,\\ Y&:=Y_1\oplus \ldots \oplus Y_n, \qquad &
R^\vee&:= \dot{R}_1^\vee\cup \ldots \cup \dot{R}_n^\vee, 
\end{alignat*}
where, for each $i$, we let $\dot{R}_i\subseteq X$ denote the image
of $R_i$ under the natural embedding $X_i\hookrightarrow X$; similarly,
$\dot{R}_i^\vee\subseteq Y$ denotes the image of $R_i^\vee$ under the 
natural embedding $Y_i\hookrightarrow Y$. Furthermore, the perfect 
bilinear pairings for the various $\cR_i$ define a unique perfect bilinear 
pairing for $\cR$ in a natural way. Also note that, if $\Pi_i$ is a base 
of $R_i$ for $i=1,\ldots,n$, then $\Pi:=\dot{\Pi}_1\cup \ldots \cup 
\dot{\Pi}_n$ is a base of $R$. 

In terms of the matrix language of \ref{Mcartan1}, the situation 
is described as follows. Each $\cR_i$ is determined by a factorisation 
$C_i=\breve{A}_i\cdot A_i^{\text{tr}}$ where $C_i$ is the Cartan matrix of 
$R_i$ with respect to a base $\Pi_i$ of $R_i$. Then $\cR$ is determined by
the factorisation $C=\breve{A}\cdot A^{\text{tr}}$, where $C$, $A$ and 
$\breve{A}$ are block diagonal matrices with diagonal blocks given by $C_i$, 
$A_i$ and $\breve{A}_i$, respectively. The matrix $C$ is the Cartan matrix 
of $R$ with respect to the base $\Pi=\dot{\Pi}_1\cup \ldots \cup 
\dot{\Pi}_n$. 
\end{exmp}

We now translate the conditions in Definition~\ref{MdefisogR} into the
matrix language of Remark~\ref{Mcorbrulu1}. This will be an extremely 
efficient tool for constructing isogenies, as it reduces the conditions 
to be checked to the verification of simple matrix identities. 

\begin{abs} \label{Mdefmatrixisog} Let $\cR=(X,R,Y,R^\vee)$ and 
$\cR'=(X',R', Y',R'^\vee)$ be root data. Assume that $X$ and $X'$ have
the same rank and that $R$ and $R'$ have bases indexed by the same set $S$.
Denote these bases by $\Pi=\{\alpha_s\mid s\in S\}$ and $\Pi'=\{\beta_s\mid 
s\in S\}$, respectively. Let $C$ and $C'$ be the corresponding Cartan 
matrices. Let us also fix a $\Z$-basis $\{\lambda_i\mid i\in I\}$ of $X$ 
and a $\Z$-basis $\{\lambda_j' \mid j\in J\}$ of $X'$. Then $\cR$ and 
$\cR'$ are determined by factorisations as in Remark~\ref{Mcorbrulu1}: 
\begin{align*}
C&=\breve{A}\cdot A^{\text{tr}} \quad \mbox{where} \quad
A=(a_{s,i})_{s\in S,i \in I} \quad \mbox{and}\quad
\breve{A}=(\breve{a}_{s,i})_{s\in S,i \in I},\\
C'&=\breve{B}\cdot B^{\text{tr}} \quad \mbox{where} \quad
B=(b_{s,j})_{s\in S,j \in J} \quad \mbox{and}\quad
\breve{B}=(\breve{b}_{s,j})_{s\in S,j \in J}.
\end{align*}
(Here, $|I|=|J|$, since $X$, $X'$ have the same rank.)
Now giving a linear map $\varphi\colon X'\rightarrow X$ is the same as 
giving a matrix $P=(p_{ij})_{i\in I,j\in J}$ with integer coefficients: 
\[ \varphi(\lambda_j')=\sum_{i\in I} p_{ij} \lambda_i \qquad
\mbox{for all $j\in J$}.\]
Assume now that $\varphi\colon X'\rightarrow X$ is a linear map which 
is ``base preserving'', in the sense that there is a permutation 
$S\rightarrow S$, $s\mapsto s^\dagger$, such that 
\[\varphi(\beta_{s^\dagger})=q_s\alpha_s \qquad \mbox{where $0\neq 
q_s\in \Z$ for all $s\in S$}.\]
We encode this in a monomial matrix $P^\circ= (p_{st}^\circ)_{s,t
\in S}$ where $p_{ss^\dagger}^\circ=q_s$ for $s\in S$. Let $p=1$ or 
$p$ be a prime number and assume that $\varphi$ is a $p$-isogeny. Then
the conditions in Definition~\ref{MdefisogR} immediately imply that the 
following conditions hold.
\begin{itemize}
\item[(MI1)] $P^\circ$ is a monomial matrix whose non-zero entries are all 
powers of~$p$.
\item[(MI2)] $P$ is invertible over $\Q$; furthermore, $P\cdot B^{\text{tr}}
 =A^{\text{tr}}\cdot P^\circ$ and $P^\circ \cdot \breve{B}=\breve{A}\cdot 
P$.
\end{itemize}
Conversely, it is straightforward to check that {\em any} pair of integer
matrices $(P,P^\circ)$ satisfying (MI1) and (MI2) defines a $p$-isogeny of
root data. The argument is similar to the proof of Lemma~\ref{Mcartan2}; 
let us just briefly sketch it. Let $\varphi \colon X'\rightarrow X$ be the
linear map with matrix $P$. Condition (I1) holds since $P$ is invertible 
over $\Q$. Since $P^\circ$ is monomial, there is a permutation 
$S\rightarrow S$, $s\mapsto s^\dagger$, such that 
\[q_s:=p_{ss^\dagger}^\circ \neq 0\qquad \mbox{for all $s\in S$}.\] 
Then (MI2) means that $\varphi(\beta_{s^\dagger})=q_s\alpha_s$ and
$\varphi^\vee(\alpha_s^\vee)=q_s\beta_{s^\dagger}^\vee$ for all $s\in S$.
Thus, (I2) holds for simple roots and coroots. To see that (I2) holds 
for all roots and coroots, note that (MI2) implies that $C\cdot P^\circ 
= P^\circ \cdot C'$. Consequently, using \ref{Mcartan0}(e) 
for $\bW$ and for $\bW'$, there is a unique group isomorphism 
\[ \sigma\colon \bW \rightarrow \bW' \qquad \mbox{such that}\qquad 
w_{\alpha_s}\mapsto w_{\beta_{s^\dagger}} \qquad (s\in S).\]
Using Lemma~\ref{Mweylfinite2}, one shows that this implies that 
\[\varphi\circ \sigma(w)=w \circ \varphi\qquad\mbox{for all 
$w\in \bW$}.\]
We can now define a bijection $R\rightarrow R'$, $\alpha\mapsto
\alpha^\dagger$, with the required properties, as follows. Let $\alpha 
\in R$ and write $\alpha=w(\alpha_s)$ for some $w\in \bW$ and $s\in S$. 
Then we set $\alpha^\dagger:=\sigma(w)(\beta_{s^\dagger})\in R'$. Now, we 
have  
\[ \varphi(\alpha^\dagger)=\varphi(\sigma(w)(\beta_{s^\dagger}))=w(\varphi
(\beta_{s^\dagger}))=q_s w(\alpha_s)=q_s \alpha.\]
Since $\varphi$ is injective, $\alpha^\dagger$ is uniquely determined by 
$\alpha$ (and does not depend on the choice of $w$ and $s$); furthermore,
the first of the two identities in (I2) holds, where $q_\alpha=q_s$. The 
argument for the second identity is similar, using the bijection 
$R\rightarrow R^\vee$ (see Lemma~\ref{Mcartan2}). Thus, (I2) is seen to 
hold for all roots and coroots.
\end{abs}


\begin{exmp}[Cf.\ \protect{\cite[\S\S 21.5, 22.4, 23.7]{Ch05}}] 
\label{MsuzukiB}  
Let $C=(c_{ij})_{1\leq i,j\leq r}$ ($r=2$ or $4$) be a Cartan matrix of 
type $C_2$, $G_2$ or $F_4$; see Table~\ref{Mdynkintbl}. Explicitly:
\[C_2:\; \left(\begin{array}{rr} 2 & -1 \\-2 & 2 \end{array}\right),\qquad
G_2:\; \left(\begin{array}{rr} 2 & -1 \\-3 & 2 \end{array}\right),\qquad
F_4:\;\left(\begin{array}{rrrr} 2 & -1 & 0 & 0 \\-1 & 2 & -1 & 0 \\
0 & -2 & 2 & -1 \\ 0 & 0 & -1 & 2  \end{array}\right).\]
(Note that $C_2=B_2$ up to relabeling the two vertices of the Dynkin 
diagram.) We set $p=2$ if $C$ is of type $C_2$ or $F_4$, and $p=3$ if $C$ is 
of type $G_2$. Let us consider the corresponding root datum $\cR=
\cR_{\text{ad}}(C)=(X,R,Y,R^\vee)$ as in Example~\ref{expadjsc}; we have
$C=\breve{A}\cdot A^{\text{tr}}$, where $A$ is the identity matrix and 
$C=\breve{A}$. For any $m\geq 0$, we define two matrices $P_m^\circ$ and 
$P_m$ as follows:
\begin{align*}
C_2 &: \qquad\qquad  P_m=P_m^\circ:=\left(\begin{array}{cc} 0 & 2^{m} \\
2^{m+1} & 0 \end{array}\right),\\
G_2 &: \qquad\qquad  P_m=P_m^\circ:=\left(\begin{array}{cc} 0 & 3^{m} \\
3^{m+1} & 0 \end{array}\right),\\
F_4 &: \qquad P_m=P_m^\circ:=\left(\begin{array}{cccc} 0 & 0 & 0 & 2^{m} \\
0 & 0 & 2^m & 0 \\ 0 & 2^{m+1} & 0 & 0 \\ 2^{m+1} & 0 & 0 & 0
\end{array}\right).
\end{align*}
Now, in the setting of \ref{Mdefmatrixisog}, let $C'=C$, $B=A$, 
$\breve{B}=\breve{A}$. Then $P_m,P_m^\circ$ satisfy (MI1), (MI2) and, hence, 
the pair $(P_m,P_m^\circ)$ defines a group homomorphism $\varphi_m\colon 
X\rightarrow X$ which is a $p$-isogeny of $\cR$ into iself, such that 
$\varphi_m^2=p^{2m+1}\, \id_X$. 

See \cite[12.3, 12.4]{Ca1} for a more detailed discussion of these 
''exceptional'' isogenies; they give rise to the finite Suzuki and Ree
groups (see Example~\ref{Msuzuki2}). Another instance of such an
''exceptional'' isogeny will be considered in Example~\ref{isogBC}.
\end{exmp}

\begin{lem} \label{Misorootdat} Let $C$ be a Cartan matrix. Let 
$\cR=(X,R,Y,R^\vee)$ and $\cR'=(X',R',Y',R'^\vee)$ be root data which are
both of Cartan type $C$. Choosing bases for $R,R'$ and for $X,X'$, let
$C=\breve{A}\cdot A^\trp=\breve{B}\cdot B^\trp$ be the corresponding
factorisations as in Remark~\ref{Mcorbrulu1} (where $A,\breve{A}$
correspond to $\cR$ and $B,\breve{B}$ correspond to $\cR'$). Then
$\cR,\cR'$ are isomorphic root data if and only if there exist
square matrices $P,P^\circ$ with integer coefficients such that
$P^\circ$ is a permutation matrix, $P$ is invertible over $\Z$ and
we have $P\cdot B^\trp=A^\trp \cdot P^\circ$, $P^\circ \cdot \breve{B}=
\breve{A}\cdot P$.
\end{lem}

\begin{proof} Recall from Definition~\ref{MdefisogR} that a $p$-isogeny 
$\varphi\colon X' \rightarrow X$ is an isomorphism of root data if and 
only if $\varphi$ is bijective and $q_\alpha=1$ for all $\alpha\in R$
Furthermore,  replacing $\varphi$ by $\varphi \circ w$ for some 
$w \in \bW'$ (the Weyl group of $\cR'$) if necessary, we may assume that
$\varphi$ sends the chosen base of $R'$ to the chosen base of $R$ (see 
Remark~\ref{Misosimple}). Hence, in the setting of \ref{Mdefmatrixisog}, 
$\varphi$ is an isomorphism if and only if $P^\circ$ is a permutation 
matrix and $P$ is invertible over $\Z$. 
\end{proof}

\begin{exmp} \label{Misoadsc} Assume that $C$ is a Cartan matrix of
type $G_2$, $F_4$ or $E_8$. Then $C$ is invertible over $\Z$ and, hence,
$\cR_{\text{ad}}(C)$ and $\cR_{\text{sc}}(C)$ are isomorphic root data. 
Indeed, $\cR_{\text{ad}}(C)$ corresponds to the factorisation
$C=\breve{A}\cdot A^\trp$ where $A$ is the identity matrix, while 
$\cR_{\text{sc}}(C)$ corresponds to the factorisation $C=\breve{B}\cdot 
B^\trp$ where $\breve{B}$ is the identity matrix. Then the conditions in 
Lemma~\ref{Misorootdat} hold, where $P=C^{-1}$ and $P^\circ$ is the 
identity matrix.
\end{exmp}

\section{Chevalley's classification theorems} \label{sec:chevalley}
Throughout this section, let $k$ be an algebraically closed field and 
$\bG$ be a linear algebraic group over $k$. We can now explain how one can 
naturally attach to $\bG$ a root datum, when $\bG$ is connected reductive.

\begin{abs} \label{subsec1roots} Assume that $\bG$ is connected reductive.
Let $\bT\subseteq \bG$ be a maximal torus, $X=X(\bT)$ and $L(\bG)$ be the 
Lie algebra. Recall from \ref{subsec1weights} that there is a finite subset
$R\subseteq X$ and a corresponding root space decomposition of $L(\bG)$:
\[L(\bG)=L(\bG)_0\oplus \bigoplus_{\alpha\in R} L(\bG)_\alpha.\]
As already mentioned in \ref{subsec1weights}, we have $L(\bG)_0=L(\bT)$, 
$R=-R$ and $\dim L(\bG)_\alpha=1$ for all $\alpha\in R$; in particular, 
\[\dim \bG= \dim L(\bG)=\dim \bT+|R|.\] 
The roots can be directly characterised in terms of $\bG$, as follows.
Let $\alpha\in X$. Then $\alpha$ is a root if and only if there exists
a homomorphism of algebraic groups $u_\alpha\colon \bkp\rightarrow \bG$
such that $u_\alpha$ is an isomorphism onto its image and we have 
\[tu_\alpha(\xi)t^{-1}=u_\alpha(\alpha(t)\xi)\qquad \mbox{for all $t \in \bT$ 
and $\xi\in k$}. \]
Thus, $\bU_\alpha:=\{u_\alpha(\xi)\mid \xi \in k\}\subseteq\bG$ is a 
one-dimensional closed connected unipotent subgroup normalized by $\bT$. It
is uniquely determined by $\alpha$ and called the \nm{root subgroup} 
corresponding to $\alpha$. Conversely, every one-dimensional 
closed connected unipotent subgroup normalized by $\bT$ is equal to 
$\bU_\alpha$ for some $\alpha\in R$.  We have 
\[\bG=\langle \bT, \bU_\alpha \mid \alpha \in R\rangle.\]
Now consider also the co-character group $Y=Y(\bT)$; we wish to define
a finite subset $R^\vee \subseteq Y$. Recall from \ref{subsec17} that 
$X$, $Y$ are free abelian groups of the same (finite) rank and that there 
is a natural pairing $\langle \;,\; \rangle \colon X \times Y \rightarrow 
\Z$. The Weyl group of $\bG$ with respect to $\bT$ is defined as 
$\bW(\bG,\bT):=N_{\bG}(\bT)/\bT$. Since $N_{\bG}(\bT)$ acts on $\bT$ by 
conjugation, we have induced actions of $\bW(\bG,\bT)$ on $X$ and on $Y$ via
\begin{align*}
(w.\lambda)(t)&= \lambda(\dot{w}^{-1}t\dot{w}) 
\qquad (\lambda\in X,\; t\in \bT),\\ (w.\nu)(\xi)&= 
\dot{w}\nu(\xi) \dot{w}^{-1} \qquad (\nu\in Y,\; \xi\in \bkm),
\end{align*}
where, for any $w\in\bW(\bG,\bT)$, we denote by $\dot{w}$ a representative 
in $N_{\bG}(\bT)$. Using these actions, we can identify $\bW(\bG,\bT)$ with 
subgroups of $\Aut(X)$ and of $\Aut(Y)$. Now let $\alpha\in R$. Then 
$\bG_\alpha:=C_{\bG}(\ker(\alpha)^\circ)=\langle \bT,\bU_\alpha,
\bU_{-\alpha}\rangle$ is a closed connected reductive subgroup of $\bG$;
its Weyl group $\bW(\bG_\alpha,\bT):=N_{\bG_\alpha}(\bT)/\bT$ has 
order~$2$. Let $w_\alpha\in \bW(\bG_\alpha,\bT)$ be the non-trivial
element and $\dot{w}_\alpha$ be a representative of $w_\alpha$ in 
$N_{\bG_\alpha}(\bT) \subseteq N_{\bG}(\bT)$. Then there exists a unique 
$\alpha^\vee\in Y$ such that 
\[w_\alpha.\lambda=\lambda-\langle \lambda,\alpha^\vee\rangle \alpha
\qquad \mbox{for all $\lambda\in X$}.\] 
Following, e.g., \cite[1.2.8]{Con14}, this element $\alpha^\vee$ can also 
be determined as follows. The maps $u_{\pm \alpha} \colon \bkp \rightarrow 
\bU_{\pm \alpha}$ can be chosen such that the assignment 
\[ \left(\begin{array}{cc} 1 & \xi \\ 0 & 1 \end{array}\right)\mapsto 
u_\alpha(\xi),\qquad  \left(\begin{array}{cc} 1 & 0 \\ \xi & 1
\end{array}\right)\mapsto u_{-\alpha}(\xi)\qquad (\xi \in k)\]
defines a homomorphism of algebraic groups $\varphi_\alpha \colon 
\SL_2(k) \rightarrow \bG$. Then we have 
\[ \alpha^\vee(\xi)=\varphi_\alpha \left(\begin{array}{cc} \xi & 0 \\ 
0 & \xi^{-1} \end{array} \right)\in \bT \qquad \mbox{for all $\xi\in\bkm$}.\]
Thus, we obtain a well-defined finite subset $R^\vee= \{\alpha^\vee\mid 
\alpha\in R\} \subseteq Y$; we have
\[ \bW(\bG,\bT)=\langle w_\alpha \mid \alpha\in R\rangle.\] 
Complete proofs of the above statements can be 
found in \cite{Bor}, \cite{Hum}, \cite{Spr} and, of course, the original 
source \cite{Ch05}. A thorough guide through this argument, with indications 
of the proofs and many worked-out examples, can be found in 
\cite[\S 8]{MaTe}. (See also \cite{All09}, \cite[Chap.~II]{Jan}.)

With this notation, we can now state the following result which shows that 
we are exactly in the situation described by Proposition~\ref{algbnpair}.
\end{abs}

\begin{thm} \label{MrootdatumG} The quadruple $\cR=(X(\bT),R, Y(\bT), 
R^\vee)$ in \ref{subsec1roots} is a root datum as defined in \ref{Mabs21}, 
with Weyl group $\bW(\bG,\bT)$ (identified with a subgroup of $\Aut(X(\bT))$ 
as above). Furthermore, let 
$R^+\subseteq R$ be the set of positive roots with respect to a base 
$\Pi\subseteq R$. Then $\bB:=\langle \bT,\bU_\alpha \mid \alpha \in R^+ 
\rangle \subseteq \bG$ is a Borel subgroup and $\bB, N_{\bG}(\bT)$ form 
a \nm{reductive $BN$-pair} in $\bG$ where $C_{\bG}(\bT)=\bT= \bB\cap 
\bN_{\bG}(\bT)$. 
\end{thm}

\begin{proof} In its essence, this is due to Chevalley \cite{Ch05},
but the notion of $BN$-pairs was not yet available at that time. A proof 
of the fact that $\cR$ is a root datum can be found, for example, in 
\cite[9.11]{MaTe}, \cite[7.4.3]{Spr}. The $BN$-pair axioms are shown in 
\cite[14.15]{Bor}, \cite[11.16]{MaTe}. For the equality 
$C_{\bG}(\bT)=\bT$, see \cite[8.13]{MaTe} or \cite[7.6.4]{Spr}. 
\end{proof}


\begin{thm} \label{MrootdatumG1} Assume that $\bG$ is connected reductive.
Then $\bG$ acts transitively (by 
simultaneous conjugation) on the set of all pairs $(\bT,\bB)$ where 
$\bT\subseteq \bG$ is a maximal torus and $\bB\subseteq \bG$ is a Borel 
subgroup such that $\bT \subseteq \bB$. In particular, the root data (as in 
Theorem~\ref{MrootdatumG}) with respect to any two maximal tori of $\bG$ 
are isomorphic in the sense of \ref{Mhomrootdata}.
\end{thm}

\begin{proof} The conjugacy results are due to Borel \cite[10.6, 
11.1]{Bor}; see also \cite[6.2.7, 6.3.5]{Spr}. A somewhat more elementary 
proof of the conjugacy of Borel subgroups is given in \cite{St1a}. (See
also \cite[\S 3.4]{mybook}.) Once these conjugacy results are shown, the
assertion about the isomorphism between root data is clear.
\end{proof}

\begin{rem} \label{MrootdatumG2} Let $\bG$, $\bT$, $\cR$, $\bB$ as in 
Theoren~\ref{MrootdatumG}; let $\bW:=\bW(\bG,\bT)$. Then the set of 
{\em all} Borel subgroups of $\bG$ containing $\bT$ is described as 
follows. Let $\bB_1$ be any Borel subgroup of $\bG$ containing $\bT$. 
By Theorems~\ref{MrootdatumG}, \ref{MrootdatumG1} and the $BN$-pair 
axioms, there is a unique $w\in \bW$ such that $\bB_1= \dot{w}^{-1}\bB 
\dot{w}$. Now, the base $\Pi\subseteq R$ used to define $\bB$ is 
transformed under $w$ to a new base $\Pi_1$ of $R$. Consequenty, we have 
$\bB_1=\langle \bT,\bU_\alpha\mid \alpha \in R_1^+\rangle$ where $R_1^+ 
\subseteq R$ is the set of positive roots with respect to $\Pi_1$. Further 
recall from \ref{Mcartan0} that any two bases of $R$ can be transformed 
into each other by a unique element of $\bW$. Thus, we obtain bijective 
correspondences
\[\{\mbox{Borel subgroups containing $\bT$}\}\quad
\stackrel{1{-}1}{\longleftrightarrow}\quad \bW \quad 
\stackrel{1{-}1}{\longleftrightarrow}\quad \{\mbox{bases of $R$}\}.\]
\end{rem}

\begin{rem} \label{Msemisimple} Assume that $\bG$ is connected reductive.
In \ref{subsec15} we have defined $\bG$ to be semisimple if $|\bZ|<\infty$ 
where $\bZ=\bZ(\bG)$ denotes the center of $\bG$; alternatively, $\bG$ is 
semisimple if and only if $\bG=\bGder$. We also have the following 
characterisation in terms of the root datum $\cR=(X,R,Y,R^\vee)$ (with
respect to a maximal torus $\bT\subseteq \bG$). By \cite[8.17(h)]{MaTe}, 
\cite[8.1.8]{Spr}, we have 
\begin{equation*}
\bZ=\{t\in \bT \mid \alpha(t)=1\mbox{ for all $\alpha\in R$}\} \tag{a}
\end{equation*}
and the isomorphism $\bT\cong \Hom(X(\bT),k^\times)$ in \ref{subsec17} 
restricts to an isomorphism 
\begin{equation*}
\bZ\cong \Hom(X(\bT)/\Z R,\,k^\times). \tag{b}
\end{equation*}
Thus, we obtain the equivalences: 
\begin{equation*}
|\bZ|<\infty \quad \Longleftrightarrow \quad |X/\Z R|<\infty\quad
\Longleftrightarrow \quad |Y/\Z R^\vee|<\infty.\tag{c}
\end{equation*}
If we consider the factorisation $C=\breve{A} \cdot A^{\text{tr}}$
determined by $\cR$ as in Remark~\ref{Mcorbrulu1}, then $\bG$ is semisimple 
if and only if $A,\breve{A}$ are square matrices. 
\end{rem}

\begin{rem} \label{MdirprodG1} Assume that $\bG$ is connected reductive.
As in \ref{subsec15}, we have $\bG=\bZ^\circ.\bGder$; furthermore, $\bGder=
\bG_1 \ldots \bG_n$ where $\bG_1, \ldots,\bG_n$ are the closed normal 
simple subgroups of $\bG$; they commute pairwise with each other. These 
subgroups have the following description in terms of the root datum $\cR
=(X,R,Y,R^\vee)$ (with respect to a maximal torus $\bT\subseteq \bG$)
and the corresponding root subgroups $\bU_\alpha$ ($\alpha\in R$).
First note that 
\[\bGder=\langle \bU_\alpha\mid \alpha\in R\rangle,\] 
see \cite[8.21]{MaTe}. Now let $C$ be the Cartan matrix of the root system 
$R$, with respect to a base $\Pi$ of $R$. Then $C$ can be expressed as a 
block diagonal matrix where the diagonal blocks are indecomposable Cartan 
matrices, $C_1,\ldots,C_n$ say. (Thus, $C_1, \ldots,C_n$ correspond to the 
connected components of the Dynkin diagram of $C$.) Let $\Pi=\Pi_1 \sqcup
\ldots\sqcup \Pi_n$ be the corresponding partition of $\Pi$. Then we also 
have $R=R_1\sqcup\ldots \sqcup R_n$ where $R_i$ consists of all roots in 
$R$ which can be expressed as linear combinations of simple roots in 
$\Pi_i$. Then we have
\[\bG_i= \langle \bU_\alpha \mid \alpha\in R_i\rangle \subseteq \bG\qquad
\mbox{for $i=1,\ldots,n$}.\]
A maximal torus of $\bG_i$ is given by $\bT_i:=\bG_i\cap \bT$ where $\bT$ is
a fixed maximal torus of $\bG$. (See \cite[Chap.~IV, \S 11]{Bor},
\cite[\S 8.4]{MaTe}, \cite[\S 8.1]{Spr} for further details.) 
\end{rem}

Before continuing with the general theory, we give three concrete examples.
We shall see that the point of view in \ref{Mcartan1}, where root data
are described in terms of factorisations of Cartan matrices, provides
a particularly efficient and convenient way of encoding the information
involved in these examples.  

\begin{exmp} \label{rootdatGL} Let $\bG=\GL_n(k)$. Let $\bB\subseteq\bG$ 
be the subgroup of all upper triangular matrices and $\bN\subseteq 
\bG$ the subgroup of all monomial matrices. It is well-known that these 
groups form a $BN$-pair; see \cite[Chap.~IV, n$^\circ$ 2.2]{bour}.
For further details see \cite[1.6.10, 3.4.5]{mybook}, where it is also 
shown that this is an algebraic $BN$-pair satisfying the conditions 
in Proposition~\ref{algbnpair}; in particular, $\bG$ is connected reductive. 
Let us describe the root datum of $\bG$ with respect to the maximal torus 
$\bT=\bB\cap \bN$ consisting of all diagonal matrices in $\bG$. 

It will be convenient to introduce some notation concerning matrices.
For $1\leq i \leq n-1$, let $n_i$ be the matrix which is obtained by
interchanging the $i$-th and the $(i+1)$-th row in the identity matrix, 
which we denote by $I_n$. More generally, if $w\in\SG_n$ is any 
permutation, let $n_w$ be the matrix which is obtained by permuting 
the rows of $I_n$ as specified by $w$. (Thus, if $\{e_1,\ldots,e_n\}$
denotes the standard basis of $k^n$, then $n_w(e_i)=e_{w(i)}$ for 
$1\leq i\leq n$; we have $n_{ww'}=n_wn_{w'}$ for all $w,w'\in \SG_n$.) 
Then $\bN=\{hn_w\mid h \in \bT, w\in \SG_n\}$ and so we have
an exact sequence
\[ \{1\}\rightarrow \bT \rightarrow \bN\rightarrow \SG_n
\rightarrow \{1\},\]
where $\bN \rightarrow \SG_n$ sends $n_w$ to $w$. Next, for $1\leq i,j
\leq n$ let $E_{ij}$ be the ``elementary'' matrix with coefficient $1$ at 
the position $(i,j)$ and $0$ otherwise. We define
\[ \bU_{ij}:=\{ I_n+\xi E_{ij} \mid \xi \in k\} \qquad \mbox{where
$1\leq i,j\leq n$, $i\neq j$}.\]
All of these are one-dimensional, closed connected subgroups of $\bG$. 
Finally, if $\xi_1,\ldots,\xi_n$ are non-zero elements of $k$, 
we denote by $h(\xi_1,\ldots,\xi_n)\in\bT$ the diagonal matrix with
$\xi_1,\ldots,\xi_n$ along the diagonal. Then the map 
\[ (\bkm)^n \rightarrow \bT,\qquad (\xi_1,\ldots,\xi_n)\mapsto
h(\xi_1,\ldots,\xi_n),\]
certainly is an isomorphism of algebraic groups. Hence $X=X(\bT)$
is the free abelian group with basis $\lambda_1,\ldots,\lambda_n$ where
$\lambda_i(h(\xi_1,\ldots,\xi_n))=\xi_i$ for all $i$. 

Each subgroup $\bU_{ij}$ is normalised by $\bT$. Let $u_{ij}\colon \bkp
\rightarrow \bG$ be the homomorphism given by $u_{ij}(\xi)=I_n+\xi E_{ij}$
for $\xi\in \bkp$. Then $\bU_{ij}$ is the image of this homomorphism, 
$u_{ij}$ is an isomorphism onto its image and we have  
\[ tu_{ij}(\xi)t^{-1}=u_{ij}(\xi_i\xi_j^{-1}\xi) \qquad \mbox{where
$t=h(\xi_1,\ldots,\xi_n)\in \bT$ and $\xi\in \bkp$}.\]
Hence, $\alpha_{ij} :=\lambda_i-\lambda_j\in X$ is a root and $\bU_{ij}$ 
is the corresponding root subgroup. To see that these are all the roots,
we can use the formula $\dim \bG=\dim \bT+|R|$ in \ref{subsec1roots}. Thus, 
since $\dim \bG=n^2$ and $\dim \bT= n$, we conclude that $R=\{\alpha_{ij} 
\mid 1\leq i,j\leq n,i\neq j\}$ is the root system of $\bG$ with respect
to $\bT$. We also see that 
\[\Pi:=\{\alpha_{i,i+1}=\lambda_i-\lambda_{i+1} \mid 1\leq i\leq n-1\}
\subseteq R\]
is a base of $R$ and that $\bB$ is the Borel subgroup associated with 
this base (as in Remark~\ref{MrootdatumG2}). Now consider coroots. The 
dual basis of $Y=Y(\bT)$ is given by the co-characters 
$\nu_j\colon \bkm\rightarrow \bT$ such that $\nu_j(\xi)$ is the 
diagonal matrix with coefficient $\xi$ at position $j$, and coefficient $1$ 
otherwise. Now, for $i\neq j$, we have a unique embedding of algebraic
groups $\varphi_{ij} \colon \SL_2(k) \hookrightarrow \bG$ such that 
\[ \varphi_{ij}\left(\begin{array}{cc} 1 & \xi \\ 0 & 1 \end{array}\right)
=u_{ij}(\xi) \quad \mbox{and} \quad \varphi_{ij}\left(\begin{array}{cc} 
1 & 0 \\ \xi & 1 \end{array}\right)=u_{ji}(\xi)\quad \mbox{for all $\xi 
\in k$}.\] 
Hence, $\varphi_{ij}$ satisfies the condition in \ref{subsec1roots} and so
we obtain $\alpha_{ij}^\vee\in Y$ such that 
\[\alpha_{ij}^\vee(\xi)=\varphi_{ij}\left(\begin{array}{cc} \xi & 0 
\\ 0 & \xi^{-1} \end{array} \right)\in\bT\]
is the diagonal matrix with coefficient $\xi$ at position $i$ and 
coefficient $\xi^{-1}$ at position $j$.  Thus, we have 
$R^\vee=\{\alpha_{ij}^\vee=\nu_i-\nu_j\mid 1\leq i,j\leq n,i\neq j\}$. 
We also see that 
\[\Pi^\vee:=\{\alpha_{i,i+1}^\vee=\nu_i-\nu_{i+1} \mid 1\leq i\leq n-1\}
\subseteq R^\vee\]
is a base of $R^\vee$. The corresponding Cartan matrix $C=(c_{ij})_{1\leq i,
j\leq n-1}$  is given by 
\[ C=\left(\begin{array}{rrrrrrr} 2 & -1 & 0 & \ldots & & 0 \\
-1 & 2 & -1 & 0 & \ldots & 0 \\
0 & -1 & 2 & -1 & \ddots & \vdots \\
\vdots & \ddots & \ddots & \ddots & \ddots & 0\\
0 & \ldots & 0 & -1 & 2 & -1 \\
0 & & \ldots &  0 & -1 & 2 \end{array}\right).\]
Thus, $C$ is of type $A_{n-1}$. The factorisation in Remark~\ref{Mcorbrulu1} 
is given by
\[ C=\breve{A}\cdot A^{\text{tr}} \quad \mbox{where} \quad A=\breve{A}=
\left(\begin{array}{rrrrrrr} 1 & -1 & 0 & \ldots & & 0 \\
0 & 1 & -1 & 0 & \ldots & 0 \\
\vdots & \ddots & \ddots & \ddots & \ddots & \vdots\\
0 & \ldots & 0 & 1 & -1 & 0 \\
0 & & \ldots &  0 & 1 & -1 \end{array}\right).\]
(Here, $A=\breve{A}$ has $n-1$ rows and $n$ columns.)
\end{exmp}

\begin{exmp} \label{rootdatSL} Let $n \geq 2$ and $\bG'=\SL_n(k)$, the 
\nm{special linear group}. We keep the notation $\bG=\GL_n(k)$, $\bU_{ij}$,
$\bB$, $\bN$, $\bT$, $X=X(\bT)$, $Y=Y(\bT)$ from the previous example. 
Then an algebraic $BN$-pair satisfying 
the conditions in Proposition~\ref{algbnpair} is given by the subgroups 
$\bB':=\bB\cap \bG'$ and $\bN':= \bN\cap \bG'$; see  
\cite[Chap.~IV, \S 2, Exercise~10]{bour}, \cite[1.6.11, 3.4.5]{mybook}. Let
us describe the root datum of $\bG'$ with respect to the maximal torus
$\bT'=\bT\cap \bG'$. Let 
\[X'=X(\bT')\qquad \mbox{and}\qquad Y'=Y(\bT').\]
For $1\leq i,j\leq n$, $i\neq j$, the subgroup $\bU_{ij}$ of $\bG$ is 
already contained in $\bG'$. So, if $\alpha_{ij}'$ denotes the restriction 
of $\alpha_{ij}\in X$ to $\bT'$, then $\alpha_{ij}' \in X'$ and 
$\alpha_{ij}'$ is a root of $\bG'$ with corresponding root subgroup 
$\bU_{ij} \subseteq \bG'$. Since $\dim \bG'=n^2-1$ and $\bT'=n-1$,
it follows as above that $R'=\{\alpha_{ij}' \mid 1\leq i, j\leq n,i 
\neq j\}$ is the root system of $\bG'$ with respect to $\bT'$ and that
\[\Pi'=\{\alpha_{i,i+1}'\mid 1\leq i\leq n-1\} \mbox{ is a base for $R'$}.\]
Furthermore, the image of the embedding $\varphi_{ij}\colon \SL_2(k)
\hookrightarrow \bG$ is clearly contained in $\bG'$. Consequently, any 
coroot $\alpha_{ij}^\vee \in Y$ also is a coroot in $Y'$. Thus, we have 
$R'^\vee=R^\vee=\{\alpha_{ij}^\vee\mid 1\leq i,j\leq n,i\neq j\}$ and
\[\Pi'^\vee=\{\alpha_{i,i+1}^\vee\mid 1\leq i\leq n-1\} \mbox{ is a base 
for $R'^\vee$}.\]
In particular, we obtain the same Cartan matrix $C$ of type $A_{n-1}$
as in Example~\ref{rootdatGL}. Now note that we have an isomorphism of 
algebraic groups 
\[(\bkm)^{n-1}\rightarrow\bT', \qquad (\xi_1,\ldots,\xi_{n-1})\mapsto 
h\bigl(\xi_1,\ldots,\xi_{n-1},(\xi_1\cdots \xi_{n-1})^{-1}\bigr).\]
(Its inverse is given by sending $h(\xi_1,\ldots,\xi_n)\in \bT'$ to
$(\xi_1,\ldots,\xi_{n-1})\in (\bkm)^{n-1}$.) 
Hence, if we define co-characters $\nu_j'\colon \bkm\rightarrow \bT'$
(for $1\leq j\leq n-1$) such that $\nu_j'(\xi)$ is the diagonal matrix with
$\xi$ at position $j$ and $\xi^{-1}$ at position~$n$, then $\{\nu_1',\ldots,
\nu_{n-1}'\}$ is a $\Z$-basis of $Y'$. But then $\Pi'^\vee$ also is a 
$\Z$-basis of $Y'$. If we consider the corresponding dual basis of $X'$, 
then the factorisation in Remark~\ref{Mcorbrulu1} is given by
\[ C=\breve{A}\cdot A^{\text{tr}} \qquad \mbox{where $\breve{A}=I_{n-1}$ and
$A=C^{\text{tr}}$}.\]
Thus, $\bG'$ is semisimple and the root datum of $\SL_n(k)$ is of 
simply-connected type (see Example~\ref{expadjsc}).
\end{exmp}

\begin{exmp} \label{rootdatPGL} Let $\bG=\GL_n(k)$ and $\bZ\subseteq \bG$
be the center of $\bG$, consisting of all non-zero scalar matrices. Assume
that $n\geq 2$ and let $\bar{\bG}=\PGL_n(k):=\bG/\bZ$, the \nm{projective 
linear group}. (This is a linear algebraic group as discussed in 
\ref{subsec1quot}.) Let us denote the canonical map $\bG\rightarrow 
\bar{\bG}$ by $g\mapsto \bar{g}$. In particular, we obtain subgroups 
$\bar{\bB}$ and $\bar{\bN}$ in $\bar{\bG}$ which form a $BN$-pair since 
$\bZ\subseteq \bB$; see \cite[Chap.~IV, \S 2, Exercise~2]{bour}. One easily 
checks that this is an algebraic $BN$-pair satisfying the conditions in 
Proposition~\ref{algbnpair}. Let us describe the root datum of $\bar{\bG}$ 
with respect to the maximal torus $\bar{\bT}$ of $\bar{\bG}$. Let 
\[\bar{X}=X(\bar{\bT})\qquad \mbox{and}\qquad \bar{Y}=
Y(\bar{\bT}).\]
For every root $\alpha$ of $\bG$, we clearly have $\bZ\subseteq 
\ker(\alpha)$. So, using the universal property of quotients, there is 
a well-defined $\bar{\alpha}\in \bar{X}$ such that $\alpha(t)=
\bar{\alpha}(\bar{t})$ for all $t\in \bT$. Now, for $1\leq i,j \leq n$, 
$i\neq j$, the image $\bar{\bU}_{ij}$ of the subgroup $\bU_{ij}\subseteq
\bG$ in $\bar{\bG}$ is still closed, connected, isomorphic to $\bkp$ and 
normalised by $\bar{\bT}$. Hence, $\bar{\alpha}_{ij}$ is a root with 
corresponding root subgroup $\bar{\bU}_{ij}\subseteq \bar{\bG}$. As above, 
it follows that $\bar{R}= \{\bar{\alpha}_{ij} \mid 1\leq i,j\leq n,
i\neq j\}$ is the root system of $\bar{\bG}$ with respect to $\bar{\bT}$ and 
that
\[\bar{\Pi}=\{\bar{\alpha}_{i,i+1}\mid 1\leq i\leq n-1\} \mbox{ is a 
base for $\bar{R}$}.\]
On the other hand, we obtain homomorphisms of algebraic groups
$\bar{\varphi}_{ij}\colon \SL_2(k)\rightarrow \bar{\bG}$, simply by 
composing $\varphi_{ij}\colon\SL_2(k)\hookrightarrow \bG$ with the canonical
map $\bG\rightarrow \bar{\bG}$. Thus, every coroot $\alpha^\vee$ of $\bG$ 
determines a coroot $\bar{\alpha}^\vee\in \bar{Y}$. Consequently, we
have $\bar{R}^\vee=\{\bar{\alpha}_{ij}^\vee \mid 1\leq i,j\leq n,i\neq j\}$ 
and
\[\bar{\Pi}^\vee=\{\bar{\alpha}_{i,i+1}^\vee\mid 1\leq i\leq n-1\} 
\mbox{ is a base for $\bar{R}^\vee$}.\]
In particular, we obtain the same Cartan matrix $C$ of type $A_{n-1}$
as in Example~\ref{rootdatGL}. Now consider the homomorphism of algebraic
groups 
\[\bT\rightarrow (\bkm)^{n-1},\qquad h(\xi_1,\ldots,\xi_n)\mapsto
(\xi_1\xi_n^{-1},\ldots,\xi_{n-1}\xi_n^{-1}).\]
It has $\bZ$ in its kernel so there is an induced homomorphism of 
algebraic groups $\bar{\bT} \rightarrow (\bkm)^{n-1}$. The latter 
homomorphism is an isomorphism: its inverse is given by sending $(\xi_1,
\ldots,\xi_{n-1})\in(\bkm)^{n-1}$ to the image of $h(\xi_1,\ldots,\xi_{n-1},
1)\in\bT$ in $\bar{\bT}$. It follows that 
\[\{\bar{\alpha}_{i,n} \mid 1\leq i\leq n-1\} \quad \mbox{is a $\Z$-basis 
of $\bar{X}$}.\]
But then $\bar{\Pi}$ also is a $\Z$-basis of $\bar{X}$. If we consider the 
corresponding dual basis of $\bar{Y}$, then the factorisation in 
Remark~\ref{Mcorbrulu1} is given by
\[ C=\breve{A}\cdot A^{\text{tr}} \qquad \mbox{where $\breve{A}=C$ and
$A=I_{n-1}$}.\]
Thus, $\PGL_n(k)$ is semisimple and the root datum of $\PGL_n(k)$ is of 
adjoint type (see Example~\ref{expadjsc}). In particular, we see that
the root data of $\PGL_n(k)$ and $\SL_n(k)$ are not isomorphic. (In the 
former, $\Z R=X$; in the latter, $\Z R\neq X$.) So, by
Theorem~\ref{MrootdatumG1}, $\PGL_n(k)$ and $\SL_n(k)$ are not isomorphic 
as algebraic groups. (This question was raised at the end of 
\ref{subsec1wrong}.) 
\end{exmp}

Now, a key feature of the whole theory is the fact that a connected reductive 
algebraic group is uniquely determined by its root datum up to isomorphism.
This follows from a more general result, the so-called \nm{isogeny 
theorem}. As a preparation, we cite the following general results concerning 
surjective homomorphisms of algebraic groups, which will be useful at
several places below.

\begin{abs} \label{Mconnredcomp} 
Let $f\colon \bG\rightarrow \bG'$ be a surjective homomorphism of 
connected algebraic groups over $k$. Then we have the following 
\nm{preservation results}. 
\begin{itemize}
\item[(a)] $f$ maps a Borel subgroup of $\bG$ onto a Borel subgroup of
$\bG'$, and all Borel subgroups of $\bG'$ arise in this way; a similar 
statement holds for maximal tori. (See \cite[11.14]{Bor}.)
\item[(b)] $f$ maps the unipotent radical of $\bG$ onto the unipotent 
radical of $\bG'$; in particular, if $\bG$ is reductive, then so is $\bG'$.
(See \cite[14.11]{Bor}.)
\item[(c)] If $\bG$ is reductive, then $f$ maps the center of $\bG$ onto 
the center of $\bG'$. (This follows from (a) and the fact that the 
center of a reductive group is the intersection of all its maximal tori;
see \cite[11.11]{Bor}.)
\item[(d)] Assume that $\bG,\bG'$ are reductive. Let $\bT$ be a maximal 
torus of $\bG$; by (a), $\bT':=f(\bT)$ is a maxmial torus of $\bG'$. 
Then $f$ induces a surjective homomorphism $\bW(\bG,\bT) \rightarrow 
\bW(\bG',\bT')$, and this is an isomorphism if $\ker(f)$ is contained 
in the center of $\bG$. (See \cite[11.20]{Bor}.)
\end{itemize}
\end{abs}

\begin{abs} \label{Mabs23} Let $\bG$ and $\bG'$ be connected and reductive. 
Let $f \colon \bG \rightarrow \bG'$ be an \nm{isogeny}, that is, a
surjective homomorphism of algebraic groups such that $\ker(f)$ is finite.
Note that then $\ker(f)$ is automatically contained in the center of $\bG$.
Further to the properties in \ref{Mconnredcomp}, $f$ also preserves the 
root and coroot structure of $\bG$. More precisely, this works as follows, 
where we refer to \cite[\S 18.2]{Ch05}, \cite[\S 11]{MaTe}, 
\cite[\S 9.6]{Spr}, \cite[\S 1]{St2} for further details. 

Let $\bT$ be a maximal torus of $\bG$; then $\ker(f) \subseteq \bT$ and  
$\bT'=f(\bT)$ is a maximal torus of $\bG'$. Let $(X,R,Y,R^\vee)$ and $(X',
R',Y',R'^\vee)$ be the corresponding root data. The map $f$ induces a 
homomorphism $\varphi \colon X' \rightarrow X$ such that $\varphi(\lambda')=
\lambda' \circ f|_{\bT}$ for all $\lambda' \in X'$. Then it follows that 
$\varphi$ is a $p$-isogeny of root data as in Definition~\ref{MdefisogR}, 
where $p$ is the \nm{characteristic exponent} of~$k$. (Recall that the 
characteristic exponent of $k$ is $1$ in case $\mbox{char}(k)=0$ and is 
equal to $\mbox{char}(k)$ otherwise.) The numbers $\{q_\alpha\mid \alpha
\in R\}$ and the bijection $R\rightarrow R'$, $\alpha \mapsto 
\alpha^\dagger$, in (I2) come about as follows.

Let $\alpha \in R$ and consider the corresponding root subgroup $\bU_\alpha 
\subseteq \bG$; see \ref{subsec1roots}. Then $f(\bU_\alpha)$ is a 
one-dimensional closed connected unipotent subgroup of $\bG'$ normalised 
by $\bT'$. Hence, there is a well-defined $\alpha^\dagger\in R'$ such that 
$f(\bU_\alpha)$ equals the root subgroup $\bU_{\alpha^\dagger}'$ in $\bG'$. 
Let $u_\alpha\colon\bkp\rightarrow \bU_\alpha$ and $u_{\alpha^\dagger}'
\colon \bkp\rightarrow \bU_{\alpha^\dagger}'$ be the corresponding 
isomorphisms. Then the map $f\colon \bU_\alpha \rightarrow 
\bU_{\alpha^\dagger}'$ has the following form. There is some $c_\alpha \in 
k^\times$ such that 
\[ f(u_\alpha(\xi))=u_{\alpha^\dagger}'(c_\alpha \xi^{q_\alpha})\qquad
\mbox{for all $\xi \in \bkp$}.\]
In this situation, the numbers $\{q_\alpha\}$ will also be called the 
\nm{root exponents} of~$f$. The above discussion shows that an isogeny of
connected reductive groups induces a $p$-isogeny of root data. Conversely, 
we have the following fundamental result.
\end{abs}

\begin{thm}[Isogeny Theorem]\label{thmiso} Let $\bG$ and $\bG'$ be connected
reductive algebraic groups over $k$, let $\bT\subseteq \bG$ and $\bT'
\subseteq \bG'$ be maximal tori, and let $\varphi \colon X(\bT')\rightarrow
X(\bT)$ be a $p$-isogeny of their root data (see Definition~\ref{MdefisogR}),
where $p$ is the characteristic exponent of $k$. Then there exists an 
isogeny $f \colon \bG \rightarrow \bG'$ which maps $\bT$ onto $\bT'$ and 
induces $\varphi$. It $f'\colon \bG\rightarrow \bG'$ is another isogeny
with these properties, then there exists some $t\in \bT$ such that
$f'(g)=f(tgt^{-1})$ for all $g\in\bG$. 
\end{thm} 

See \cite{St2} for a recent, quite short proof of this fundamental result 
which, for semisimple groups, is one of the main results of the 
{\em S\'eminaire Chevalley}; see \cite[\S 18.2]{Ch05}. As a first 
consequence, we have:

\begin{cor}[Isomorphism Theorem] \label{Mremiso1} In the setting of
Theorem~\ref{thmiso}, assume that $\varphi$ is bijective. Then $f \colon 
\bG\rightarrow \bG'$ is an isomorphism of algebraic groups.
\end{cor}

\begin{proof} We use the notation in \ref{Mabs23}. First note that, since 
$\varphi$ is bijective, we must have $q_\alpha=1$ for all $\alpha\in R$. 
Hence, the inverse map $\varphi^{-1}\colon X\rightarrow X'$ also defines 
an isogeny of root data. By Theorem~\ref{thmiso} there exist isogenies 
$f \colon \bG\rightarrow \bG'$ and $g\colon \bG'\rightarrow \bG$ 
corresponding to $\varphi$ and $\varphi^{-1}$. Then $g \circ f$ induces the 
identity isogeny of the root datum of $\bG$ and hence equals the inner 
automorphism $\iota_t$ for some $t\in \bT$. Thus $g'\circ f=\id_{\bG}$ 
with $g'=\iota_t^{-1} \circ g$, and then $f\circ g'\circ f=f$ and 
$f\circ g'=\id_{\bG'}$ because $f$ is surjective. Hence $f$ is an 
isomorphism with $g'$ as its inverse. 
\end{proof}

The general theory is completed by the following existence result.

\begin{thm}[Existence Theorem] \label{thmexi}  Let $\cR=(X,R,Y, R^\vee)$ 
be a root datum. Then there exists a connected reductive algebraic
group $\bG$ over $k$ and a maximal torus $\bT\subseteq \bG$ such that 
$\cR$ is isomorphic to the root datum of $\bG$ relative to~$\bT$. 
\end{thm}

For semisimple groups, this is originally due to Chevalley; see \cite{Chev}
and the comments in \cite[\S 24]{Ch05}. See \cite{Ca1}, 
\cite[\S 5, Theorem~6]{St} where this is explained in detail, following 
and extending Chevalley's original approach. The general case can be 
reduced to this one; see \cite[\S 10.1]{Spr} and \cite[Expos\'e~XXV]{sga33}.
Only recently, Lusztig \cite{Lu09c} found a new approach to the general 
case based on the theory of ``canonical bases'' of quantum groups. 

\begin{exmp} \label{Mclassif1} Let us see what the above results mean
in the simplest non-trivial case where $\cR=(X,R,Y,R^\vee)$ is a root 
datum of Cartan type $A_1$. Let $\bG$ be a corresponding connected reductive
algebraic group over $k$. Now, since $C=(2)$ is the Cartan matrix in this
case, $\cR$ is determined by an equation 
\[2=\sum_{1\leq i \leq d} \breve{a}_ia_i \qquad \mbox{where $d=
\mbox{rank}\,X=\mbox{rank}\,Y$ and $a_i,\breve{a}_i \in \Z$ for all $i$};\] 
see Remark~\ref{Mcorbrulu1}. Up to isomorphism (where isomorphisms are 
determined by an invertible matrix $P$ over $\Z$ as in
Lemma~\ref{Misorootdat}), there are three possible cases:
\begin{itemize}
\item[(1)] $(a_1,\ldots,a_d)=(2,0,\ldots,0)$ and $(\breve{a}_1,\ldots,
\breve{a}_d)=(1,0,\ldots,0)$, in which case $\bG\cong \SL_2(k) \times
(\bkm)^{d-1}$.
\item[(2)] $(a_1,\ldots,a_d)=(1,0,\ldots,0)$ and $(\breve{a}_1,\ldots,
\breve{a}_d)=(2,0,\ldots,0)$, in which case $\bG\cong \PGL_2(k) \times
(\bkm)^{d-1}$.
\item[(3)] $d\geq 2$, $(a_1,\ldots,a_d)=(1,1,0,\ldots,0)$ and $(\breve{a}_1,
\ldots, \breve{a}_d)=(1,1,0,\ldots,0)$, in which case $\bG\cong \GL_2(k) 
\times (\bkm)^{d-2}$.
\end{itemize}
This is contained in \cite[2.2]{St98}; we leave it as an exercise to the
reader. In particular, for $d=1$ (that is, $\bG$ semisimple), we have 
either $\bG\cong \SL_2(k)$ or $\bG\cong \PGL_2(k)$.
\end{exmp}

Besides its fundamental importance for the classification of connected 
reductive algebraic groups, the Isogeny Theorem is an indispensible tool 
for showing the existence of homomorphisms with prescribed properties. Here 
are the first examples. 

\begin{exmp} \label{Mopposite} Let $\bG$ be a connected reductive algebraic
group over $k$ and $\bT$ be a maximal torus of $\bG$. Let $(X,R,Y,R^\vee)$ 
be the corresponding root datum. Then there exists an 
automorphism of algebraic groups $\tau\colon \bG\rightarrow \bG$ such that 
\[\tau(t)=t^{-1} \quad    (t\in \bT) \qquad\mbox{and}\qquad 
\tau(\bU_\alpha)=\bU_{-\alpha} \quad (\alpha\in R).\]
Indeed, $\varphi \colon X\rightarrow X$, $\lambda \mapsto 
-\lambda$, certainly is a $p$-isogeny, where $q_\alpha=1$ for all 
$\alpha\in R$. Hence, since $\varphi$ is bijective, 
Corollary~\ref{Mremiso1} shows that there exists an automorphism 
$\tau\colon \bG\rightarrow \bG$ such that $\tau(\bT)=\bT$ and such 
that $\varphi$ is the map induced by $\tau$ on $X$. Now, as discussed in
\ref{subsec17}, there is a natural bijection between group homomorphisms 
of $X$ into itself and algebraic homomorphisms of $\bT$ into itself. Under 
this bijection, $\varphi$ clearly corresponds to the map $t\mapsto t^{-1}$ 
on $\bT$. Hence, $\tau$ has the required properties.
\end{exmp}

\begin{exmp} \label{canfrob2a} Let $p$ be a prime number and $\bG$ be a
connected reductive algebraic group over $k=\overline{\F}_p$. Let $\bT$ 
be a maximal torus of $\bG$ and $(X,R,Y, R^\vee)$ be the corresponding 
root datum. Then $\varphi\colon X\rightarrow X$, $\lambda \mapsto p\lambda$,
certainly is a $p$-isogeny of root data, where $q_\alpha=p$ for all $\alpha
\in R$. Hence, by Theorem~\ref{thmiso}, there exists an isogeny $F_p\colon 
\bG \rightarrow \bG$ such that $F_p(\bT)=\bT$ and such that $F_p$ induces
$\varphi$ on $X$. Arguing as in the previous example, it follows that
\[F_p(\bU_\alpha)=\bU_\alpha \quad (\alpha \in R) \qquad \mbox{and} \qquad 
F_p(t)=t^p \quad (t \in \bT).\]
For example, $F_p \colon \GL_n(k)\rightarrow \GL_n(k)$, $(a_{ij}) 
\mapsto (a_{ij}^p)$, is an isogeny satisfying the above conditions. 

We shall see in Section~\ref{sec:steinberg} that the fixed point set of 
$\bG$ under $F_p$ is a finite group. More generally, we shall consider 
isogenies $F\colon \bG \rightarrow \bG$ such that $F^d=F_p^m$ for some 
$d,m\geq 1$. The finite groups arising as fixed point sets of connected 
reductive groups under such isogenies are the {\em finite groups of
Lie type}; see Definition~\ref{mydeffr}.
\end{exmp}
 
\begin{exmp} \label{MdirprodG2} Let $\cR_i=(X_i,R_i, Y_i,R_i^\vee)$ (for 
$i=1,\ldots,n$) be root data. Let $\cR=(X,R,Y,R^\vee)$ be the direct sum
of these root data; see Example~\ref{MdirprodR}. For $i=1,\ldots,n$, let
$\bG_i$ be a connected reductive algebraic group with root datum 
isomorphic to $\cR_i$ (relative to a maximal torus $\bT_i\subseteq\bG_i$). 
Then, using Corollary~\ref{Mremiso1}, one easily sees that the direct 
product $\bG:=\bG_1\times \ldots \times \bG_n$ has root datum isomorphic 
to $\cR$ (relative to the maximal torus $\bT:=\bT_1\times\ldots \times 
\bT_n$ of $\bG$). 
\end{exmp}

\begin{exmp} \label{isounitary} Let $\bG=\GL_n(k)$, with root datum
$\cR=(X,R,Y,R^\vee)$ as in Example~\ref{rootdatGL}. It is given by a 
factorisation $C=\breve{A} \cdot A^{\text{tr}}$ where $C=(c_{ij})_{1\leq i,
j\leq n-1}$ is the Cartan matrix of type $A_{n-1}$ and $A=\breve{A}$ is a 
certain matrix of size $(n-1)\times n$. Then, by the procedure in 
\ref{Mdefmatrixisog}, we obtain an isogeny $\varphi\colon X \rightarrow X$ 
via the pair of matrices $(P, P^\circ)=(J_n,J_{n-1})$ where, for any 
$m\geq 1$, we set 
\[J_m:=\left(\begin{array}{cccc} 0 & \cdots & 0 & 1 \\ \vdots &\dddots 
& \dddots & 0 \\ 0 & 1 &\dddots & \vdots \\ 1 & 0 & \cdots & 0 \end{array}
\right) \in M_m(k).\]
Then $\varphi$ has order~$2$. So there is a corresponding automorphism
of algebraic groups $\gamma \colon \GL_n(k) \rightarrow \GL_n(k)$ which 
maps the maximal torus $\bT$ into itself and induces $\varphi$ on $X$. 
Concretely, the map given by
\[\gamma\colon\GL_n(k) \rightarrow \GL_n(k), \qquad g \mapsto 
J_n(g^{\text{tr}})^{-1}J_n, \]
is an automorphism with this property.
\end{exmp}

\begin{rem} \label{Mcentriso} Let $f \colon \bG\rightarrow\bG'$ be an 
isogeny of connected reductive algebraic groups over $k$. In the
setting of \ref{Mabs23}, let $\{q_\alpha\mid \alpha \in R\}$ be the root 
exponents of~$f$. Following \cite[9.6.3]{Spr}, we say that $f$ is a 
\nm{central isogeny} if $q_\alpha=1$ for all $\alpha\in R$. The 
terminology is justified as follows. Consider the corresponding 
homomorphism of Lie algebras $d_1 f\colon L(\bG) \rightarrow L(\bG')$. 
Then, by \cite[22.4]{Bor}, $f$ is a central isogeny if and only if the 
kernel of $d_1f$ is contained in the center of $L(\bG)$. For example,
the isogeny in Example~\ref{Mopposite} is central while that in 
Example~\ref{canfrob2a} is not.
\end{rem}

There are extensions of the Isogeny Theorem to the case where we
consider homomorphisms whose kernel is still central but not finite:
We shall only formulate the following version here. (This will be needed, 
for example, in Section~\ref{sec:regemb}.) 

\begin{abs} \label{Mdefcentralhom} Let $\bG$, $\bG'$ be connected 
reductive algebraic groups over $k$, and $f\colon \bG \rightarrow \bG'$ be 
a homomorphism of algebraic groups. 

(a) Following \cite[Chap.~I, 3.A]{Bo3}, we say that $f$ is 
an \nm{isotypy} if $\ker(f)\subseteq \bZ(\bG)$ and $\bGder'\subseteq 
f(\bG)$. If this is the case, then we have $\bG'= f(\bG).\bZ(\bG')$, 
$f(\bGder)=\bGder'$ and $f$ restricts to an isogeny $\bGder \rightarrow 
\bGder'$.

(b) Now let $\bT\subseteq \bG$ and $\bT' \subseteq \bG'$ be maximal tori 
such that $f(\bT)\subseteq \bT'$. Then $f$ induces a group homomorphism 
$\varphi \colon X(\bT')\rightarrow X(\bT)$, $\lambda \mapsto\lambda \circ 
f|_{\bT}$. In analogy to Remark~\ref{Mcentriso}, we say that $f$ is a 
\nm{central isotypy} if $\varphi$ is a homomorphism of root data as in 
\ref{Mhomrootdata}. (Note that, as pointed out in the remarks following 
\cite[II, Prop.~1.14]{Jan}, a central isotopy is automatically an isotypy.)
\end{abs}

\begin{thm}[Extended Isogeny Theorem; see \protect{\cite[II, 1.14, 1.15]{Jan},
\cite[\S 5]{St2}}] \label{thmiso1} Let $\bG$ and $\bG'$ be connected 
reductive algebraic groups over $k$, let $\bT\subseteq \bG$ and $\bT' 
\subseteq \bG'$ be maximal tori, and let $\varphi\colon X(\bT') \rightarrow 
X(\bT)$ be a homomorphism of their root data (see \ref{Mhomrootdata}). Then 
there exists a central isotypy $f \colon \bG \rightarrow \bG'$ such 
that $f(\bT)\subseteq \bT'$ and $f$ induces $\varphi$. 
Furthermore, the following hold.
\begin{itemize}
\item[(a)] If $f'\colon \bG\rightarrow \bG'$ is another central isotypy 
inducing $\varphi$, then there exists some $t\in \bT$ such that 
$f'(g)=f(tgt^{-1})$ for all $g\in\bG$.
\item[(b)] If $f|_{\bT}\colon \bT\rightarrow \bT'$ is an 
isomorphism, then so is $f\colon \bG\rightarrow \bG'$. 
\end{itemize}
\end{thm}

\begin{proof} Let $\Pi$ be a base of the root system $R\subseteq X(\bT)$. 
For $\alpha\in \Pi$, consider the subgroup $\bG_\alpha=\langle \bT,
\bU_\alpha,\bU_{-\alpha} \rangle\subseteq \bG$ defined in 
\ref{subsec1roots}. Then $\bG_\alpha\cap \bG_\beta=\bT$ for $\alpha\neq 
\beta$ in $\Pi$. As in \cite[II, \S 1.13]{Jan}, one sees that there 
exists a map 
\[f\colon \bigcup_{\alpha \in \Pi} \bG_\alpha\rightarrow \bG'\]
which is a homomorphism on each $\bG_\alpha$ and is such that $f$ maps $\bT$
into $\bT'$ and induces $\alpha$. Now, $\bU_\alpha$ and $\bU_{-\beta}$ 
certainly commute with each other for all $\alpha\neq \beta$ in $\Pi$ (by 
Chevalley's commutator relations; see \cite[11.8]{MaTe}). Hence, by 
\cite[Theorem~5.3]{St2}, $f$ extends to a homomorphism of algebraic groups
from $\bG$ to $\bG'$. The uniqueness statement in (a) is proved as in the
case of Theorem~\ref{thmiso}; see \cite[\S 3]{St2}. Finally, (b) holds by 
\cite[II, \S 1.15]{Jan}.
\end{proof}

\section{Frobenius maps and Steinberg maps} \label{sec:steinberg}

We assume in this section that $k=\overline{\F}_p$ is an algebraic closure
of the finite field with $p$ elements, where $p$ is a prime number. We 
consider a particular class of isogenies in this context, the so-called 
{\em Steinberg maps}. This will be treated in some detail, where one aim 
is to work out explicitly some useful characterisations of Steinberg maps 
in terms of isogenies of root data. In particular, in 
Proposition~\ref{Mgenfrob}, we recover the set-up in Example~\ref{canfrob2a}. 
We also establish a precise characterisation of Frobenius maps among
all Steinberg maps; see Proposition~\ref{Mgenfrob3}. 

\begin{defn} \label{Mabs31} Let $\bX$ be an affine variety over $k$. Let 
$q$ be a power of $p$ and $\F_q \subseteq k$ be the finite field with $q$ 
elements. We say that $\bX$ has an $\F_q$-\nm{rational structure} (or that 
$\bX$ is \nm{defined over $\F_q$}) if there exists some $n\geq 1$ and an 
isomorphism of affine varieties $\iota\colon \bX\rightarrow \bX'$ where 
$\bX'\subseteq \bkn$ is Zariski closed and stable under the 
\nm{standard Frobenius map}
\[F_q\colon \bkn\rightarrow \bkn, \qquad (\xi_1,\ldots,\xi_n)\mapsto 
(\xi_1^q,\ldots,\xi_n^q).\]
In this case, there is a unique morphism of affine varieties 
$F\colon \bX\rightarrow \bX$ such that $\iota \circ F=F_q\circ\iota$; it
is called the \nm{Frobenius map} corresponding to the $\F_q$-rational
structure of $\bX$. Note that $F_q$ is a bijective morphism whose fixed point
set is $\F_q^n$. Consequently, $F$ is a bijective morphism such that 
\[|\bX^F|<\infty \qquad\mbox{where}\qquad \bX^F:=\{x\in\bX\mid F(x)=x\}.\]
\end{defn}

\begin{exmp} \label{MconcreteF} A Zariski closed subset $\bX \subseteq \bkn$ 
is called $\F_q$-{\em closed} if $\bX$ is defined by a set of polynomials in 
$\F_q[T_1,\ldots,T_n]$. If this holds, then $\bX$ is stable under $F_q$ 
and so $\bX$ has an $\F_q$-rational structure, as defined above; the fixed 
point set $\bX^{F_q}$ consists precisely of those $x\in \bX$ which have 
all their coordinates in $\F_q$. Conversely, if $F_q(\bX)\subseteq \bX$,
then $\bX$ is $\F_q$-closed. (The proof uses the fact that $k\supseteq 
\F_q$ is a separable field extension; see \cite[4.1.6]{mybook}, 
\cite[AG.14.4]{Bor}. In general, the discussion of rational structures 
is much more complicated.)
\end{exmp}

\begin{rem} \label{frobclosed} Let $\bX$ be an affine variety over $k$
and assume that $\bX$ is defined over $\F_q$, with Frobenius map $F\colon 
\bX \rightarrow \bX$. Here are some basic properties of $F$. First note 
that $F^2,F^3,\ldots$ are also Frobenius maps. Furthermore, for any 
$x\in \bX$, we have $F^m(x)=x$ for some $m\geq 1$. Hence,
\[\bX=\bigcup_{m\geq 1} \bX^{F^m} \quad \mbox{where}\quad
|\bX^{F^m}|<\infty \quad \mbox{ for all $m\geq 1$}.\]
(Note that every element of $k$ lies in a finite subfield of $k$.) 
Finally, it is also clear that, if $\bX'\subseteq \bX$ is a closed subset 
such that $F(\bX') \subseteq \bX'$, then $\bX'$ is defined over $\F_q$, 
with Frobenius map given by the restriction of $F$ to $\bX'$. 
\end{rem} 

\begin{rem} \label{MintrinsF} Let $\bX$ be an affine variety over $k$
and let $A$ be the algebra of regular functions on $\bX$. There is an
intrinsic characterisation of Frobenius maps in terms of $A$, as follows.
A morphism $F\colon  \bX\rightarrow \bX$ is the Frobenius map corresponding
to an $\F_q$-rational structure of $\bX$ if and only if the following
two conditions hold for the associated algebra homomorphism $F^*
\colon A\rightarrow A$:
\begin{itemize}
\item[(a)] $F^*$ is injective and $F^*(A)=\{a^q\mid a\in A\}$. 
\item[(b)] For each $a\in A$, there exists some $e\geq 1$ such that
$(F^*)^e(a)=a^{q^e}$.
\end{itemize}
One easily sees that (a) and (b) hold for the standard Frobenius map
$F_q\colon \bkn\rightarrow \bkn$. This implies that (a) and (b) hold for
any $F_q$-stable closed subset $\bX\subseteq \bkn$ as in 
Example~\ref{MconcreteF}. The converse requires a bit more work; see 
\cite[\S 4.1]{mybook} or \cite[Chap.~II]{Sr} for details. 
One advantage of this characterisation of Frobenius maps is, for example, 
that it provides an easy proof of the following statement:
\begin{itemize}
\item[(c)] If $F$ is a Frobenius map (with respect to $\F_q$, as above) and 
$\gamma\colon \bX \rightarrow \bX$ is an automorphism of affine varieties 
of finite order which commutes with $F$, then $\gamma \circ F$ also is a 
Frobenius map on $\bX$ (with respect to $\F_q$, same $q$).
\end{itemize}
(See, for example, \cite[Exercise~4.4]{mybook}.) The above characterisation 
is also equivalent to the definition of an ``abstract affine algebraic 
$(\F_q,k)$-set'' in \cite{cart}.
\end{rem}

In the sequel, $\bG$ will always be a linear algebraic group over $k=
\overline{\F}_p$. 

\begin{abs} \label{Mfrobgroups} Assume that, as an affine variety, $\bG$ 
is defined over $\F_q$ with corresponding Frobenius map $F$. Then we say 
that $\bG$ (as an algebraic group) is defined over $\F_q$ if $F$ is a 
group homomorphism. In this case, the set of fixed points $\bG^F$ is a 
finite group. There is a more concrete description, similar to 
Definition~\ref{Mabs31}. We have the standard 
Frobenius map 
\[F_q\colon \GL_n(k) \rightarrow \GL_n(k), \qquad (a_{ij}) \mapsto
(a_{ij}^q).\]
Then $\bG$ is defined over $\F_q$ if and only if there is a homomorphism
\[ \iota\colon \bG\rightarrow \GL_n(k)\qquad \mbox{(for some $n\geq 1$)}\]
of algebraic groups such that $\iota$ is an isomorphism onto its image and 
the image is stable under $F_q$; in this case, the corresponding Frobenius 
map $F\colon \bG \rightarrow \bG$ is defined by the condition that $\iota
\circ F=F_q\circ \iota$. (See \cite[4.1.11]{mybook} for further details.) 
In particular, if $\bG\subseteq \GL_n(k)$ is a closed subgroup defined
by a collection of polynomials with coefficients in $\F_q$, then $F_q$
restricts to a Frobenius map on $\bG$.
\end{abs}

\begin{exmp} \label{Msplittorus} Let $\bT\subseteq \bG$ be an abelian 
subgroup consisting of semisimple elements (e.g., a torus). We claim that 
there always exists {\it some} Frobenius map $F\colon \bG\rightarrow \bG$ 
(with respect to an $\F_q$-rational structure on $\bG$) such that $\bT$ is 
$F$-stable and $F(t)=t^q$ for all $t\in\bT$. 

Indeed, we can realise $\bG$ as a closed subgroup $\bG \subseteq \GL_n(k)$ 
for some $n\geq 1$. Since $\bT$ consists of commuting semisimple elements, 
we can assume that then $\bT$ consists of diagonal matrices. Now, the 
defining ideal of $\bG$ (as an algebraic subset of $\GL_n(k)$) is generated 
by a finite set of polynomials with coefficients in $k$. So there is some 
$q=p^m$ ($m\geq 1$) such that all these coefficients lie in $\F_q$. Then 
$\bG$ is stable under the standard Frobenius map $F_q$ on $\GL_n(k)$. So 
$F_q$ restricts to a Frobenius map $F\colon \bG\rightarrow \bG$. Since
any $t\in \bT$ is a diagonal matrix, we have $F(t)=t^q$. 
\end{exmp}

\begin{defn} \label{mydeffr} Let $F\colon \bG\rightarrow\bG$ be an 
endomorphism of algebraic groups. Then $F$ is called a \nm{Steinberg map} 
if some power of $F$ is the Frobenius map with respect to an 
$\F_q$-rational structure on 
$\bG$, for some power $q$ of $p$. Note that, in this case, $F$ is a 
bijective homomorphism of algebraic groups and $\bG^F$ is a finite group. 
If $\bG$ is connected and reductive, then $\bG^F$ will be called a 
\nm{finite group of Lie type} or a \nm{finite reductive group}. 
\end{defn}

The following result is the key tool to pass from 
properties of $\bG$ to properties of the finite group $\bG^F$; we shall see 
numerous applications in what follows.

\begin{thm}[Lang \protect{\cite{Lang}}, Steinberg 
\protect{\cite[10.1]{St68}}] \label{langst} Assume that $\bG$ is connected 
and let $F\colon \bG \rightarrow \bG$ be a Steinberg map (or, more 
generally, any endomorphism such that $|\bG^F|<\infty$). Then the map
$\cL\colon \bG \rightarrow \bG$, $g\mapsto g^{-1}F(g)$, is surjective. 
\end{thm}

\begin{proof} If $F$ is a Steinberg map (and this is the case that we are
mainly interested in), then \cite{Mu03} gives a quick proof, as follows. 
The group $\bG$ acts on itself (on the right) where $g\in \bG$ 
sends $x\in \bG$ to $g^{-1}xF(g)$. Any action of an algebraic group on an 
affine variety has a closed orbit; see \cite[2.5.2]{mybook}. Let $\Omega$ 
be such a closed orbit and let $x\in \Omega$. Since $\bG$ is connected, 
it will be sufficient to show that $\dim \bG=\dim \Omega$ (because then 
$\bG=\Omega$ and so $1\in \Omega$). We have $\dim \Omega= \dim \bG-
\dim \mbox{Stab}_{\bG}(x)$ (see \cite[2.5.3]{mybook}), so it will be 
sufficient to show that $\mbox{Stab}_{\bG}(x)$ is finite. Now, an element 
$g\in\bG$ belongs to this stabiliser if and only if $g^{-1}xF(g)=x$, which 
is equivalent to $f(g)=g$, where $f(g):=xF(g)x^{-1}$. Let $m\geq 1$ be such 
that $F^m$ is a Frobenius map and $F^m(x)=x$ (see
Remark~\ref{frobclosed}). Let $r\geq 1$ be the order of the element 
$xF(x)F^2(x)\cdots F^{m-1}(x)\in \bG$. Then $f^{mr}(g) =F^{mr}(g)$ for 
all $g\in \bG$. So $f^{mr}(g)=g$ has only finitely many solutions in $\bG$, 
hence $f(g)=g$ has only finitely many solutions in $\bG$.
\end{proof}
 
For various parts of the subsequent discussion it would be 
sufficient to work with endomorphisms of $\bG$ whose fixed point set is
finite. However, we will just formulate everything in terms of Steinberg
maps, as defined above. We note that the discussion in \cite[11.6]{St68} 
in combination with Proposition~\ref{Mgenfrob} below implies that an 
endomorphism of a {\em simple} algebraic group with a finite fixed point 
set is automatically a Steinberg map; see also Example~\ref{Mdiffexmp} below.

Here is the prototype of an application of the above theorem. 

\begin{prop} \label{Mfrobconj0} Assume that $\bG$ is connected  and let 
$F\colon \bG \rightarrow \bG$ be a Steinberg map. Let $X$ be an abstract 
set on which $\bG$ acts transitively; let $F'\colon X\rightarrow X$ be a 
map such that $F'(g.x)=F(g).F'(x)$ for all $g\in\bG$ and $x\in X$. Then 
there exists some $x_0\in X$ such that $F'(x_0)=x_0$. 
\end{prop}

\begin{proof} Take any $x\in X$. Since $\bG$ acts transitively, we have 
$F'(x)=g^{-1}.x$ for some $g\in \bG$. By Theorem~\ref{langst}, we can 
write $g=h^{-1}F(h)$. Then one immediately checks that $x_0:=h.x$ is 
fixed by~$F'$.  
\end{proof}

\begin{exmp} \label{Mfrobconj} Assume that $\bG$ is connected  and let 
$F\colon \bG \rightarrow \bG$ be a Steinberg map. Let $C$ be a conjugacy 
class of $\bG$ such that $F(C)=C$. Then $\bG$ acts transitively on $C$ by 
conjugation; let $F'$ be the restriction of $F$ to $C$. Applying 
Proposition~\ref{Mfrobconj0} yields that there exists an element $x\in C$ 
such that $F(x)=x$.

Similarly, there exists a pair $(\bT,\bB)$ where $\bT$ is an $F$-stable 
maximal torus of $\bG$ and $\bB$ is an $F$-stable Borel subgroup such that
$\bT\subseteq \bB$. (Just note that, by Theorem~\ref{MrootdatumG1}, all 
these pairs are conjugate in $\bG$ and, by \ref{Mconnredcomp}(a), $F$ 
preserves the set of all these pairs.) An $F$-stable maximal torus of
$\bG$ which is contained in an $F$-stable Borel subgroup of $\bG$ will
be called a \nm{maximally split} torus. 
\end{exmp}

\begin{prop}[Cf.\ \protect{\cite[10.10]{St68}}] \label{canfrob11} Let 
$\bG$ be connected reductive and $F\colon \bG \rightarrow \bG$ be a 
Steinberg map. Then all maximally split tori of $\bG$ are $\bG^F$-conjugate.
More precisely, all pairs $(\bT,\bB)$ consisting of an $F$-stable
maximal torus $\bT$ and an $F$-stable Borel subgroup $\bB$ such that 
$\bT\subseteq\bB$ are conjugate in $\bG^F$.
\end{prop}

\begin{proof} Let $(\bT, \bB)$ and $(\bT_1,\bB_1)$ be two pairs as above. 
By Theorem~\ref{MrootdatumG1}, there exists some $x\in \bG$ such that 
$x\bB x^{-1}=\bB_1$ and $x\bT x^{-1}= \bT_1$. Since $\bB, \bB_1$ are 
$F$-stable, this implies that $x^{-1}F(x)\in N_{\bG}(\bB)=\bB$, where the 
last equality holds by \cite[11.16]{Bor} or \cite[6.4.9]{Spr}. Similarly, 
since $\bT,\bT_1$ are $F$-stable, we have $x^{-1}F(x)\in N_{\bG}(\bT)$. 
Hence, $x^{-1}F(x)\in \bB\cap N_{\bG}(\bT)=\bT$ (see 
Theorem~\ref{MrootdatumG}). Applying Theorem~\ref{langst} to the 
restriction of $F$ to $\bT$, we obtain an element $t\in \bT$ such that 
$x^{-1}F(x)=t^{-1}F(t)$. Then $g:=xt^{-1} \in \bG^F$ and $g$ simultaneously 
conjugates $\bB$ to $\bB_1$ and $\bT$ to $\bT_1$.
\end{proof}

The following result deals with a subtletly concerning Steinberg maps:
A surjective homomorphism of algebraic groups will not necessarily
induce a surjective map on the level of the fixed point sets unter 
Steinberg maps.  But one can say precisely what happens in this situation:

\begin{prop}[Cf.\ \protect{\cite[4.5]{St68}}] 
\label{Msameorder} Let $f\colon \bG\rightarrow \bG'$ be a surjective 
homomorphism of connected algebraic groups such that $\bK:=\ker(f)$ is
contained in the center of $\bG$. Let $F\colon \bG \rightarrow \bG$ and 
$F' \colon \bG' \rightarrow \bG'$ be Steinberg maps such that $F'\circ f=
f\circ F$. We denote $G=\bG^F$ and $G'=\bG'^{F'}$. Then the following hold. 
\begin{itemize}
\item[(a)] Let $\cL\colon \bG \rightarrow \bG$, $g\mapsto g^{-1} 
F(g)$. Then $\cL(\bK)$ is a normal subgroup of $\bK$.
\item[(b)] $f(G) \subseteq G'$ is a normal subgroup and
$G'/f(G) \cong \bK/\cL(\bK)$. In particular, if $\bK$ is connected,
then $\cL(\bK)=\bK$ and $f(G)=G'$.
\item[(c)] If $\bK$ is finite (that is, $f$ is an isogeny), then $|G|=|G'|$.
\end{itemize}
\end{prop}

\begin{proof} First note that $\bK$ is $F$-stable. Since $\bK$ is 
contained in the center of $\bG$, this implies (a).
Now Steinberg \cite[4.5]{St68} states a general result 
(about arbitrary groups) which shows that $f\colon \bG \rightarrow \bG'$
induces a long exact sequence 
\[\{1\} \longrightarrow \bK^F \longrightarrow G
\stackrel{f}{\longrightarrow} G' \stackrel{\delta}{\longrightarrow}
(\cL(\bG)\cap \bK)/\cL(\bK) \longrightarrow \{1\},\]
where $\delta$ is given as follows. Let $g'\in G'$ and choose $g\in 
\bG$ such that $f(g)=g'$. Then $g^{-1}F(g)\in \bK$ and $\delta(g')$ is the 
image of $g^{-1}F(g)$ in $(\cL(\bG)\cap \bK)/\cL(\bK)$. 

Since $\bG$ is connected, we have $\cL(\bG)=\bG$ by Theorem~\ref{langst};
this yields (b). Now $\ker(\cL|_\bK)=\{z\in \bK\mid z^{-1}F(z)=1\}=\bK^F=
\ker (f|_G)$. So, if $\bK$ is finite, then $|\bK/\cL(\bK)|=|\ker(f|_G)|$ 
and, hence, $|G|=|f(G)| |\bK/\cL(\bK)|$. But, by (b), $\bK/\cL(\bK)$ and 
$G'/f(G)$ have the same order and so $|G|=|G'|$, that is, (c) holds.
\end{proof}

\begin{lem}[Cf.\ \protect{\cite[10.9]{St68}}] \label{canfrob0a} Assume that
$\bG$ is connected and let $F\colon \bG \rightarrow \bG$ be a Steinberg map.
Let $y\in \bG$ and define $F'\colon \bG\rightarrow \bG$ by $F'(g)=
yF(g)y^{-1}$ for all $g\in \bG$. Then $F'$ is a Steinberg map and 
we have $\bG^{F'} \cong \bG^F$. 

Furthermore, if $F$ is a Frobenius map corresponding
to an $\F_q$-rational structure, then so is $F'$ (with the same
$q$). 
\end{lem}

\begin{proof} Since $\bG$ is connected, Theorem~\ref{langst} shows that
we can write $y=x^{-1}F(x)$ for some $x\in \bG$. Then $F'(g)=x^{-1}F(xg
x^{-1})x$ for all $g\in \bG$. Thus, we have $F'=\iota_x^{-1} \circ F \circ 
\iota_x$ where $\iota_x$ denotes the inner automorphism of $\bG$ defined 
by~$x$. This formula shows that $F'(g)=g$ if and only if $xgx^{-1}\in 
\bG^F$. Hence, conjugation with $x$ defines a group isomorphism $\bG^{F'}
\cong \bG^F$.

Now we show that $F'$ is a Steinberg map. For $m\geq 1$, we have $F'^m(g)
=x^{-1}F^m(xgx^{-1})x$ for all $g\in \bG$. By Remark~\ref{frobclosed} (and 
the definition of Steinberg maps), there exists some $m\geq 1$ such that 
$F^m(x)=x$. For this $m$, we have $F'^m(g)= F^m(g)$ for all $g\in \bG$. 
Thus, $F'$ is a Steinberg map. 

Finally, assume that $F$ is a Frobenius map. We use the characterisation in
Remark~\ref{MintrinsF} to show that $F'$ is a Frobenius map. Since
$F'$ is the conjugate of $F$ by an automorphism of $\bG$, we have that
$F'^*$ is the conjugate of $F^*$ by an algebra automorphism of $A$. 
This implies that $F'^*$ is injective and $F'^*(A)=\{a^q\mid a\in A\}$. 
On the other hand, we have $F'^m=F^m$. So, if $a\in A$ and $e\geq 1$ are 
such that $(F^*)^e(a)=a^{q^e}$, then $(F'^*)^{em}(a)=(F^*)^{em}(a)=
a^{q^{em}}$, as required.
\end{proof}

\begin{lem} \label{canfrob12} Assume that $\bG$ is connected reductive. 
Let $F\colon \bG\rightarrow \bG$ be a Steinberg map and $\bT$ be an
$F$-stable maximal torus of $\bG$. Let $F'\colon \bG \rightarrow \bG$ be 
a further isogeny of $\bG$ such that $F'(\bT)=\bT$. If $F$ and $F'$ induce 
the same map on $X(\bT)$, then there exists some $y\in \bT$ such that
$F'(g)=yF(g)y^{-1}$ for all $g\in \bG$. In particular, the conclusions of
Lemma~\ref{canfrob0a} apply to $F'$. 
\end{lem}

\begin{proof} Since $F$, $F'$ induce the same map on $X(\bT)$, 
Theorem~\ref{thmiso} implies that there exists some $y\in \bT$ such 
that $F'(g)=yF(g)y^{-1}$ for all $g\in \bG$. 
\end{proof}

\begin{exmp} \label{canfrob2} 
Assume that $\bG$ is connected reductive and let $\bT \subseteq \bG$ be a 
maximal torus, with associated root datum $\cR=(X,R,Y,R^\vee)$. Let 
$F_p\colon \bG\rightarrow \bG$ be an isogeny as in Example~\ref{canfrob2a}, 
such that 
\[F_p(\bU_\alpha)=\bU_\alpha \quad (\alpha \in R) \qquad \mbox{and} \qquad 
F_p(t)=t^p \quad (t \in \bT).\]
(Note that $F_p$ is only unique up to composition with inner automorphisms
defined by elements of $\bT$.) Now, if one is willing to appeal to a 
stronger version of Theorem~\ref{thmexi} (involving fields of definition), 
then one can find an $F_p$ as above such that $F_p$ is the Frobenius map 
with respect to an $\F_p$-rational structure on $\bG$; see \cite[16.3.3]{Spr}, 
\cite[Expos\'e~XXV]{sga33}. (Alternatively, one could use Lusztig's 
approach \cite{Lu09c}, as pointed out in \cite[Expos\'e~XXV, 
footnote~1]{sga33}.) If $\bG$ is semisimple, then this is also 
contained in \cite[Theorem~6 (p.~58)]{St}, \cite[Part A, \S 3.3 and 
\S 4.3]{Bo131}. Once some $F_p$ is known to be a Frobenius map, 
Lemma~\ref{canfrob12} shows that {\em any} $F_p$ satisfying the above 
conditions is a Frobenius map. 

In any case, for most of our purposes here, it will be sufficient to know 
that $F_p$ is a Steinberg map; this is easily seen as follows. By
Example~\ref{Msplittorus}, there exists a Frobenius map $F\colon\bG
\rightarrow \bG$ such that $F(t)=t^q$ for all $t\in\bT$, where $q=p^m$ 
for some $m\geq 1$. Then $F$ induces multiplication with $q$ on $X$. 
Hence, $F$ induces the same map on $X$ as $F_p^m$. So Lemma~\ref{canfrob12} 
shows that $F_p$ is a Steinberg map. 
\end{exmp}

\begin{lem} \label{Mgenfrob3a} Let $\bT$ be a torus over $k$ and $F\colon 
\bT\rightarrow \bT$ be the Frobenius map corresponding to an $\F_q$-rational 
structure on $\bT$, where $q$ is a power of~$p$. Then the map induced 
by $F$ on $X=X(\bT)$ is given by $q\psi_0$ where $\psi_0 \colon X 
\rightarrow X$ is an invertible endomorphism of finite order. 
\end{lem}

\begin{proof} (Cf.\ \cite[p.~40]{DiMi2}, \cite[\S I.2.4]{Sat71}.) Let $A$
be the algebra of regular functions on $\bT$. Let $\lambda\in X$. Composing 
$\lambda$ with the inclusion $\bkm\hookrightarrow k$, we can regard
$\lambda$ as a regular function on $\bT$, that is, $\lambda \in A$. 
By Remark~\ref{MintrinsF}, we have $F^*(A)=\{a^q\mid a\in A\}$. Hence, 
$\lambda^q=F^*(\lambda^\bullet)$ for some $\lambda^\bullet \in A$. Then we 
have 
\begin{equation*}
\lambda^\bullet(F(t))=\lambda(t)^q=\lambda(t^q) \qquad \mbox{for all $t\in 
\bT$}.\tag{$*$}
\end{equation*}
Since $F\colon \bT\rightarrow \bT$ is a bijective group homomorphism, 
$\lambda^\bullet$ is uniquely determined by ($*$); furthermore, 
$\lambda^\bullet(\bT)\subseteq \bkm$ and $\lambda^\bullet\colon \bT
\rightarrow \bkm$ is a group homomorphism. Hence, $\lambda^\bullet 
\in X$. We also see that the map $\psi \colon X \rightarrow X$, $\lambda 
\mapsto \lambda^\bullet$, is linear. Finally, ($*$) implies that 
$\bigl(\psi^m(\lambda)\bigr)(F^m(t)) =\lambda(t^{q^m})$ for all $m\geq 1$. 
Now, by Example~\ref{Msplittorus}, we can find some $m\geq 1$ such that 
$F^m(t)=t^{q^m}$ for all $t\in \bT$.  For any such $m$, we then have 
$\psi^m(\lambda)=\lambda$ for all $\lambda \in X$. Hence, $\psi$ is an 
endomorphism of $X$ of order dividing $m$. Setting $\psi_0 :=\psi^{-1}$, 
the map on $X$ induced by $F$ is given by $q\psi_0$.
\end{proof}

We now obtain the following characterisation of Steinberg maps.

\begin{prop} \label{Mgenfrob} Let $\bG$ be connected reductive, $F\colon 
\bG\rightarrow \bG$ be an isogeny and $\bT\subseteq \bG$ be an $F$-stable 
maximal torus. Then the following are equivalent.
\begin{itemize}
\item[(i)] $F$ is a Steinberg map.
\item[(ii)] There exist integers $d,m\geq 1$ such that the map induced by 
$F^d$ on $X=X(\bT)$ is given by scalar multiplication with $p^m$.
\item[(iii)] There is an isogeny $F_p$ as in Example~\ref{canfrob2} such 
that some positive power of $F$ equals some positive power of $F_p$.
\end{itemize}
\end{prop}

\begin{proof} ``(i) $\Rightarrow$ (ii)'' Let $d_1\geq 1$ be such that 
$F^{d_1}$ is a Frobenius map with respect to some $\F_{q_0}$-rational 
structure on $\bG$ where $q_0$ is a power of $p$. Let $\varphi \colon 
X\rightarrow X$ be the map induced by $F$. By Lemma~\ref{Mgenfrob3a}, we 
have $\varphi^{d_1}=q_0\psi_0$ where $\psi_0\colon X\rightarrow X$ has 
finite order, $e\geq 1$ say. Then $\varphi^{d_1e}=q_0^e\,\id_X$.

``(ii) $\Rightarrow$ (iii)'' Assume that the map induced by $F^d$ on $X$ is
given by scalar multiplication with $p^m$. Let $F_p$ be as in 
Example~\ref{canfrob2}. Then $F^d$ and $F_p^m$ induce the same map on $X$ 
and so there is  some $y\in \bT$ such that $F^d(g)=yF_p^m(g)y^{-1}$ for 
all $g\in\bG$; see Lemma~\ref{canfrob12}. By Theorem~\ref{langst}, we can 
write $y=x^{-1}F_p^m(x)$ for some $x\in \bT$. As in the proof of 
Lemma~\ref{canfrob0a}, we have $F^d=\iota_x^{-1}\circ F_p^m \circ \iota_x$.
But then we also have $F^d=(\iota_x^{-1}\circ F_p \circ \iota_x)^m$ and it 
remains to note that $F_p'i:=\iota_x^{-1}\circ F_p \circ \iota_x$ is a map 
satisfying the conditions in Example~\ref{canfrob2}.

``(iii) $\Rightarrow$ (i)'' This is clear by definition, since $F_p$ 
is known to be a Steinberg map (see Example~\ref{canfrob2}).
\end{proof}

\begin{prop} \label{Mdefnq} Assume that $\bG$ is connected. Let $F\colon 
\bG\rightarrow \bG$ be a Steinberg map. Let $q$ be the positive real number
defined by $q^d=q_0$, where $d\geq 1$ is an integer such that $F^d$ is
a Frobenius map with respect to some $\F_{q_0}$-rational structure on $\bG$ 
(where $q_0$ is a power of $p$). Then $q$ does not depend on $d,q_0$.
Furthermore, the following hold for every $F$-stable maximal torus $\bT$ of
$\bG$.
\begin{itemize}
\item[(a)] We have $\det(\varphi)=\pm q^{\operatorname{rank} X}$ where 
$\varphi\colon X(\bT)\rightarrow X(\bT)$ is the linear map induced by $F$. 
\item[(b)] The map induced by $F$ on $X_\R:=\R\otimes_\Z X(\bT)$ is of 
the form $q\varphi_0$ where $\varphi_0\in \GL(X_\R)$ has finite order.
\end{itemize}
\end{prop}

\begin{proof} The independence of $q$ of $d,q_0$ is clear, once (a) is
established. So let $\bT$ be any $F$-stable maximal torus of $\bG$ 
(which exists by Example~\ref{Mfrobconj}). Let $X=X(\bT)$ and $\varphi
\colon X\rightarrow X$ be the linear map induced by~$F$. 

(a) By Remark~\ref{frobclosed}, the restriction of $F^d$ to $\bT$ is
a Frobenius map with respect to an $\F_{q_0}$-rational structure on $\bT$.
So, by Lemma~\ref{Mgenfrob3a}, we have $\varphi^d=q_0\psi_0$ where $\psi_0 
\colon X\rightarrow X$ has finite order. Then $\det(\varphi)^d=
q_0^{\operatorname{rank} X} \det(\psi_0)$. Since $\det(\varphi)$ is an 
integer and $\det(\psi_0)$ is a root of unity, we must have 
$\det(\varphi)^d=\pm q_0^{\operatorname{rank} X}$ and, hence, 
$\det(\varphi)=\pm q^{\operatorname{rank} X}$. 

(b) Denote by $\varphi_\R$ the canonical extension of $\varphi$ to $X_\R$. 
Then $\varphi_0:=q^{-1} \varphi_\R$ is a linear map such that
$\varphi_0^d= \psi_0$. Hence, $\varphi_\R=q\varphi_0$ where $\varphi_0$
has finite order. 
\end{proof}

Having defined $q$, one may also write $\bG(q)$ instead of $\bG^F$ if
there is no danger of confusion. An alternative characterisation of $q$ 
will be given in Remark~\ref{Morderform1}(a). The defining formula in 
Proposition~\ref{Mdefnq} shows that $q$ is a $d$-th 
root of a prime power. The examples below will show that 
all such roots do actually occur. 

\begin{exmp} \label{Mdiffexmp} This example is just meant to give a simple
illustration of the difference between Steinberg maps and arbitrary isogenies
with a finite fixed point set. Let $q,q'$ be two distinct powers of $p$. Let 
$\bG=\SL_2(k) \times \SL_2(k)$ and define $F\colon \bG\rightarrow \bG$ 
by $F(x,y)=(F_q(x),F_{q'}(y))$ where $F_q$ and $F_{q'}$ denote the standard
Frobenius maps with respect to $q$ and $q'$, respectively. Then $F$ is a
bijective homomorphism of algebraic groups and $\bG^F=\SL_2(q) \times 
\SL_2(q')$ certainly is finite. Let $\bT\cong \bkm$ be the standard 
maximal torus of $\SL_2(k)$. Then $\bT\times \bT$ is an $F$-stable maximal 
torus of $\bG$ and we can identify $X(\bT\times\bT)$ with $\Z^2$. Under 
this identification, the map induced by $F$ is given by $(n,m) \mapsto 
(qn,q'm)$ for all $(n,m)\in \Z^2$. Thus, Proposition~\ref{Mgenfrob}(b) 
shows that $F$ is not a Steinberg map.
\end{exmp}

\begin{exmp} \label{Mquasisplit} 
Assume that $\bG$ is connected reductive. Let $\bT\subseteq \bG$ be a 
maximal torus and $\cR=(X,R,Y,R^\vee)$ be the root datum relative to 
$\bT$. Assume that we have an automorphism $\varphi_0 \colon X \rightarrow 
X$ of finite order such that $\varphi_0(R)=R$ and $\varphi_0^\vee(R^\vee)
=R^\vee$. (In particular, $\varphi_0$ is an isogeny of root data with 
all root exponents equal to~$1$.) Let $q=p^m$ for some $m\geq 1$. Then 
$q\varphi_0$ is a $p$-isogeny and so, by Theorem~\ref{thmiso}, there is a 
corresponding isogeny $F\colon \bG \rightarrow \bG$ such that $F(\bT)=\bT$.
Now $F$ is a Steinberg map by Proposition~\ref{Mgenfrob}; the number $q=p^m$ 
satisfies the conditions in Proposition~\ref{Mdefnq}. If $\bG$ is semisimple,
then $G=\bG^F$ is an untwisted ($\varphi_0=\id_X$) or twisted Chevalley 
group; see Steinberg's lecture notes \cite[\S 11]{St} for further details.
We discuss the various possibilities in more detail in 
Section~\ref{sec:generic}. 

Let us just give one concrete example. Let $\bG=\GL_n(k)$. If $\varphi_0=
\id_X$, then we obtain a Frobenius map $F\colon \GL_n(k)\rightarrow 
\GL_n(k)$ such that $\GL_n(k)^F=\GL_n(q)$, the finite general linear
group over $\F_q$. On the other hand, the bijective isogeny $\varphi_0 
\colon X \rightarrow X$ of order~$2$ defined in Example~\ref{isounitary} 
also satisfies the above conditions. The corresponding isogeny $F'\colon 
\GL_n(k) \rightarrow \GL_n(k)$ is a Steinberg map such that $\GU_n(q):=
\GL_n(k)^{F'}$ is the {\em finite} \nm{general unitary group}.
Similarly, we have $\SL_n(k)^{F}=\SL_n(q)$ and $\SL_n(k)^{F'}=\SU_n(q)$.
\end{exmp}

\begin{exmp} \label{Msuzuki2} Assume that $\bG$ is connected reductive and
that the root datum $\cR=(X,R,Y,R^\vee)$ relative to a maximal 
torus $\bT\subseteq \bG$ is as in Example~\ref{MsuzukiB}, where $p=2$
or $3$. For any $m\geq 0$, we have a $p$-isogeny $\varphi_m$ on $X$ such 
that $\varphi_m^2=p^{2m+1}\,\id_X$. Let $F\colon\bG\rightarrow \bG$ be the 
corresponding isogeny such that $F(\bT)=\bT$. Then 
Proposition~\ref{Mgenfrob} shows that $F$ is a Steinberg map; the number 
$q$ in Proposition~\ref{Mdefnq} is given by $q=\sqrt{p}^{2m+1}$. In these 
cases\footnote{In finite group theory, it is common to write 
${^2\!B}_2(q^2)$ etc., although this is not entirely consistent with the 
general setting of algebraic groups where the notation should be 
${^2\!B}_2(q)$.}, $G=\bG^F$ is the Suzuki group ${^2\!B}_2(q^2)=
{^2\!C}_2(q^2)$, the Ree group ${^2\!G}_2(q^2)$ or the Ree group 
${^2\!F}_4(q^2)$, respectively. See \cite[Chap.~13]{Ca1} or Steinberg's 
lecture notes \cite[\S 11]{St} for further details.
\end{exmp}

\begin{exmp} \label{expdelu1} Let $F\colon \bG\rightarrow \bG$ be the
Frobenius map corresponding to some $\F_{q_0}$-rational structure on $\bG$
where $q_0$ is a power of $p$. Consider the direct product 
$\bG'=\bG\times \cdots \times \bG$ (with $r$ factors) and define a map 
\[F'\colon \bG'\rightarrow \bG', \qquad (g_1,g_2,\ldots,g_r)\mapsto 
(F(g_r),g_1, \ldots,g_{r-1}).\]
Then $F'$ is a homomorphism of algebraic groups and we have $(F')^r(g_1,
\ldots,g_r)=(F(g_1),F(g_2),\ldots,F(g_r))$ for all $g_i \in \bG$. Clearly, 
the latter map is a Frobenius map on $\bG'$. Thus, $F'$ is a Steinberg map.
The number $q$ in Proposition~\ref{Mdefnq} is given by $q=\sqrt[r]{q_0}$.
Note also that we have a group isomorphism
\[ \bG'^{F'} \stackrel{\sim}{\longrightarrow} \bG^F,\qquad (g_1,g_2,
\ldots,g_r) \mapsto g_1.\]
(This example is mentioned in \cite[\S 11]{DeLu}.)
\end{exmp}

We have the following extension of the Isogeny Theorem~\ref{thmiso},
taking into account the presence of Steinberg maps.

\begin{lem} \label{lemcompatiso} In the set-up of Theorem~\ref{thmiso} assume,
in addition, that there are Steinberg maps $F\colon\bG \rightarrow \bG$ and 
$F'\colon\bG' \rightarrow \bG'$ such that $F(\bT)=\bT$, $F'(\bT') =\bT'$ 
and $\Phi\circ \varphi=\varphi \circ \Phi'$ where $\Phi\colon X(\bT) 
\rightarrow X(\bT)$ and $\Phi'\colon X(\bT') \rightarrow X(\bT')$ are the 
maps induced by $F$ and $F'$. Then there exists an isogeny $f\colon \bG 
\rightarrow \bG'$ which maps $\bT$ onto $\bT'$ and induces $\varphi$, and 
such that $f\circ F=F'\circ f$.
\end{lem}

\begin{proof} By Theorem~\ref{thmiso}, there exists an isogeny $f'\colon \bG
\rightarrow \bG'$ which maps $\bT$ onto $\bT'$ and induces~$\varphi$. Then
$f'\circ F$ and $F'\circ f'$ both induce $\Phi\circ\varphi=\varphi\circ
\Phi'$. Hence, by Theorem~\ref{thmiso}, there exists some $t\in \bT$ such 
that $(F'\circ f')(g)=(f'\circ F)(t^{-1}gt)$ for all $g\in \bG$. Then 
$f'(F(t))\in \bT'$ and so, by Theorem~\ref{langst}, we can write 
$f'(F(t))=x^{-1}F'(x)$ for some $x\in \bT'$. We define $f\colon \bG 
\rightarrow \bG$ by $f(g)=xf'(g)x^{-1}$ for all $g\in \bG$. Then $f$ is 
an isogeny which maps $\bT$ onto $\bT'$ and also induces $\varphi$.  
Furthermore, 
\[(F'\circ f)(g)=F'(x)(f' \circ F)(t^{-1}gt) F'(x)^{-1}=x(f'\circ F)(g) 
x^{-1}=(f\circ F)(g)\]
for all $g\in \bG$, as required.
\end{proof}

\begin{exmp} \label{Mopposite1} Assume that $\bG$ is connected reductive. 
Let $F\colon \bG\rightarrow \bG$ be a Steinberg map and $\bT$ be an 
$F$-stable maximal torus of $\bG$. 

(a) Lemma~\ref{lemcompatiso} immediately shows that an automorphism $\tau
\colon \bG \rightarrow \bG$ as in Example~\ref{Mopposite} can be chosen 
such that we also have $\tau \circ F=F\circ \tau$.

(b) Consider an isogeny $F_p\colon \bG\rightarrow \bG$ as in
Example~\ref{canfrob2} and let $\varphi$ be the map induced on $X=X(\bT)$ 
by $F$. Since $F_p$ is a Steinberg map, Lemma~\ref{lemcompatiso} shows that 
there is an isogeny $F'\colon \bG \rightarrow \bG$ which maps $\bT$ onto 
$\bT$ and induces $\varphi$, and such that $F'\circ F_p=F_p\circ F'$. Since 
$F,F'$ induce the same map on $X$, we have that $F'$ is a Steinberg map
such that $\bG^F \cong \bG^{F'}$; see Lemma~\ref{canfrob12}. (Thus, 
replacing $F$ by $F'$ if necessary, we can always assume that $F\circ F_p=
F_p\circ F$, that is, we are in the setting of \cite[\S 2.1]{LuB}.)
\end{exmp}

\begin{lem} \label{Mgenfrob2} Assume that $\bG$ is connected reductive. 
Let $\bK$ be a closed normal subgroup of $\bG$. Then $\bK$ is reductive and 
$\bar{\bG}:=\bG/\bK$ is connected and reductive. If, furthermore, $F\colon 
\bG \rightarrow \bG$ is a Steinberg map such that $F(\bK)=\bK$, then 
the map $\bar{F}\colon \bar{\bG}\rightarrow \bar{\bG}$ induced by $F$ 
is a Steinberg map.
\end{lem}

\begin{proof} Since the unipotent radical in a linear algebraic group
is invariant under any automorphism of algebraic groups, it is clear that 
every closed normal subgroup of $\bG$ is reductive. Now consider 
$\bar{\bG}=\bG/\bK$. First recall from \ref{subsec1quot} that $\bar{\bG}$ 
is a linear algebraic group; it is also connected since it is the quotient 
of a connected group. Finally, $\bar{\bG}$ is reductive by 
\ref{Mconnredcomp}(b). 

Now let $F\colon \bG\rightarrow \bG$ be a Steinberg map and assume that 
$F(\bK)\subseteq \bK$. Then we obtain an induced (abstract) group
homomorphism $\bar{F} \colon \bar{\bG} \rightarrow \bar{\bG}$, which is
bijective. By the universal property of quotients, $\bar{F}$ is a 
homomorphism of algebraic groups. Let $\bT\subseteq \bG$ be an $F$-stable
maximal torus. Let $\bar{\bT}$ be the image of $\bT$ in $\bar{\bG}$. By
\ref{Mconnredcomp}(a), $\bar{\bT}$ is a maximal torus of $\bar{\bG}$; we 
also have $\bar{F}(\bar{\bT})=\bar{\bT}$. Let $X=X(\bT)$, $\bar{X}=
X(\bar{\bT})$ and $\varphi \colon \bar{X} \rightarrow X$ be the map 
induced by the canonical map $f\colon \bG \rightarrow \bar{\bG}$; note 
that $\varphi$ is injective. Let $\psi\colon X \rightarrow X$ and 
$\bar{\psi}\colon\bar{X} \rightarrow \bar{X}$ be the maps induced by $F$ 
and $\bar{F}$. Since $\bar{F}\circ f=f\circ F$, we have $\varphi \circ 
\bar{\psi}= \psi \circ \varphi$ and so $\varphi \circ \bar{\psi}^m=\psi^m
\circ \varphi$ for all $m\geq 1$. Now there is some $m\geq 1$ such that 
$\psi^m$ is given by scalar multiplication with a power of~$p$. Since
$\varphi$ is injective, this implies that $\bar{\psi}^m$ is also given by 
scalar multiplication with a power of~$p$. Hence, $\bar{F}$ is a Steinberg 
map by Proposition~\ref{Mgenfrob}. 
\end{proof}

Finally, we address the question of characterising Frobenius maps among 
all Steinberg maps on $\bG$. The results are certainly well-known to the
experts and are contained in more advanced texts on reductive groups (like 
\cite{BoTi}, \cite{Sat71}), where they appear as special cases of  
general considerations of rationality questions. Since in our case the 
rational structures are given by Frobenius maps, one can give more direct 
proofs. The key property is contained in the following result.

\begin{lem} \label{Mgenfrob3b} Let $\bG$ be connected reductive
and $F\colon \bG\rightarrow \bG$ be a Frobenius map with respect to some
$\F_q$-rational structure on $\bG$. Let $\bT$ be an $F$-stable maximal
torus. Then the root exponents of $F$ (relative to $\bT$) are all equal 
to~$q$.
\end{lem}

\begin{proof} (Cf.\ \cite[6.2]{BoTi}, \cite[\S II.2.1]{Sat71}.) Let $R$
be the set of roots with respect to $\bT$. Let $\alpha
\in R$ and $u_\alpha\colon \bkp\rightarrow \bG$ be the corresponding
homomorphism with image $\bU_\alpha\subseteq \bG$. We have $F(\bU_\alpha)=
\bU_{\alpha^\dagger}$, where $\alpha^\dagger\in R$; see \ref{Mabs23}. In 
order to identify the root exponents, we need to exhibit a homomorphism 
$u_{\alpha^\dagger}\colon  \bkp\rightarrow \bG$ whose image is 
$\bU_{\alpha^\dagger}$ and such that $u_{\alpha^\dagger}$ is an
isomorphism onto its image. This is done as follows. Let $A$ be the algebra 
of regular functions on $\bG$. The algebra of
regular functions on $\bkp$ is given by the polynomial ring $k[c]$
where $c$ is the identity function on $\bkp$. Since $u_\alpha$ is an
isomorphism onto its image, the induced algebra homomorphism $u_\alpha^*
\colon A\rightarrow k[c]$ is surjective (see \cite[2.2.1]{mybook}). Now
consider the standard Frobenius map $F_1\colon \bkp\rightarrow \bkp$, $\xi 
\mapsto \xi^q$. By Remark~\ref{MintrinsF}, we have $F^*(A)=\{a^q\mid a\in 
A\}$ and $F_1^*(k[c])=k[c^q]$. Hence, the composition $u_\alpha^* \circ F^*$ 
sends $A$ onto $k[c^q]$. Since $F_1^*\colon k[c]\rightarrow k[c^q]$ is an
isomorphism, we obtain an algebra homomorphism $\gamma \colon A \rightarrow
k[c]$ by setting $\gamma:=(F_1^*)^{-1}\circ u_\alpha^* \circ F^*$; note 
that $\gamma$ is surjective. Let $u_{\alpha^\dagger}\colon \bkp\rightarrow 
\bG$ be the morphism of affine varieties such that $u_{\alpha^\dagger}^*=
\gamma$. Then $F\circ u_\alpha= u_{\alpha^\dagger}\circ F_1$ and so 
\[u_{\alpha^\dagger}(\xi^q)=(u_{\alpha^\dagger}\circ F_1)(\xi)=(F\circ 
u_\alpha)(\xi)=F(u_\alpha(\xi)) \qquad \mbox{for all $\xi\in k$}.\]
This shows, first of all, that $u_{\alpha^\dagger}$ is a group homomorphism 
with image $F(\bU_\alpha)=\bU_{\alpha^\dagger}$. Furthermore, since $\gamma$ 
is surjective, $u_{\alpha^\dagger}$ is an isomorphism onto its image (see 
again \cite[2.2.1]{mybook}). For all $t\in \bT$ and $\xi\in k$, we have
\[ F(t)u_{\alpha^\dagger}(\xi^q)F(t)^{-1}=F(tu_\alpha(\xi)t^{-1})=
F(u_\alpha(\alpha(t)\xi))=u_{\alpha^\dagger}(\alpha(t)^q\xi^q),\]
which shows that $\alpha^\dagger(F(t))=\alpha(t)^q$ for all $t\in \bT$, 
as desired.
\end{proof}

In the following result, the proof of the implication ``(iii) $\Rightarrow$ 
(i)'' relies on the fact that $F_p$ as in Example~\ref{canfrob2} is a 
Frobenius map.

\begin{prop} \label{Mgenfrob3} Assume that $\bG$ is connected reductive.
Let $F\colon \bG\rightarrow \bG$ be an isogeny and $\bT\subseteq \bG$ 
be an $F$-stable maximal torus. Let $\varphi$ be the map induced by $F$ 
on $X=X(\bT)$. Then the following conditions are equivalent.
\begin{itemize}
\item[(i)] $F$ is a Frobenius map (corresponding to a rational 
structure on $\bG$ over a finite subfield of $k$).
\item[(ii)] We have $\varphi=p^m\varphi_0$ where $m\in \Z_{\geq 1}$ and 
$\varphi_0 \colon X \rightarrow X$ is an automorphism of finite order
such that $\varphi_0(R)=R$ and $\varphi_0^\vee(R^\vee)=R^\vee$. (In 
particular, $\varphi_0$ is an isogeny of root data with all root exponents
equal to~$1$.)
\item[(iii)] There exists an automorphism of algebraic groups $\gamma\colon 
\bG\rightarrow \bG$ of finite order such that $\gamma(\bT)=\bT$ and some 
$m'\geq 1$ such that $F=\gamma \circ F_p^{m'}=F_p^{m'}\circ \gamma$, where 
$F_p$ is an isogeny as in Example~\ref{canfrob2}.
\end{itemize}
If these conditions hold, then $m$ (as in {\rm (ii)}) equals $m'$ (as in 
{\rm (iii)}) and $F$ is the Frobenius map with respect to an $\F_q$-rational 
structure where $q=p^m$. Furthermore, all root exponents of $F$ are equal 
to~$q$ and $q$ is the number defined in Proposition~\ref{Mdefnq}.
\end{prop}

\begin{proof} ``(i) $\Rightarrow$ (ii)'' Let $F$ be a Frobenius map
corresponding to an $\F_q$-rational structure on $\bG$ where $q=p^m$ for 
some $m\geq 1$. By assumption, $\bT$ is $F$-stable so $\bT$ is also defined
over $\F_q$; see Remark~\ref{frobclosed}. Hence, we can apply 
Lemma~\ref{Mgenfrob3a} and so $\varphi=q\varphi_0$ where $\varphi_0 \colon
X \rightarrow X$ has finite order. Furthermore, using Lemma~\ref{Mgenfrob3b},
one sees that $\varphi_0(\alpha^\dagger)=\alpha$ and $\varphi_0^\vee
(\alpha^\vee)=(\alpha^\dagger)^\vee$ for all $\alpha\in R$. Thus, 
(I1) and (I2) hold for $\varphi_0$, where the root exponents of $\varphi_0$ 
are all equal to~$1$.

``(ii) $\Rightarrow$ (iii)''  Let $F_p\colon \bG\rightarrow \bG$ be as in 
Example~\ref{canfrob2}. Then $F_p^m$ is a Steinberg map 
and the map induced by $F_p^m$ on $X$ is scalar multiplication with $p^m$. 
So, by Lemma~\ref{lemcompatiso}, there exists
an isogeny $f\colon \bG\rightarrow \bG$ which maps $\bT$ onto itself and
induces $\varphi_0$, and such that $f \circ F_p^m=F_p^m\circ f$. 
Now $\varphi_0$ has finite order, say~$d$. Then $f^d$ induces the identity 
on $X$. Hence, by Theorem~\ref{thmiso}, there exists some $t\in \bT$ such 
that $f^d(g)= tgt^{-1}$ for all $g\in\bG$. Since $t$ also has finite order,
we conclude that some positive power of $f^d$ is the identity. Hence, $f$ 
itself has finite order. Now, $F$ and $F':=f\circ F_p^m$ are isogenies 
which induce the same map on $X$. Hence, by Theorem~\ref{thmiso}, there
exists some $y\in \bT$ such that $F'(g)=yF(g)y^{-1}$ for all $g\in \bG$.
As in the proof of Lemma~\ref{canfrob0a}, there exists some $x \in 
\bT$ such that $F'=\iota_x^{-1} \circ F\circ \iota_x$, where $\iota_x$ 
denotes the inner automorphism of $\bG$ defined by $x$. Then 
\[F=\iota_x\circ F' \circ \iota_x^{-1}=(\iota_x \circ f \circ 
\iota_x^{-1}) \circ (\iota_x \circ F_p \circ \iota_x^{-1})^m\]
(and the two factors still commute). Now, since $x \in \bT$, the isogeny
$F_p':=\iota_x \circ F_p \circ \iota_x^{-1}$ also satisfies the conditions 
in Example~\ref{canfrob2}. Furthermore, $\gamma:=\iota_x \circ f \circ 
\iota_x^{-1}$ is an automorphism of finite order such that $\gamma(\bT)=
\bT$. Thus, (c) holds. 

``(iii) $\Rightarrow$ (i)'' As already mentioned in Example~\ref{canfrob2}, 
we can assume that $F_p$ is the Frobenius map corresponding to an 
$\F_p$-rational structure on $\bG$. Then $F_p^m$ is the Frobenius map 
corresponding to an $\F_{q}$-rational structure on $\bG$ where $q=p^m$.
Hence so is $F=\gamma \circ F_p^m$ by Remark~\ref{MintrinsF}(c).

Finally, assume that (i), (ii), (iii) hold. Then the above arguments show 
that $m=m'$. Furthermore, (ii) shows that $\det(\varphi)=\pm 
(p^m)^{\text{rank} X}$ and so $q=p^m$ satisfies the conditions in 
Proposition~\ref{Mdefnq}.
\end{proof}

The following example indicates that Steinberg maps can be much more 
complicated than Frobenius maps.

\begin{exmp} \label{Mstrangexp} (a) In the setting of Example~\ref{expdelu1},
one easily sees that neither the conclusion of Lemma~\ref{Mgenfrob3a} nor
of that of Lemma~\ref{Mgenfrob3b} hold for $F'$.  Hence, although $F$ is
a Frobenius map, the map $F'$ is not.

(b) Let $\bG=\SL_2(k) \times \PGL_2(k)$ and $p=2$. Then $\bG$ is semisimple 
of type $A_1\times A_1$, with Cartan matrix $C=2 I_2$ where $I_2$ denotes 
the identity matrix. The root datum of $\bG$ is determined by the 
factorisation
\[C=\breve{A}\cdot A^{\text{tr}} \quad \mbox{where} \quad  
A=\left(\begin{array}{cc}  2& 0  \\ 0   & 1   \end{array}\right)\quad
\mbox{and}\quad 
\breve{A}=\left(\begin{array}{cc} 1  & 0  \\ 0  & 2   \end{array}\right).\]
For a fixed $m\geq 1$, we define 
\[ P=\left(\begin{array}{cc} 0 & 2^m  \\ 2^m  & 0   \end{array}\right)\quad
\mbox{and}\quad 
P^\circ=\left(\begin{array}{cc} 0 & 2^{m-1}  \\ 2^{m+1} & 0 
\end{array}\right).\]
Then the pair $(P,P^\circ)$ satisfies the conditions in \ref{Mdefmatrixisog}
and so there is a corresponding isogeny $F\colon \bG\rightarrow \bG$, with
root exponents $2^{m+1}$, $2^{m-1}$. Since $P^2=4^mI_2$, we know that $F$ is 
a Steinberg map (see Proposition~\ref{Mgenfrob}). Furthermore, we have $P=2^m
P_0$ where $P_0\in M_2(\Z)$ has order $2$; thus, the conclusion of
Lemma~\ref{Mgenfrob3a} holds for $F$ where $q=2^m$. Note also that the two 
projections (on the first and on the second factor) define isomorphisms of 
finite groups $\bG^F\cong \SL_2(q) \cong \PGL_2(q)$. 

On the other hand, since the root exponents are not all equal, 
Lemma~\ref{Mgenfrob3b} shows that $F$ is not a Frobenius map! One easily 
checks directly that there is no matrix $P_0^\circ$ such that the pair 
$(P_0, P_0^\circ)$ satisfies the conditions in \ref{Mdefmatrixisog}. Thus, 
$P_0$ does not come from an isogeny of root data. 
\end{exmp}

\section{Working with isogenies and root data; examples} 
\label{sec:semisimple}

We now discuss some applications and examples of the theory developed
so far. We start with some basic material about semisimple groups. 
Up until Proposition~\ref{MdirprodG}, $k$ may be any algebraically closed 
field. 

\begin{abs} \label{MlambdaC} Let us fix a Cartan matrix $C=(c_{st})_{s,
t\in S}$. Let $\Lambda(C)$ be the finite abelian group defined in  
Remark~\ref{Mcartanfund}. Then the semisimple algebraic groups with a 
root datum of Cartan type $C$ are classified in terms of subgroups of 
$\Lambda(C)$. This works as follows. We have $\Lambda(C):=\Omega/\Z C$
where $\Omega$ is the free abelian group with basis $\{\omega_s\mid s
\in S\}$ and $\Z C$ is the subgroup generated by $\{\sum_{s\in S}c_{st} 
\omega_s\mid t\in S\}$. Thus, giving a subgroup of $\Lambda(C)$ is the 
same as giving a lattice $L$ such that $\Z C \subseteq L \subseteq \Omega$. 
Such a lattice $L$ is free abelian of the same rank as $\Omega$. We choose a
set of free generators $\{x_s\mid s \in S\}$ of $L$. Since $\Z C\subseteq 
L$, we have unique expressions
\begin{equation*}
\sum_{s\in S}c_{st}\omega_s=\sum_{u\in S} a_{tu}x_u \quad (t\in S),\qquad
\mbox{where} \qquad A=(a_{tu})_{t,u\in S}\tag{$*$}
\end{equation*}
and $A$ is a square matrix with integer coefficients. We also write 
$x_u=\sum_{s
\in S} \breve{a}_{su} \omega_s$ where $\breve{A}=(\breve{a}_{su})_{s,u\in 
S}$ is a square matrix with integer coefficients. Substituting this into 
($*$) and comparing coefficients, we obtain $C=\breve{A}\cdot 
A^{\text{tr}}$. As in \ref{Mcartan1}, such a factorisation of $C$ 
determines a root datum $\cR_L=(X,R,Y,R^\vee)$ of Cartan type $C$, where 
$R$ has a base given by $\alpha_t:=\sum_{s\in S} a_{ts} x_s$ for $t\in S$. 
We have $|X/\Z R|<\infty$ since $A, \breve{A}$ are square matrices. If we 
choose a different set of generators of $L$, say $\{y_t\mid t\in S\}$,
then we obtain another factorisation $C=\breve{B}\cdot B^{\text{tr}}$ 
where $B, \breve{B}$ are square integer matrices. Writing $y_t=\sum_{u\in S}
p_{ut} x_u$ where $P=(p_{ut})_{u, t\in S}$ is invertible over $\Z$, we 
have $PB^{\text{tr}}=A^{\text{tr}}$ and $\breve{B}=\breve{A}P$. Hence, 
the root data defined by $C=\breve{A}\cdot A^{\text{tr}}$ and by $C=
\breve{B}\cdot B^{\text{tr}}$ are isomorphic; see Lemma~\ref{Misorootdat}. 
Thus, every lattice $L$ such that $\Z C \subseteq L \subseteq \Omega$ 
determines a root datum $\cR_L$ as above, which is unique up to isomorphism. 
By Theorem~\ref{thmexi}, there exists a corresponding connected reductive 
algebraic group $\bG_L$ over $k$ (unique up to isomorphism by
Corollary~\ref{Mremiso1}). The group $\bG_L$ is semisimple since 
$|X/\Z R|<\infty$; see Remark~\ref{Msemisimple}.
\end{abs}

\begin{prop} \label{Mssclass} Let $\bG$ be a semisimple algebraic group
over $k$ with root datum $\cR=(X,R,Y,R^\vee)$ (relative to some maximal 
torus of $\bG$). Let $C=(c_{st})_{s,t\in S}$ be the Cartan matrix of $\cR$.
Then there exists a lattice $L$ as in \ref{MlambdaC} such that $\bG\cong 
\bG_L$. We have $X/\Z R \cong L/\Z C$ and, hence, $\bZ(\bG) \cong 
\Hom(L/\Z C,\bkm)$. 
\end{prop}

\begin{proof} Let $\Pi$ be a base of $R$; we have $|\Pi|= \mbox{rank}\, X$ 
since $\bG$ is semisimple and, hence, $X/\Z R$ is finite. Also choose
a $\Z$-basis of $X$. By Remark~\ref{Mcorbrulu1}, we have a corresponding
factorisation $C=\breve{A} \cdot A^{\text{tr}}$ where $A$, $\breve{A}$
are {\em square} integral matrices. In particular, we can use $S$ as an 
indexing set for both the rows and the columns of these matrices. Then 
let $L$ be the sublattice of $\Omega$ spanned by the elements
\[x_t:=\sum_{s\in S} \breve{a}_{st} \omega_s\quad (t\in S), \qquad
\mbox{where}\qquad \breve{A}=(\breve{a}_{st})_{s,t\in S}.\]
We have $\Z C\subseteq L\subseteq \Omega$, since $C=\breve{A} \cdot 
A^{\text{tr}}$. Applying the construction in \ref{MlambdaC} to $L$,
we obtain a group $\bG_L$. Then Corollary~\ref{Mremiso1} shows that 
$\bG \cong \bG_L$. Finally, \ref{MlambdaC}($*$) implies that 
$X/R\cong L/\Z C$ and this yields $\bZ(\bG)$; see Remark~\ref{Msemisimple}.
\end{proof}

\begin{exmp} \label{Madjscexmp} Let $C=(c_{st})_{s,t\in S}$ be a Cartan 
matrix and consider the group $\Lambda(C)=\Omega/\Z C$, as above. 
\begin{itemize}
\item[(a)] If $L=\Z C$, then we choose the generators $\{x_s\mid s\in S\}$ 
of $L$ to be the given generators of $\Z C$. So $A$ in \ref{MlambdaC}($*$) 
is the identity matrix and $\breve{A}=C$. Hence, in this case, $\cR_L$ is 
the root datum $\cRad(C)$ in Example~\ref{expadjsc}. The corresponding 
group $\bG_L$ will be denoted by $\bGad$; we have $\bZ(\bGad)=\{1\}$.
\item[(b)] If $L=\Omega$, then we can take $x_s=\omega_s$ for all $s\in S$.
So $A=C^{\text{tr}}$ and $\breve{A}$ is the identity matrix. Hence, in 
this case, $\cR_L$ is the root datum $\cRsc(C)$ in Example~\ref{expadjsc}. 
The corresponding group $\bG_L$ will be denoted by $\bGsc$; we have 
$\bZ(\bGsc)\cong \Hom(\Lambda(C),\bkm)$.
\end{itemize}
The groups  $\bGsc$ and $\bGad$ have some important univeral properties
which will be discussed in further detail below. We shall call $\bGsc$ 
the \nm{semisimple group of simply-connected type} $C$ and $\bGad$ the 
\nm{semisimple group of adjoint type} $C$.
\end{exmp}

\begin{exmp} \label{Mvartypes} Assume that $C$ is an indecomposable
Cartan matrix. The isomorphism types of $\Lambda(C)$ are listed in
Remark~\ref{Mcartanfund}.
\begin{itemize}
\item[(a)] If $C$ is of type $A_{n-1}$, then $\Lambda(C)\cong\Z/n\Z$. 
Hence, for each divisor $d$ of $n$, we have a unique lattice $L_d\subseteq 
\Omega$ such that $|L_d/\Z C|=d$; let $\bG_{(d)}$ be the corresponding 
group. We have $\bG_{(1)}=\bGad\cong \PGL_n(k)$ and $\bG_{(n)}=\bGsc\cong 
\SL_n(k)$; see Examples~\ref{rootdatPGL} and~\ref{rootdatSL}. The 
remaining groups $\bG_{(d)}$ are explicitly constructed in 
\cite[\S 20.3]{Ch05}, as images of $\SL_n(k)$ under certain representations.
\item[(b)] If $C$ is of type $B_n$, $C_n$, $E_6$ or $E_7$, then 
$\Lambda(C)$ is cyclic of prime order. Hence, either $L=\Z C$ or 
$L=\Omega$. So, in this case, the only possible groups are $\bGad$ and 
$\bGsc$. 

In type $B_n$, we have $\bGad\cong \SO_{2n+1}(k)$ and $\bGsc\cong
\Spin_{2n+1}(k)$.

In type $C_n$, we have $\bGad\cong \text{PCSp}_{2n}(k)$ and $\bGsc\cong
\Sp_{2n}(k)$.

\noindent (See the references in \ref{subsecclassic} for the precise 
definitions of these groups.)
\item[(c)] If $C$ is of type $D_n$, then $\Lambda(C)$ has order $4$ and 
there are $3$ (for $n$ odd) or $5$ (for $n$ even) possible lattices $L$. 
We have $\bGad\cong \text{PCO}_{2n}^\circ(k)$ and $\bGsc\cong \Spin_{2n}(k)$.
Using the labelling in Table~\ref{Mdynkintbl}, the group $\SO_{2n}(k)$ 
corresponds to the unique $L$ of index $2$ in $\Omega$ which is 
invariant under the involution of $\Omega$ obtained by exchanging 
$\omega_1$ and $\omega_2$. In terms of our matrix language in 
Section~\ref{sec:rootdata}, the root datum of $\SO_{2n}(k)$ is given
by the factorisation $C=\breve{A}\cdot A^\trp$ where 
\[ A=\breve{A}=\left(\begin{array}{rrrrrrr} 1 & 1 & 0 & \ldots & & 0 \\
-1 & 1 & 0 & 0 & \ldots & 0 \\
0 & -1 & 1 & 0 & \ddots & \vdots \\
\vdots & \ddots & \ddots & \ddots & \ddots & 0\\
0 & \ldots & 0 & -1 & 1 & 0 \\
0 & & \ldots &  0 & -1 & 1 \end{array}\right).\]
(See \cite[Exercise~7.4.7]{Spr}.) If $n$ is even, then there are two 
further lattices of index $2$, which both give rise to the half-spin 
group $\text{HSpin}_{2n}(k)$.
\item[(d)] Finally, if $C$ is of type $G_2$, $F_4$ or $E_8$, then 
$\Lambda(C)=\{0\}$. So, in this case, all semisimple algebraic groups 
over a fixed field $k$ with a root datum of Cartan type $C$ are isomorphic 
to each other; in particular, $\bGsc \cong \bGad$.
\end{itemize}
We refer to \cite{Grove2}, \cite[\S 1.7]{mybook}, \cite[\S 7.4]{Spr}, 
for further details about the various groups of classical type $B_n$, 
$C_n$ and $D_n$.
\end{exmp}

\begin{exmp} \label{isogBC} The Cartan matrices of type $B_n$ and 
$C_n$ are the $n\times n$-matrices given by 
\[\renewcommand{\arraystretch}{1.1} \renewcommand{\arraycolsep}{4pt}
B_n:\;\left(\begin{array}{rrrrrrr} 2 & -2 & 0 & \ldots & & 0 \\
-1 & 2 & -1 & 0 & \ldots & 0 \\
0 & -1 & 2 & -1 & \ddots & \vdots \\
\vdots & \ddots & \ddots & \ddots & \ddots & 0 \\
0 & \ldots & 0 & -1 & 2 & -1 \\
0 & & \ldots &  0 & -1 & 2 \end{array}\right), \quad  C_n:\;
\left(\begin{array}{rrrrrrr} 2 & -1 & 0 & \ldots & & 0 \\
-2 & 2 & -1 & 0 & \ldots & 0 \\
0 & -1 & 2 & -1 & \ddots & \vdots \\
\vdots & \ddots & \ddots & \ddots & \ddots & 0\\
0 & \ldots & 0 & -1 & 2 & -1 \\
0 & & \ldots &  0 & -1 & 2 \end{array}\right),\]
respectively. Let $C$ denote the second matrix, and $C'$ the first. Let 
$P^{\circ}$ be the diagonal matrix of size $n$ with diagonal entries $1,2,2,
\ldots,2$. Then $CP^\circ =P^\circ C'$. Thus, if we also set $P=P^\circ$, 
then the two conditions in \ref{Mdefmatrixisog} are satisfied and so the pair 
$(P,P^\circ)$ defines a $2$-isogeny from $\cR_{\text{sc}}(C)$ to 
$\cR_{\text{sc}}(C')$ (see Example~\ref{expadjsc}). Let 
$k$ be an algebraically closed field of characteristic~$2$. Let $\bGsc$ 
and $\bGsc'$ be the semisimple algebraic groups over $k$ corresponding to 
$\cR_{\text{sc}}(C)$ and $\cR_{\text{sc}}(C')$, respectively. We have
$\bGsc\cong \Spin_{2n+1}(k)$ and $\bGsc'\cong \Sp_{2n}(k)$. Then
Theorem~\ref{thmiso} yields the existence of an isogeny $f\colon 
\bGsc\rightarrow \bGsc'$. This is one of Chevalley's exceptional 
isogenies considered at the end of \cite[\S 23.7]{Ch05}.
\end{exmp}

\begin{exmp} \label{Mcentersc} Let $\bG$ be connected reductive over $k$ 
with root datum $\cR=(X,R,Y,R^\vee)$, relative to some maximal torus 
$\bT$ of $\bG$. Dually to the isomorphism in \ref{subsec17}(c), we have 
a canonical isomorphism (see \cite[\S 3.1]{Ca2})
\[ Y(\bT)\otimes_{\Z} \bkm \stackrel{\sim}{\longrightarrow} \bT,
\qquad \nu\otimes \xi \mapsto \nu(\xi).\]
Hence, if $\{\nu_1,\ldots,\nu_n\}$ is a $\Z$-basis of $Y(\bT)$, then
every element $t\in\bT$ can be written uniquely in the form $t=\nu_1(\xi_1)
\cdots \nu_n(\xi_n)$ where $\xi_1,\ldots,\xi_n\in k^\times$.

Now assume that $\bG$ is semisimple of simply-connected type. Then
$Y(\bT)=\Z R^\vee$. Let $\Pi=\{\alpha_1,\ldots,\alpha_n\}$ be a 
base for $R$ and $\{\alpha_1^\vee,\ldots,\alpha_n^\vee\}$ be the
corresponding base for $R^\vee$. Hence, we have
\[ \bT=\{h(\xi_1,\ldots,\xi_n):=\alpha_1^\vee(\xi_1)\cdots
\alpha_n^\vee(\xi_n)\mid \xi_1,\ldots,\xi_n\in\bkm\}.\]
In this setting, one can now explicitly determine the center of $\bG$
as a subset of $\bT$. Indeed, using Remark~\ref{Msemisimple} and the above 
description of $\bT$, we obtain 
\[\bZ(\bG)=\{h(\xi_1,\ldots,\xi_n)\in \bT\mid \xi_1^{\langle \alpha_j,
\alpha_1^\vee \rangle} \cdots \xi_n^{\langle \alpha_j,\alpha_n^\vee
\rangle}=1 \mbox{ for $1\leq j\leq n$}\}.\]
Now the numbers $c_{ij}=\langle \alpha_j,\alpha_i^\vee\rangle$ ($1\leq 
i,j\leq n$) are just the entries of the Cartan matrix of $\cR$, so
this yields an explicit system of $n$ equations which we need to solve for
$\xi_1,\ldots,\xi_n$. Let us describe this explicitly in all cases, where 
we refer to the labelling of the simple roots in Table~\ref{Mdynkintbl}. 
(This will also be relevant in Section~\ref{sec:regemb}.) 
\begin{itemize}
\item[$A_n$:] $\bG\cong\SL_{n+1}(k)$ and $\bZ(\bG)=\{h(\xi,\xi^2,
\xi^3,\ldots,\xi^n)\mid \xi^{n+1}=1\}$; this is contained in the subtorus 
$\bS:=\{h(\xi, \xi^2,\xi^3,\ldots, \xi^n)\mid \xi\in\bkm\}\subseteq \bT$.
\item[$B_n$:] $\bG\cong\mbox{Spin}_{2n+1}(k)$. If $n\geq 2$ is even, then 
\[ \bZ(\bG)=\{h(1,\xi,1,\xi,1,\xi,1,\ldots)\mid \xi^2=1\}\]
is contained in $\bS:=\{h(1,\xi,1,\xi,1,\xi,1,\ldots)\mid 
\xi\in \bkm\}$. For $n\geq 3$ odd, 
\[ \bZ(\bG)=\{h(\xi,1,\xi,1,\xi,1,\xi,\ldots)\mid \xi^2=1\}\]
is contained in $\bS:=\{h(\xi,1,\xi,1,\xi,1,\xi,\ldots)\mid \xi\in 
\bkm\}$.
\item[$C_n$:] $\bG\cong\mbox{Sp}_{2n}(k)$ and $\bZ(\bG)=\{h(\xi,1,1,1,
\ldots)\mid \xi^2=1\}$; this is contained in $\bS:=\{h(\xi,1,1,1,\ldots)\mid 
\xi\in \bkm\}$.
\item[$D_n$:] $\bG\cong \mbox{Spin}_{2n}(k)$. If $n\geq 4$ is even, then
\[\bZ(\bG)=\{h(\xi_1,\xi_2,1,\xi_1\xi_2,1,\xi_1\xi_2,1,\xi_1\xi_2,\ldots)
\mid \xi_1^2=\xi_2^2=1\}\]
is contained in $\bS:=\{h(\xi_1,\xi_2,1, \xi_1\xi_2,1, 
\xi_1\xi_2,\ldots )\mid \xi_1,\xi_2 \in\bkm\}$. If 
$n\geq 3$ is odd, then 
\[\bZ(\bG)=\{h(\xi,\xi^{-1},\xi^2,1,\xi^2,1,\xi^2,1,\xi^2, \ldots)\mid 
\xi^4=1\}\]
is contained in $\bS:=\{h(\xi,\xi^{-1},\xi^2,1,\xi^2,1,\xi^2,\ldots)
\mid \xi\in\bkm\}$.
\item[$G_2$:] Since $\det(c_{ij})=1$, we have $\bZ(\bG)=\{1\}$; we set 
$\bS:=\{1\}$.
\item[$F_4$:] Since $\det(c_{ij})=1$, we have $\bZ(\bG)=\{1\}$; we set 
$\bS:=\{1\}$.
\item[$E_6$:] $\bZ(\bG)=\{h(\xi,1,\xi^{-1},1,\xi,\xi)\mid \xi^3{=}1\}\subseteq
\bS:=\{h(\xi,1,\xi^{-1},1, \xi,\xi)\mid \xi\in\bkm\}$.
\item[$E_7$:] $\bZ(\bG)=\{h(1,\xi,1,1,\xi,1,\xi)\mid \xi^2{=}1\}\subseteq
\bS:=\{h(1,\xi,1,1, \xi,1,\xi)\mid \xi\in\bkm\}$.
\item[$E_8$:] Since $\det(c_{ij})=1$, we have $\bZ(\bG)=\{1\}$; we set 
$\bS:=\{1\}$.
\end{itemize}
In each case, $\bS$ is a subtorus of $\bT$ such that $\bZ(\bG)
\subseteq \bS$. We have $\dim \bS\leq 1$ except for type $D_n$ with 
$n\geq 4$ even, in which case $\dim \bS=2$. Note that, in order to obtain 
these descriptions, we did not have to rely on explicit realisations of 
groups of classical type as matrix groups: the abstract setting in terms 
of root data has actually been more efficient in this context. 
\end{exmp}

For the construction of isogenies between groups of the same Cartan type, 
the following remarks will be useful. 

\begin{abs} \label{st9160} Let $\bG_1$, $\bG_2$ be connected reductive 
algebraic groups over $k$. For $i=1,2$, let $\cR_i=(X_i,R_i,Y_i,R_i^\vee)$ 
be the corresponding root datum, relative to a maximal torus $\bT_i 
\subseteq \bG_i$. Furthermore, let us choose Borel subgroups $\bB_i
\subseteq \bG_i$ such that $\bT_i \subseteq\bB_i$. By 
Remark~\ref{MrootdatumG2}, this is equivalent to choosing bases $\Pi_i 
\subseteq R_i$. We assume that $X_1,X_2$ have the same rank and that 
$R_1,R_2$ have the same Cartan matrix $C=(c_{st})_{s,t \in S}$. If we 
also choose $\Z$-bases of $X_1$ and $X_2$, then $\cR_1$ and $\cR_2$ 
are determined by factorisations as in Remark~\ref{Mcorbrulu1}:
\[\cR_1:\quad C=\breve{A}_1\cdot A_1^{\text{tr}} \qquad\mbox{and}\qquad 
\cR_2:\quad C=\breve{A}_2\cdot A_2^{\text{tr}},\]
where $A_1$, $A_2$, $\breve{A}_1$, $\breve{A}_2$ are integer matrices, all 
of the same size. Now, this setting gives rise to a correspondence:
\[\left\{\begin{array}{c} \mbox{Isogenies $f\colon \bG_1\rightarrow\bG_2$ 
with} \\ \mbox{$f(\bT_1)=\bT_2$, $f(\bB_1)=\bB_2$}\end{array}\right\}
\quad\leftrightarrow\quad \left\{\begin{array}{c} \mbox{Pairs of integer
matrices $(P,P^\circ)$}\\\mbox{satisfying (MI1), (MI2) in 
\ref{Mdefmatrixisog}} \end{array}\right\}.\]
(Note that, here, the relations in (MI2) read $PA_2^{\text{tr}}=
A_1^{\text{tr}}P^\circ$ and $P^\circ \breve{A}_2=\breve{A}_1P$.)
Indeed, by \ref{Mabs23}, each isogeny of groups on the left determines a 
$p$-isogeny of root data; since $f(\bB_1)=\bB_2$, this $p$-isogeny is
``base-preserving'' as in \ref{Mdefmatrixisog} and, hence, it determines
a unique pair of matrices on the right. Conversely, a pair of matrices on 
the right determines a $p$-isogeny of root data by \ref{Mdefmatrixisog};
by the Isogeny Theorem~\ref{thmiso}, there is a corresponding isogeny 
of groups on the left, which is unique up to inner automorphisms given by 
elements of $\bT_1$. 
\end{abs}

\begin{prop} \label{st916a} Let $\bG$ be semisimple and assume that the 
root datum of~$\bG$ (relative to some maximal torus $\bT$ and some
Borel subgroup $\bB$ containing $\bT$) is of Cartan type~$C$. Let $\bGsc$ 
and $\bGad$ be of Cartan type $C$, as in Example~\ref{Madjscexmp}, relative
to $\tilde{\bT}\subseteq \tilde{\bB}\subseteq \bGsc$ and $\bT'
\subseteq \bB'\subseteq \bGad$. Then there exist central isogenies 
\[ \tilde{f}\colon \bGsc \longrightarrow\bG \qquad \mbox{and}\qquad
f'\colon \bG \longrightarrow \bGad,\]
such that $\tilde{f}(\tilde{\bT})=\bT$, $\tilde{f}(\tilde{\bB})=\bB$,
$f'(\bT)=\bT'$, $f'(\bB)=\bB'$.
\end{prop}

\begin{proof} First consider $\bGsc$. We place ourselves in the setting of 
\ref{st9160} where $(\bG_1,\bT_1,\bB_1)=(\bGsc,\tilde{\bT},\tilde{\bB})$ 
and $(\bG_2,\bT_2,\bB_2)=(\bG,\bT,\bB)$. The root datum of 
$\bG_1$ is given by a factorisation of $C$ as above where $A_1=
C^{\text{tr}}$ and $\breve{A}_1=I$ (identity matrix). The only extra 
information about the root datum of $\bG_2$ is that, in the factorisation 
$C=\breve{A}_2 \cdot A_2^{\text{tr}}$, both $A_2$, $\breve{A}_2$ are square 
matrices (since $\bG$ is semisimple). In order to find $\tilde{f} \colon 
\bG_1\rightarrow \bG_2$, we need to specify a pair of (square) integral 
matrices $(\tilde{P}, \tilde{P}^\circ)$ where $\tilde{P}A_2^{\text{tr}}=
C\tilde{P}^{\circ}$, $\tilde{P}^\circ \breve{A}_2 =\tilde{P}$ and 
$\tilde{P}^\circ$ is a monomial matrix whose non-zero entries are powers 
of~$p$. (Note that $\tilde{P}$ will be automatically invertible over $\Q$.)
There is a natural choice for such a pair, namely, $(\tilde{P},
\tilde{P}^\circ):=(\breve{A}_2,I)$. Thus, the correspondence in 
\ref{st9160} yields the existence of $\tilde{f}$. The root exponents of 
$\tilde{f}$ (which are the non-zero entries of $\tilde{P}^\circ=I$) are 
all equal to~$1$, hence $\tilde{f}$ is a central isogeny.

Now consider $\bGad$. We argue as above, where now $(\bG_1,\bT_1,\bB_1)=(\bG,
\bT,\bB)$ and $(\bG_2,\bT_2,\bB_2)=(\bGad,\bT',\bB')$. The root datum of 
$\bG_2$ is given by $C=\breve{A}_2\cdot A_2^\trp$ where $A_2=I$ and 
$\breve{A}_2= C$. We need to specify a pair of (square) integral matrices 
$(P',P'^\circ)$ where $P'=A_1^{\text{tr}} P'^{\circ}$, $P'^\circ C=
\breve{A}_1P'$ and $P'^\circ$ is a monomial matrix
whose non-zero entries are powers of~$p$. Again, there is a natural choice 
for such a pair, namely, $(P',P'^\circ):=(A_1^{\text{tr}}, I)$. As above, 
this yields the existence of $f'$.
\end{proof}

\begin{prop}[Steinberg \protect{\cite[9.16]{St68}}] \label{st916} Let
$\bG$ be semisimple and consider central isogenies $\tilde{f}\colon \bGsc
\rightarrow \bG$ and $f'\colon \bG \rightarrow\bGad$ as in 
Proposition~\ref{st916a}. Assume, furthermore, that $F\colon \bG\rightarrow 
\bG$ is an isogeny. Then the following hold.
\begin{itemize}
\item[(a)] The isogeny $F$ lifts to $\bGsc$; more precisely, there exists 
a unique isogeny $\tilde{F}\colon \bGsc\rightarrow \bGsc$ such that 
$F\circ \tilde{f}=\tilde{f}\circ \tilde{F}$.  
\item[(b)] The isogeny $F$ descends to $\bGad$; more precisely, there is a 
unique isogeny $F' \colon \bGad\rightarrow \bGad$ such that $F' \circ f'=
f'\circ F$. 
\end{itemize}
In both cases, the root exponents of $\tilde{F}$ and of $F'$ are equal to 
those of $F$. If, moreover, $F$ is a Steinberg map, then so are $F'$ and 
$\tilde{F}$.
\end{prop}

\begin{proof} First note that $F'$, if it exists, is clearly unique. 
The uniqueness of $\tilde{F}$ (if it exists) is shown as follows. Let 
$F_1\colon \bGsc\rightarrow \bGsc$ be another isogeny such that $F\circ 
\tilde{f}=\tilde{f}\circ F_1$. Then the map sending $g\in \bGsc$ to 
$\tilde{F}(g)F_1(g)^{-1}$ is a homomorphism of algebraic groups from 
$\bGsc$ to the center of $\bGsc$. Since $\bGsc$ is connected and the 
center of $\bGsc$ is finite, that map must be constant and so $\tilde{F}(g)
F_1(g)^{-1}=1$ for all $g\in \bGsc$. We now turn to the problem of showing
the existence of $\tilde{F}$ and $F'$.

Let $\bT\subseteq \bB\subseteq \bG$, $\tilde{\bT}\subseteq \tilde{\bB}
\subseteq \bGsc$, $\bT'\subseteq \bB'\subseteq \bGad$ be as in 
Proposition~\ref{st916a}. We consider the corresponding root data of $\bG$, 
$\bGsc$, $\bGad$, and write $X=X(\bT)$, $\tilde{X}=X(\tilde{\bT})$, $X'=
X(\bT')$. Then $\tilde{f}$ induces a $p$-isogeny $\tilde{\varphi}\colon 
X\rightarrow \tilde{X}$ and $f'$ induces a $p$-isogeny $\varphi'\colon X'
\rightarrow X$. Thus, we are in the setting of \ref{st9160}. Now consider 
the isogeny $F\colon \bG \rightarrow \bG$. As already pointed out in the 
proof of \cite[9.16]{St68}, one easily sees that (a) and (b) hold for $F$ 
if and only if (a) and (b) hold for $\iota_g\circ F$, where $\iota_g$ is 
an inner automorphism of $\bG$ (for any $g\in\bG$). Hence, using 
Theorem~\ref{MrootdatumG1} and replacing $F$ by $\iota_g\circ F$ for a 
suitable $g$, we may assume without loss of generality that $F(\bT)=\bT$ 
and $F(\bB)=\bB$. Then $F$ induces a $p$-isogeny $\Phi\colon X\rightarrow 
X$ and, again, we are in the setting of \ref{st9160}. Now, if we can show 
that there exist $p$-isogenies $\tilde{\Phi} \colon \tilde{X} \rightarrow 
\tilde{X}$ and $\Phi'\colon X'\rightarrow X'$ such that 
\[\tilde{\varphi} \circ \Phi= \tilde{\Phi} \circ \tilde{\varphi}\qquad
\mbox{and}\qquad \varphi'\circ \Phi'=\Phi\circ \varphi',\]
then the existence of $\tilde{F}$ and $F'$ follows from a general result 
about isogenies, which can already be found in \cite[\S 18.4]{Ch05} and 
which is a step in the proof of the Isogeny Theorem~\ref{thmiso} (see also 
\cite[3.2, 3.3]{St2}). In order to see how $\tilde{\Phi}$ and $\Phi'$ can 
be constructed, we use the correspondence in \ref{st9160} to describe 
everything in terms of pairs of square integral matrices. (The following
part of the proof is somewhat different from the original proof in
\cite{St68}.)

Let $(\tilde{P}, \tilde{P}^\circ)$ and $(P',P'^\circ)$ be the pairs of 
matrices corresponding to $\tilde{\varphi}$ and $\varphi'$, respectively. 
Furthermore, let $(Q, Q^\circ)$ be the pair of matrices corresponding 
to~$\Phi$. Now recall that the root data of $\bGsc$ and $\bGad$ are given 
by the factorisations $C=I \cdot (C^{\text{tr}})^{\text{tr}}$ and $C=C
\cdot I^{\text{tr}}$, respectively. Let $A$, $\breve{A}$ be square integral 
matrices such that the root datum of $\bG$ is given by the factorisation
$C=\breve{A}\cdot A^{\text{tr}}$. Then the conditions in 
\ref{Mdefmatrixisog} imply that 
\[\tilde{P}=\tilde{P}^\circ\breve{A}, \qquad P'=A^{\text{tr}}P'^\circ,
\qquad QA^{\text{tr}}=A^{\text{tr}}Q^\circ, \qquad 
Q^\circ \breve{A}=\breve{A}Q,\]
where $\tilde{P}^\circ$ and $P'^\circ$ are monomial matrices all of whose
non-zero entries are equal to~$1$ (since $\tilde{f}$ and $f'$ are central). 

Assume first that $\tilde{\Phi}$ exists and let $(\tilde{Q},\tilde{Q}^\circ)$ 
be the corresponding pair of matrices. Since the root datum of $\bGsc$ is 
given by the factorisation $C=I \cdot (C^{\text{tr}})^{\text{tr}}$, the
conditions in \ref{Mdefmatrixisog} imply that $\tilde{Q}=\tilde{Q}^\circ$. 
Since $\tilde{\varphi}\circ\Phi=\tilde{\Phi}\circ\tilde{\varphi}$, we must
have $\tilde{P}Q= \tilde{Q}\tilde{P}$. Using the above relations, we deduce 
that 
\[\tilde{Q}=\tilde{Q}^\circ=\tilde{P}^\circ Q^\circ (\tilde{P}^\circ)^{-1}.\] 
Thus, if $\tilde{\Phi}$ exists, then $(\tilde{Q}, \tilde{Q}^\circ)$ is
determined by $Q^\circ$ and $\tilde{P}^\circ$; in particular, the 
root exponents of $\tilde{\Phi}$ are equal to those of $\Phi$. Conversely, 
it is straightforwad to check that the map $\tilde{\Phi} \colon \tilde{X} 
\rightarrow\tilde{X}$ defined by the matrix $\tilde{Q}$ given by the above 
formula has the required properties. Thus, (a) is proved.

Similarly, assume first that $\Phi'$ exists and let $(Q',Q'^\circ)$ 
be the corresponding pair of matrices. Then one deduces that 
\[Q'=Q'^\circ=(P'^\circ)^{-1} Q^\circ P'^\circ.\] 
Conversely, one checks that the map $\Phi' \colon X' \rightarrow X'$ 
defined by the matrix $Q'$ given by the above formula has the required 
properties. Thus, (b) is proved.

Finally, assume that $F$ is a Steinberg map. Then, using the characterisation
in Proposition~\ref{Mgenfrob}, one easily sees that $F'$ and $\tilde{F}$
are also Steinberg maps. 
\end{proof}


\begin{prop}[Cf.\ \protect{\cite[p.~46]{St}}] \label{MdirprodG} Let $C$ be a 
Cartan matrix and $\bGsc$, $\bGad$ be as in Example~\ref{Madjscexmp}. Assume 
that $C$ is a block diagonal matrix, with diagonal blocks $C_1,\ldots,C_n$.
Then we have
\[ \bGsc=\tilde{\bG}_1 \times \ldots \times \tilde{\bG}_n
\qquad \mbox{and} \qquad \bGad=\bG_1'\times \ldots \times \bG_n',\]
where $\tilde{\bG}_i\subseteq \bGsc$ and $\bG_i'\subseteq \bGad$ are the 
normal subgroups corresponding to the various diagonal blocks $C_i$, as in 
Remark~\ref{MdirprodG1}. For each $i$, the group $\tilde{\bG}_i$ is simple
of simply-connected type $C_i$ and $\bG_i'$ is simple of adjoint type $C_i$.
\end{prop}

\begin{proof} The definition of $\cRsc(C)$ shows that this root datum is
the direct sum of the root data $\cRsc(C_1),\ldots,\cRsc(C_n)$. Hence,
the assertion concerning $\bGsc$ immediately follows from 
Remark~\ref{MdirprodG1}. The argument for $\bGad$ is analogous.
\end{proof}

\begin{abs} \label{Mssquot} We shall assume from now on that 
$\bG$ is connected reductive over $k=\overline{\F}_p$ (where $p$ is a 
prime number) and $F\colon \bG\rightarrow \bG$ is a Steinberg map.
Let $\bZ$ be the center of $\bG$ and $\bGder$ be the derived 
subgroup of $\bG$. Clearly, we have $F(\bZ)=\bZ$ and $F(\bGder)=\bGder$. 
Since $\bG= \bZ^\circ.\bGder$ and $\bZ^\circ\cap \bGder$ is finite, 
we obtain isogenies
\[\renewcommand{\arraystretch}{1.2} \begin{array}{ccc}
\bZ^\circ \times  \bGder &\longrightarrow &\bG\\
 (z,g) &\mapsto &zg\end{array}\qquad\mbox{ and }\qquad
\begin{array}{ccc}
\bG &\longrightarrow & \bG/\bGder\times \bG/\bZ^\circ \\ 
g&\mapsto &(g\bGder,\;g\bZ^\circ).\end{array}\]
(Note that these are maps between groups of the same dimension; the
first map is clearly surjective and, hence, has a finite kernel; the
second map has a finite kernel and, hence, is surjective.) Recall from 
\ref{subsec15} that $\bGder$ is semisimple. The group $\bG/\bGder$ is a 
torus. (For, it is connected, abelian and consists of elements of order 
prime to $p$; see \ref{subsec14}.) Furthermore, $\bGss:=\bG/\bZ^\circ$ is 
reductive (see Lemma~\ref{Mgenfrob2}) with a finite center and, hence, is 
semisimple. Using also the isogenies in Proposition~\ref{st916a}, we obtain 
isogenies
\[\bZ^\circ \times (\bGder)_{\text{sc}}\;\longrightarrow\; \bG\;
\longrightarrow\; \bG/\bGder\times (\bGss)_{\text{ad}}.\]
Now, by Lemma~\ref{Mgenfrob2}, we have induced Steinberg maps on $\bGss$ 
and on $\bG/\bGder$. By Proposition~\ref{st916}, there are also induced 
Steinberg maps on $(\bGder)_{\text{sc}}$ and on $(\bGss)_{\text{ad}}$. 
Since all these maps are induced and uniquely determined by $F$, we will 
now simplify our notation and denote all these induced maps by $F$ as 
well. Using Proposition~\ref{Msameorder}(c), we conclude that
$|\bGder^F|=|(\bGder)_{\text{sc}}^F|$, $|\bGss^F|=|(\bGss)_{\text{ad}}^F|$
and 
\[ |\bG^F|= |(\bZ^\circ)^F| |\bGder^F|=|(\bG/\bGder)^F||\bGss^F|.\]
Also note that the natural map $\bG\rightarrow \bGss$ induces a surjective
map $\bG^F\rightarrow \bGss^F$ (since the kernel of $\bG\rightarrow
\bGss$ is connected).
\end{abs}

\begin{rem} \label{Madjquot} By a slight abuse of notation, we shall denote 
$(\bGss)_{\text{ad}}$ simply by $\bGad$. Thus, as above, we obtain a central
isogeny $\bGss\rightarrow \bGad$ which commutes with the action 
of $F$ on both sides. Composing this isogeny with the natural map $\bG
\rightarrow \bGss= \bG/\bZ^\circ$, we obtain a surjective, central 
isotypy
\[\pi_{\text{ad}}\colon \bG\rightarrow \bGad\qquad \mbox{with} \qquad 
\ker(\pi_{\text{ad}})=\bZ,\]
which commutes with the action of $F$ on both sides and which we call an 
\nm{adjoint quotient} of $\bG$. We certainly have an inclusion 
$\pi_{\text{ad}}(\bG^F) \subseteq \bGad^F$ but, in general, equality 
will not hold. By Proposition~\ref{Msameorder}, $\pi_{\text{ad}}(\bG^F)$ 
is a normal subgroup of $\bGad^F$ and we have an isomorphism 
\[ \bGad^F/\pi_{\text{ad}}(\bG^F)\cong \bZ/\cL(\bZ),\]
where $\cL\colon \bG\rightarrow \bG$, $g\mapsto g^{-1}F(g)$. One easily 
sees that $\bZ/\cL(\bZ)=(\bZ/\bZ^\circ)_F$, where the subscript $F$ 
denotes ``$F$-coinvariants'', that is, the largest quotient on which $F$ 
acts trivially. The above isomorphism is explicitly obtained by sending 
$g\in \bGad^F$ to $\dot{g}^{-1}F(\dot{g}) \in \bZ$ where $\dot{g} \in 
\bG$ is any element satisfying $\pi_{\text{ad}}(\dot{g})=g$. 

Also note that each $g\in\bGad^F$ defines an automorphism $\alpha_g\colon
\bG^F \rightarrow \bG^F$, $g_1\mapsto \dot{g} g_1\dot{g}^{-1}$ (where, as
above, $\dot{g} \in \bG$ is such that $\pi_{\text{ad}}(\dot{g})=g$; the 
map $\alpha_g$ obviously does not depend on the choice of~$\dot{g}$). In 
this way, we obtain a group homomorphism 
\[ \bGad^F \rightarrow \mbox{Aut}(\bG^F), \qquad g\mapsto \alpha_g.\]
The automorphisms $\alpha_g$ are called \nm{diagonal automomorphisms} of
$\bG^F$.
\end{rem}


\begin{rem} \label{Msccover} Again, by a slight abuse of notation,
we shall denote $(\bGder)_{\text{sc}}$ simply by $\bGsc$. Thus, as above,
we obtain a central isotypy
\[ \pi_{\text{sc}}\colon \bGsc \rightarrow \bG \qquad \mbox{with} 
\qquad \pi_{\text{sc}}(\bGsc)=\bGder,\]
which commutes with the action of $F$ on both sides and which we call 
a \nm{simply-connected covering} of the derived subgroup of $\bG$. 
We certainly have $\pi_{\text{sc}}(\bGsc^F) \subseteq \bGder^F$ but, again, 
equality will not hold in general. In fact, we have:
\begin{itemize}
\item[(a)] $\pi_{\text{sc}}(\bGsc^F)=\langle u\in \bG^F\mid 
u\mbox{ unipotent} \rangle\subseteq \bGder^F$ (\cite[12.6]{St68}),
and 
\item[(b)] $\bG^F/\pi_{\text{sc}} (\bGsc^F)$ is abelian of order prime 
to $p$ (\cite[1.23]{DeLu}).
\end{itemize} 
Note that $\pi_{\text{sc}}(\bGsc^F)$ is a characteristic subgroup of 
$\bG^F$; furthermore, we have the inclusions $[\bG^F,\bG^F] \subseteq 
\pi_{\text{sc}} (\bGsc^F)\subseteq \bGder^F$ but each of these may be 
strict, as can be seen already in the example where $\bG=\bGder=
\mbox{PGL}_2(k)$. 

Following \cite[p.~45]{St}, \cite[9.1]{St68}, we 
call $\ker(\pi_{\text{sc}})$ the \nm{fundamental group} of~$\bG$. For 
example, if $\bG$ is simple of simply-connected type, then the fundamental 
group of $\bG$ is trivial (but the converse is not necessarily true). If 
$\bG$ is simple of adjoint type, with corresponding Cartan matrix $C$, 
then $\ker(\pi_{\text{sc}}) \cong \Hom(\Lambda(C), \bkm)$ where 
$\Lambda(C)$ is the fundamental group of $C$; see 
Remark~\ref{Mcartanfund} and Example~\ref{Madjscexmp}.
\end{rem}

\begin{abs} \label{MdirprodG3} We keep the above notation. As already stated
in \ref{subsec15}, we have $\bGder=\bG_1 \ldots \bG_n$ where $\bG_1,
\ldots,\bG_n$ are the closed normal simple subgroups of $\bG$. (They
elementwise commute with each other.) Now, one complication of the theory 
arises from the fact that, in general, this product decomposition
is not stable under the Steinberg map $F\colon\bGder \rightarrow \bGder$. 
What happens is the following. Consider the set-up in 
Remark~\ref{MdirprodG1}, with partitions $\Pi=\Pi_1\sqcup
\ldots \sqcup\Pi_n$ and $R=R_1 \sqcup\ldots \sqcup R_n$; then 
$\bG_i=\langle \bU_\alpha \mid \alpha\in R_i\rangle$ for $i=1,\ldots,n$. 
Now, $F$ will permute the simple 
subgroups $\bG_i$ and, correspondingly, the permutation $\alpha \mapsto 
\alpha^\dagger$ of $R$ (induced by $F$) will permute the subsets $R_i$. 
Hence, there is an induced permutation $\rho$ of $\{1,\ldots,n\}$ such that, 
for all $i=1,\ldots,n$, we have $R_{\rho(i)} =\{\alpha^\dagger \mid 
\alpha \in R_i\}$ and 
\begin{equation*}
F(\bG_i)=\langle F(\bU_ \alpha)\mid \alpha\in R_i\rangle=
\langle \bU_{\alpha^\dagger} \mid \alpha\in R_i\rangle=\bG_{\rho(i)}.\tag{a}
\end{equation*}
Following \cite[p.~78]{St68}, we say that $\bGder$ is 
\nms{$F$-simple}{F-simple} if $\rho$ is a cylic permutation (it has a 
single orbit). Thus, by grouping together the factors in the various 
$\rho$-orbits on $\{1,\ldots,n\}$, we can write $\bGder$ as a product of 
various $F$-stable and $F$-simple semisimple groups. An analogous statement 
holds for $\bGss$. Indeed, for each~$i$, let $\bar{\bG}_i$ be the image of 
$\bG_i$ under the canonical map $\bG\rightarrow \bGss$. Then we have $\bGss=
\bar{\bG}_1 \ldots \bar{\bG}_n$ and the induced Steinberg map $F\colon \bGss
\rightarrow \bGss$ permutes the factors $\bar{\bG}_i$ according to the
permutation $\rho$. 

Now consider the isogenies in \ref{Mssquot} and the groups $\bGsc$, 
$\bGad$. By Proposition~\ref{MdirprodG}, we have direct product 
decompositions
\begin{equation*}
\bGsc=\tilde{\bG}_1\times \ldots \times \tilde{\bG}_n\qquad
\mbox{and} \qquad \bGad=\bG_1'\times \ldots \times \bG_n'\tag{b}
\end{equation*}
such that, under the isogeny $\bGsc\rightarrow \bGder$, the 
factor $\tilde{\bG}_i$ is mapped to $\bG_i$ and, under the isogeny
$\bGss\rightarrow \bGad$, the factor $\bar{\bG}_i$ is mapped to $\bG_i'$.
By the compatibility of all of the above isogenies with the various
Steinberg maps involved, it follows that
\begin{equation*}
F(\tilde{\bG}_i)=\tilde{\bG}_{\rho(i)} \qquad \mbox{and}
\qquad F(\bG_i')=\bG_{\rho(i)}' \qquad \mbox{for $i=1,\ldots,n$}.\tag{c}
\end{equation*}
The following result deals with $F$-simple semisimple groups. (If $F$
is a Frobenius map with respect to some $\F_q$-rational structure, then
the construction below is related to the operation ``{\em restriction of 
scalars}''; see \cite[\S 11.4]{Spr}.)
\end{abs}

\begin{lem} \label{Fsimpleorder} Assume that $\bG$ is semisimple and
$F$-simple, as defined in \ref{MdirprodG3}. Then $\bG=\bGder=\bG_1\ldots 
\bG_n$ as above. Assume that this product is an abstract direct product. 
Then $F^n(\bG_1)=\bG_1$ and 
\[\iota\colon\bG_1\rightarrow \bG,\qquad g\mapsto gF(g)\cdots F^{n-1}(g),\] 
is an injective homomorphism of algebraic groups which restricts to an 
isomorphism $\bG_1^{F^n}\cong \bG^{F}$. Furthermore, if $q$ is the positive 
real number attached to $(\bG,F)$ (see \ref{MstdcalC}), then $q^n$ is the 
positive real number attached to $(\bG_1,F^n)$.
\end{lem}

\begin{proof} By assumption, $F$ cyclically permutes the factors $\bG_i$.
So we can choose the labelling such that $F^i(\bG_1)=\bG_{i+1}$ for 
$i=1,\ldots, n-1$  and $F^n(\bG_1)=\bG_1$. The map $\iota$ clearly
is a morphism of affine varieties, which is injective since the product
is direct. This map is a group homomorphism because the groups $\bG_i$ 
elementwise commute with each other. For the same reason, we have $\iota
(\bG_1^{F^n}) \subseteq \bG^F$. Since we have a direct product, every 
$g\in \bG$ can be written uniquely as $g=g_1\cdots g_n$ with $g_i \in 
\bG_i$. Then $F(g)=g$ if and only if $F^i(g_1)=g_{i+1}$ for $i=1,\ldots,
n-1$ and $F^n(g_1)=g_1$. Thus, $F(g)=g$ if and only if $g=\iota(g_1)$ 
and $F^n(g_1)=g_1$. The statement concerning $q$, $q^n$ immediately 
follows from the definition of these numbers in Proposition~\ref{Mdefnq}.
\end{proof}

\begin{cor} \label{Mdirprodfin} Assume that $\bG$ is connected reductive;
let $\bG=\bZ^\circ.\bGder$ and $\bGder=\bG_1\ldots\bG_n$, as above. Let 
$F\colon \bG\rightarrow\bG$ be a Steinberg map and $I\subseteq \{1,\ldots,
n\}$ be a set of representatives of the $\rho$-orbits on $\{1,\ldots,n\}$, 
with $\rho$ as in \ref{MdirprodG3}(a). For each $i\in I$, let $n_i$ be the 
length of the corresponding $\rho$-orbit. Then we have isomorphisms
(of abstract finite groups)
\[\bGsc^F\cong\prod_{i\in I}\tilde{\bG}_i^{F^{n_i}}
\qquad \mbox{and} \qquad \bGad^F\cong \prod_{i\in I} \bG_i'^{F^{n_i}},\]
where $\bGsc=\tilde{\bG}_1\times \ldots \times \tilde{\bG}_n$ and 
$\bGad=\bG_1'\times \ldots \times \bG_n'$ are as in \ref{MdirprodG3}(b).
\end{cor}

\begin{proof} This is immediate from \ref{MdirprodG3} and 
Lemma~\ref{Fsimpleorder}. Also recall our identification in \ref{Mssquot} 
of the various Steinberg maps involved.
\end{proof}

See Example~\ref{Mstrangexp} for a good illustration of the above result;
the group $\bG=\SL_2(k)\times \PGL_2(k)$ (with $\mbox{char}(k)=2$) 
considered there is $F$-simple! In particular, the simple 
factors of an $F$-simple semisimple group need not all be 
isomorphic to each other as algebraic groups. (They are 
isomorphic as abstract groups.)


To complete this section, we discuss another fundamental construction
involving the formalism of root data: ``dual groups''.

\begin{defn}[\protect{\cite[5.21]{DeLu}}; see also 
\protect{\cite[\S 4.3]{Ca2}}, \protect{\cite[8.4]{LuB}}] \label{Mdefdual} 
Consider two pairs $(\bG,F)$ and $(\bG^*,F^*)$ where $\bG,\bG^*$ are 
connected reductive and $F\colon \bG\rightarrow \bG$, $F^*\colon \bG^*
\rightarrow \bG^*$ are Steinberg maps. We say that $(\bG,F)$ and $(\bG^*,F^*)$
are in \nm{duality} if there is a \nm{maximally split} torus $\bT_0
\subseteq \bG$ and a {\it maximally split} torus $\bT_0^*\subseteq \bG^*$ 
such that the following conditions hold, where $\cR=(X,R,Y,R^\vee)$ is the 
root datum of $\bG$ (with respect to $\bT_0)$ and $\cR^*=(X^*,R^*,Y^*,
{R^*}^\vee)$ is the root datum of $\bG^*$ (with respect to $\bT_0^*$):
\begin{itemize}
\item[(a)] There is an isomorphism $\delta\colon X \rightarrow Y^*$ such 
that $\delta(R)={R^*}^\vee$ and 
\[\langle \lambda,\alpha^\vee\rangle=\langle \alpha^*,\delta(\lambda)\rangle
\qquad \mbox{for all $\lambda\in X(\bT_0)$ and $\alpha\in R$},\]
where $\alpha^*\in R$ is defined by $\delta(\alpha)={\alpha^*}^\vee$.
\item[(b)] If $\lambda\colon \bT_0\rightarrow \bk^\times$ (an element of $X=
X(\bT_0)$) and $\nu\colon \bk^\times \rightarrow \bT_0^*$ (an element of 
$Y^*=Y(\bT_0^*)$) correspond to each other under~$\delta$, then 
$\lambda\circ F\colon \bT_0 \rightarrow \bk^\times$ and $F^*\circ \nu
\colon \bk^\times \rightarrow \bT_0^*$ also correspond to each other 
under~$\delta$.
\end{itemize}
The relation of being {\it in duality} is symmetric: the above conditions 
on $\delta\colon X\rightarrow Y^*$ are equivalent to analogous conditions 
concerning a map $\varepsilon\colon Y\rightarrow X^*$ obtained by
transposing $\delta$; see \cite[4.2.2, 4.3.1]{Ca2}. In particular, 
$\delta$ defines an isomorphism of root data between $\cR^*$ and the dual 
root datum considered in Lemma~\ref{Mweylfinite1}(b). Thus, connected 
reductive groups in duality have dual root data. Note also that the
dual of $(\bG^*,F^*)$ can be naturally identified with $(\bG,F)$.
\end{defn}

\begin{abs} \label{Mdefdualmat} It may be worthwhile to reformulate the 
above definition in terms of the matrix language of 
Section~\ref{sec:rootdata}; the following discussion will also show that,
for any given pair $(\bG,F)$, there exists a corresponding dual pair 
$(\bG^*,F^*)$.  

So let $\bG$ be connected reductive and $F\colon \bG \rightarrow \bG$ 
be a Steinberg map. Let $\bT_0\subseteq \bG$ be a maximally split torus 
and $\bB_0\subseteq \bG$ be an $F$-stable Borel subgroup such that 
$\bT_0\subseteq \bB_0$. Let $\cR=(X,R,Y,R^\vee)$ be the root datum of 
$\bG$ with respect to $\bT_0$; let $\Pi$ be the base for $R$ determined
by $\bB_0$ (see Remark~\ref{MrootdatumG2}). As in Remark~\ref{Mcorbrulu1},
we choose a $\Z$-basis of $X$. Then $\cR$ determines a factorisation 
\[ C=\breve{A}\cdot A^{\trp}\qquad\mbox{where}\qquad C=\mbox{ Cartan matrix
of $\cR$}.\]
Furthermore, as discussed in \ref{st9160},
the isogeny $F\colon \bG\rightarrow \bG$ determines a pair of integer 
matrices $(P,P^\circ)$ satisfying the conditions (MI1), (MI2) in 
\ref{Mdefmatrixisog}. Now we notice that $C^\trp$ also is a Cartan matrix
(see Remark~\ref{Mcorbrulu1}). Hence, transposing the above matrix
equation, we obtain a factorisation
\[C^\trp=\breve{B}\cdot B^\trp \qquad \mbox{where} \qquad \breve{B}:=A
\quad\mbox{and}\quad B:=\breve{A}.\]
Furthermore, setting $Q:=P^\trp$ and $Q^\circ:=(P^\circ)^\trp$, the
pair $(Q,Q^\circ)$ satisfies the conditions (MI1), (MI2) with respect
to $C^\trp=\breve{B}\cdot B^\trp$. Now the latter factorisation of
$C^\trp$ determines a root datum $\cR^*$ and a base for its root system
(see Lemma~\ref{Mcartan2}). One easily sees that these are independent
of the choice of the $\Z$-basis for $X$. Hence, going the above argument 
backwards, there exists a connected reductive group $\bG^*$ such that $\cR^*$ 
(together with the base of its root system) is isomorphic to the root 
datum of $\bG^*$ with respect to a maximal torus $\bT_0^*\subseteq \bG^*$ 
and a Borel subgroup $\bB_0^*\subseteq \bG^*$ containing $\bT_0^*$. The pair 
of matrices $(Q,Q^\circ)$ determines an isogeny $F^*\colon \bG^*\rightarrow 
\bG^*$ such that $\bT_0^*$ and $\bB_0^*$ are $F$-stable. Since $F$ is a 
Steinberg map, the characterisation in Proposition~\ref{Mgenfrob} 
immediately shows that $F^*$ also is a Steinberg map. The number $q$ in 
Proposition~\ref{Mdefnq} is the same for $F$ and for $F^*$.
\end{abs}

\begin{rem} \label{Mdualtori} Let $\bT,\bT^*$ be tori and assume
that we are given Steinberg maps $F\colon \bT\rightarrow \bT$ and 
$F^*\colon \bT^*\rightarrow \bT^*$. Then $\bT,\bT^*$ are connected
reductive groups (with empty root systems). Hence, $(\bT,F)$ and 
$(\bT^*,F^*)$ are in duality if and only if there exists an isomorphism 
$\delta\colon X(\bT) \rightarrow Y(\bT^*)$ such that condition (b) in 
Definition~\ref{Mdefdual} holds. For future reference we note that, if
$(\bT,F)$ and $(\bT^*,F^*)$ are in duality as above, then there is a
corresponding (non-canonical) isomorphism of abelian groups
\[ {\bT^*}^{F^*} \stackrel{\sim}{\longrightarrow} \Irr(\bT^F),
\qquad s \mapsto \theta_s,\]
where $\Irr(\bT^F)$ denotes the set of complex irreducible characters 
of $\bT^F$ (which is a group under the tensor product of characters, since 
$\bT^F$ is abelian); see \cite[4.4.1]{Ca2}. The construction of the above 
isomorphism depends on some choices, e.g., the choice of an embedding
$\overline{\F}_p^\times \hookrightarrow \C^\times$. (This will be discussed
later in further detail.)
\end{rem}

We shall have to say more about groups in duality in later sections. It
will be useful and important to know how properties of $\bG$ translate 
or connect to properties of $\bG^*$. The lemma below contains just one 
example. (See Lusztig \cite{Lu09d} where a number of such ``bridges'' 
of a much deeper nature are discussed.)

\begin{lem} \label{Mconncent} Let $(\bG,F)$ and $(\bG^*,F^*)$ be two
pairs in duality, as in Definition~\ref{Mdefdual}. Then the following
conditions are equivalent.
\begin{itemize}
\item[(i)] The center of $\bG$ is connected.
\item[(ii)] The abelian group $X/\Z R$ has  no $p'$-torsion.
\item[(iii)] The fundamental group of $\bG^*$ (see Remark~\ref{Msccover})
is trivial.
\end{itemize}
\end{lem}

\begin{proof} For the equivalence of (i) and (ii), see \cite[4.5.1]{Ca2}.
The equivalence of (ii) and (iii) is shown in \cite[4.5.8]{Ca2}.
\end{proof}

\begin{exmp} \label{Mdefdualbeisp} (a) Let $\bG=\GL_n(k)$. Then
$\bG$ has Cartan type $A_{n-1}$ and the Cartan matrix $C$ factorises as
$C=\breve{A}\cdot A^\trp$ where $A=\breve{A}$; see Example~\ref{rootdatGL}. 
Hence the discussion in \ref{Mdefdualmat} immediately shows that $\bG$ 
is dual to itself.

(b) Assume that $\bG$ is semisimple of adjoint type. In the setting of 
\ref{Mdefdualmat}, this means that the Cartan matrix factorises as 
$C=\breve{A}\cdot A^\trp$ where $A$ is the identity matrix and 
$\breve{A}=C$. Hence, $C^{\trp}=\breve{B}\cdot B^\trp$ where $\breve{B}$ 
is the identity matrix and $B=C$. Thus, $\bG^*$ is seen to be semisimple 
of simply-connected type. If $C$ is symmetric, then
Proposition~\ref{st916a} yields a central isogeny $\bG^*\rightarrow \bG$.

(c) The examples in (a), (b) seem to indicate that dual groups are related
in quite a strong way. However, as pointed out in the introduction of 
\cite{Lu09d}, dual groups in general are related only through a very weak 
connection (via their root system); in particular there is no direct, 
elementary construction which produces $\bG^*$ from~$\bG$. Perhaps the most 
striking example is the case where $\bG=\mbox{SO}_{2n+1}(k)$. Then $\bG$ 
is simple of adjoint type, with Cartan matrix $C$ of type $B_n$. As in 
(b), $\bG^*$ will be simple of simply-connected type. However, since 
$C^\trp$ has type $C_n$, we see that $\bG^*\cong \mbox{Sp}_{2n}(k)$. If 
$\mbox{char}(k)\neq 2$, then there is not even any abstract non-trivial
group homomorphism between $\bG$ and~$\bG^*$!
\end{exmp}

\section{Generic finite reductive groups} \label{sec:generic}

Recall from Definition~\ref{mydeffr} that a finite group of Lie type 
is a finite group of the form $G=\bG^F$, where $\bG$ is a connected reductive 
algebraic group over $k=\overline{\F}_p$ and $F\colon \bG\rightarrow \bG$ is 
a Steinberg map. Then it is common to speak of the (twisted or untwisted) 
``type'' of $\bG^F$: for example, we say that the finite general linear
groups are of untwisted type $A_{n-1}$, the finite unitary groups are of
twisted type $A_{n-1}$ (denoted ${^2\!}A_{n-1}$; see
Example~\ref{Mquasisplit}), or that the Suzuki groups are of ``very 
twisted'' type $B_2$ (denoted ${^2\!}B_2$; see Example~\ref{Msuzuki2}). Here, 
the superscript (as in ${^2\!}A_{n-1}$) indicates the order of the 
automorphism of the Weyl group of $\bG$ which is induced by $F$; in 
particular, $\bG^F$ is of ``untwisted'' type if $F$ induces the identity 
map on the Weyl group. Using the machinery developed in the previous 
sections, we can now give a somewhat more precise definition, as follows. 

\begin{abs} \label{MstdcalC} Assume that $\bG$ is connected reductive and
let $F\colon \bG \rightarrow \bG$ be a Steinberg map. Then we can 
canonically attach to $\bG$ and $F$ a pair 
\[\cC(\bG,F):=(C,P^\circ)\]
consisting of a Cartan matrix $C=(c_{st})_{s,t \in S}$ and a monomial matrix 
$P^\circ=(p_{st})_{s,t\in S}$ whose non-zero entries are positive powers of 
$p$ and such that $C P^\circ=P^\circ C$. Let us recall how this is done. 
First, we choose a \nm{maximally split} torus $\bT_0\subseteq \bG$. Recall
from Example~\ref{Mfrobconj} that this means that $\bT_0$ is an $F$-stable 
maximal torus of $\bG$ which is contained in an $F$-stable Borel subgroup 
$\bB_0\subseteq \bG$. (By Proposition~\ref{canfrob11}, the pair $(\bT_0,
\bB_0)$ is unique up to conjugation by elements of $\bG^F$.) Let $\cR=(X,
R, Y,R^\vee)$ be the root datum of $\bG$ relative to $\bT_0$ and $\varphi 
\colon X \rightarrow X$ be the $p$-isogeny induced by~$F$. By 
Remark~\ref{MrootdatumG2}, there is a unique base $\Pi$ of $R$ such 
that $\bB_0=\langle \bT_0, \bU_\alpha\mid \alpha\in R^+ \rangle$ where
$R^+$ are the positive roots with respect to~$\Pi$. Let us write 
$\Pi=\{\alpha_s\mid s\in S\}$ and let $C=(c_{st})_{s,t\in S}$ be the 
corresponding Cartan matrix. 

Since $\varphi$ is a $p$-isogeny, there is a permutation $\alpha\mapsto 
\alpha^\dagger$ of $R$ such that $\varphi (\alpha^\dagger)=q_\alpha\,
\alpha$ for all $\alpha\in R$. The fact that $\bB_0$ is $F$-stable implies 
that this permutation leaves $R^+$ invariant. Hence, this permutation will 
also leave the base $\Pi$ invariant and so there is an induced permutation 
$S\rightarrow S$, $s \mapsto s^\dagger$, such that $\alpha_s^\dagger=
\alpha_{s^\dagger}$ for all $s\in S$. Thus, $\varphi$ is ``base preserving''
as in \ref{Mdefmatrixisog} and we have a corresponding monomial matrix
$P^\circ=(p_{st}^\circ)_{s,t \in S}$ whose non-zero entries are given by 
$p_{s s^\dagger}^\circ=q_s:=q_{\alpha_s}$ for all $s\in S$. The condition 
(MI2) implies that $CP^\circ=P^\circ C$, which means that 
\begin{equation*}
q_{t}c_{st}=q_s c_{s^\dagger t^\dagger} \qquad\mbox{for all $s,t\in S$}.
\tag{a}
\end{equation*}  
Since all pairs $(\bT_0,\bB_0)$ as above are conjugate by elements of 
$\bG^F$, the pair $(C, P^\circ)$ is uniquely determined by $\bG,F$ up to 
relabeling the elements of $S$. 

Now consider the Weyl group $\bW$ of $\cR$. Recall from Remark~\ref{MidentW}
that we identify $S$ with a subset of $\bW$ via $s \leftrightarrow 
w_{\alpha_s}$; thus, we have $\bW=\langle S \rangle$. Let $l\colon \bW
\rightarrow \Z_{\geq 0}$ be the corresponding length function. By 
Remark~\ref{MdefisogR1}, the $p$-isogeny $\varphi$ induces a group 
automorphism $\sigma\colon \bW\rightarrow \bW$ such that 
\begin{equation*}
\sigma(s)=s^\dagger \quad (s \in S) \qquad \mbox{and} \qquad 
\varphi\circ \sigma(w)=w \circ \varphi \quad (w\in \bW). \tag{b}
\end{equation*}  
In Theorem~\ref{MrootdatumG}, we have seen that we can naturally identify 
$\bW=N_{\bG}(\bT_0)/\bT_0$. (Under this identification, the reflection 
$w_\alpha \in \bW$ corresponds to the element $\dot{w}_\alpha \in N_{\bG}
(\bT_0)$ in \ref{subsec1roots}.) Since $\bT_0$ and, hence, $N_{\bG}(\bT_0)$ 
are $F$-stable, $F$ naturally induces an automorphism $\sigma_F\colon \bW
\rightarrow \bW$, $g\bT_0\mapsto F(g) \bT_0$ ($g\in N_{\bG}(\bT_0)$). It 
is straighforward to check that all of the above constructions and 
identifications are compatible, that is, we have $\sigma_F(w)=\sigma(w)$ 
for all $w\in\bW$. 


Finally, the numbers $\{q_s\}$ satisfy the following conditions. If 
$S_1,\ldots, S_r$ are the orbits of the permutation $s\mapsto s^\dagger$ 
on $S$, then 
\begin{equation*}
q^{|S_i|}=\prod_{s \in S_i} q_s \qquad (i=1,\ldots,r),\tag{c}
\end{equation*}  
where $q>0$ is the real number defined in Proposition~\ref{Mgenfrob}. 
(This easily follows from the equation $\varphi(\alpha_{s^\dagger})=
q_s \alpha_s$ for all $s\in S$; see \cite[11.17]{St68}.) Hence, we also 
have $q^{|S|}=\prod_{s\in S} q_s$, which provides an alternative
characterisation of $q$. 
\end{abs}

\begin{abs} \label{MstdcalC1} As in \cite[3.1]{LuB}, we say that 
$\dagger$ (or $\sigma\colon \bW \rightarrow \bW$) is {\em ordinary} if the 
following condition is satisfied: whenever $s\neq t$ in $S$ are in the 
same $\dagger$-orbit on $S$, then the order of the product $st$ is $2$ or 
$3$. With this notion, we have the following distinction of cases. The 
group $\bG^F$ is 
\begin{itemize}
\item either ``untwisted'', that is, $\dagger$ is the identity (and 
$q_s=q$ for all $s\in S$); 
\item or ``twisted'', that is, $\dagger$ is not the identity but 
ordinary (as defined above); 
\item or ``very twisted'', otherwise.
\end{itemize}
The typical examples to keep in mind are: the finite general linear groups 
(untwisted), the finite general unitary groups (twisted) and the finite 
Suzuki groups (very twisted). Note that these are notions which depend on 
$\bG$ and $F$ used to define $\bG^F$, not just on the finite group $\bG^F$: 
In Example~\ref{Mstrangexp}, there is a realisation of $\SL_2(q)$ (where 
$q$ is a power of $2$) as a twisted (but not very twisted) group.
\end{abs}

\begin{rem} \label{Luordi} Assume that $F$ is a Frobenius map, with respect 
to some $\F_q$-rational structure on $\bG$. Then, as pointed out by Lusztig 
\cite[3.4.1]{LuB}, the induced automorphism $\sigma\colon \bW \rightarrow 
\bW$ is {\em ordinary} in the sense defined above.

This is seen as follows. Let $s\neq t$ in $S$ be in the same $\dagger$-orbit.
Replacing $F$ by a power of $F$ if necessary, we can assume without loss of 
generality that $t=s^\dagger$. Now, applying $\dagger$ repeatedly to $\{s,
t\}$, we obtain a whole $\dagger$-orbit of pairs $\{s',t'\}$ where $s'\neq 
t'$ are in $S$. Then \ref{MstdcalC}(a) shows that $c_{s't'}\neq 0$ for all 
these pairs.  So, if this $\dagger$-orbit of pairs had more than one element,
then we would obtain a closed path in the Dynkin diagram of $C$, which is 
impossible (see Table~\ref{Mdynkintbl}, p.~\pageref{Mdynkintbl}). Hence, we 
must have $t=s^\dagger$ and $s= t^\dagger$. But then \ref{MstdcalC}(a) 
implies that $q_tc_{st}=q_s c_{s^\dagger t^\dagger}=q_sc_{ts}$. However, by 
Lemma~\ref{Mgenfrob3b}, we have $q_s=q_t=q$. Hence, $c_{st}=c_{ts}$ and 
so $c_{st}\in\{0,-1\}$, which means that $st$ has order $2$ or $3$, as 
required.
\end{rem}

\begin{abs} \label{Mgeneralord} We now have all the ingredients to state the
order formula for $\bG^F$. Since $\bB_0$ is $F$-stable, the unipotent 
radical $\bU_0=R_u(\bB_0)$ is also $F$-stable. Since $\bB_0=\bU_0.\bT_0$ 
where $\bU_0\cap \bT_0=\{1\}$, we obtain
\[|\bG^F|=|\bU_0^F|\cdot |\bT_0^F|\cdot |\bG^F/\bB_0^F|.\]
The factor $|\bT_0^F|$ is evaluated as follows. As in 
Proposition~\ref{Mdefnq}, we extend scalars from $\Z$ to $\R$ and 
consider the induced linear map $\varphi_\R \colon X_\R \rightarrow X_\R$
where $X_\R=\R\otimes_\Z X$. We have $\varphi_\R=q\varphi_0$ where
$\varphi_0\in \GL(X_\R)$ is a linear map of finite order. Then 
\[|\bT_0^F|=\varepsilon_\varphi\,\det(\varphi -\id_X)=\det(q\,\id_{X_\R}-
\varphi_0^{-1}),\]
where $\varepsilon_\varphi\in\{\pm 1\}$ is the sign such that 
$\varepsilon_\varphi\det(\varphi)>0$. In particular, this shows that the 
order of $\bT_0^F$ is obtained from the characteristic polynomial of 
$\varphi_0^{-1}$ by evaluation at $q$. This is discussed in detail in 
\cite[\S 3.3]{Ca2}, \cite[\S 25.1]{MaTe}. 

A further evaluation of the remaining two factors in the above expression
for $|\bG^F|$ leads to the following formula, due to Chevalley \cite{Chev} 
(in the case where $\bG$ is semisimple and $\dagger$ is the identity) and 
Steinberg \cite[\S 11]{St68} (in general). 
\end{abs}

\begin{thm}[Order formula] \label{Morderform} With the above notation, 
we have 
\[|\bG^F|=q^{|R|/2}\det(q\, \id_{X_\R}-\varphi_0^{-1})\sum_{w \in \bW^\sigma}
q^{l(w)},\]
where $\bW^\sigma=\{w\in \bW\mid \sigma(w)=w\}$ (which is a finite 
Coxeter group). When $\bG$ is simple, explicit formulae for the various
possibilities are given as in Table~\ref{Mordertab}.
\end{thm}

The fact that the list in Table~\ref{Mordertab} exhausts all
the possible pairs $(\bG,F)$ where $\bG$ is simple and $F\colon \bG
\rightarrow\bG$ is a Steinberg map is shown in \cite[\S 11.6]{St68}; 
note that, in this case, we have $|\bG^F|=|\bGad^F|=|\bGsc^F|$ (see 
\ref{Mssquot}). 

\begin{table}[htbp] 
\caption{Order formulae for $|\bG^F|$ when $\bG$ is simple} \label{Mordertab} 
\centerline{\small $\renewcommand{\arraystretch}{1.3}
\begin{array}{ll} \hline \mbox{Type} & |\bG^F|\\\hline 
A_{n{-}1} & q^{n(n{-}1)/2}(q^2{-}1)(q^3{-}1) \cdots (q^n{-}1)\\
B_n & q^{n^2}(q^2{-}1)(q^4{-}1) \cdots (q^{2n}{-}1) \\
C_n & \text{same as $B_n$} \\
D_n & q^{n^2{-}n}(q^2{-}1)(q^4{-}1) \cdots 
(q^{2n{-}2}{-}1)(q^n{-}1) \\
G_2 & q^6(q^2{-}1)(q^6{-}1)\\
F_4 & q^{24}(q^2{-}1)(q^6{-}1)(q^8{-}1)(q^{12}{-}1)\\
E_6 & q^{36}(q^2{-}1)(q^5{-}1)(q^6{-}1)(q^8{-}1)(q^9{-}1)
(q^{12}{-}1)\\
E_7 & q^{63}(q^2{-}1)(q^6{-}1)(q^8{-}1)(q^{10}{-}1)(q^{12}{-}1)
(q^{14}{-}1)(q^{18}{-}1)\\
E_8 & q^{120}(q^2{-}1)(q^8{-}1)(q^{12}{-}1)(q^{14}{-}1)
(q^{18}{-}1)(q^{20}{-}1)(q^{24}{-}1)(q^{30}{-}1)\\\hline
{^2\!A}_{n{-}1} & q^{n(n{-}1)/2}(q^2{-}1)(q^3{+}1) \cdots 
(q^n{-}({-}1)^n)\\
{^2\!D}_n & q^{n^2{-}n}(q^2{-}1)(q^4{-}1) \cdots 
(q^{2n{-}2}{-}1)(q^n{+}1) \\ 
{^3\!D}_4 & q^{12}(q^2{-}1)(q^6{-}1)(q^8{+}q^4{+}1)\\
{^2\!E}_6 & q^{36}(q^2{-}1)(q^5{+}1)(q^6{-}1)(q^8{-}1)
(q^9{+}1)(q^{12}{-}1)\\\hline
{^2\!B}_2 & q^4(q^2{-}1)(q^4{+}1)
\qquad\qquad\qquad\qquad\qquad\quad (q=\sqrt{2}^{2m{+}1}) \\
{^2\!G}_2 & q^6(q^2{-}1)(q^6{+}1)
\qquad\qquad\qquad\qquad\qquad\quad (q=\sqrt{3}^{2m{+}1}) \\ 
{^2\!F}_4 & q^{24}(q^2{-}1)(q^6{+}1)(q^8{-}1)(q^{12}{+}1)
\qquad \qquad\; (q=\sqrt{2}^{2m{+}1}) \\\hline  
\multicolumn{2}{l}{\text{(The first $9$ are ``untwisted'', the next 
$4$ ``twisted'' and the last $3$ ``very twisted''.)}} 
\end{array}$}
\end{table}

\begin{rem} \label{Morderform1} (a)  The formula shows that $q^{|R|/2}$ is
the $p$-part of the order of $\bG^F$; this provides a further 
characterisation of the number $q$. We also note that the above expression 
for $|\bG^F|$ can be interpreted as a polynomial in one variable evaluated 
at $q$, where the polynomial only depends on the root datum of $\bG$ and the 
maps $\varphi_0$, $\sigma$ derived from $F$.  This will be formalised 
in Definition~\ref{MdefBrMa} below.

(b) Steinberg \cite[14.14]{St68} shows that the number of $F$-stable maximal
tori of $\bG$ is equal to $q^{|R|}$. This equality is in fact equivalent
to the following Molien series identity which yields another expression for 
the order of $\bG^F$:
\[\frac{q^{|R|}}{|\bG^F|}=\frac{1}{|\bW|} \sum_{w\in \bW} \frac{1}{\det(q\,
\id_{X_\R}- w\varphi_0^{-1})};\]
see \cite[\S 3.4]{Ca2}, \cite[Exc.~30.15]{MaTe} for further details. One 
advantage of this expression is that it does not involve the length 
function on $\bW$ or the induced automorphism $\sigma$ of $\bW$. The inverse
of the sum on the right hand side can be further expressed as a product
of various cyclotomic polynomials; in this way, one obtains the familiar 
formulae for the order of $\bG^F$ when $\bG$ is simple; see
Table~\ref{Mordertab}. Also note that, instead of considering the maps 
$w\circ \varphi_0^{-1}\colon X_\R \rightarrow X_\R$ in the above formula, 
one can take their transposes $(w\circ \varphi_0^{-1})^\trp\colon Y_\R
\rightarrow Y_\R$ where $Y_\R=\R\otimes_\Z Y$: the characteristic 
polynomials will certainly remain the same. The same is true when we 
replace $w\circ \varphi_0^{-1}$ by $\varphi_0^{-1}\circ w$.

(c) If $F$ is a Frobenius map, then Theorem~\ref{Morderform} can be proved
by a general argument; see \cite[4.2.5]{mybook}. The general case (where
the root exponents $q_s$ may not all be equal) is treated in 
\cite[\S 11]{St68} (see also \cite[\S 24.1]{MaTe}), assuming that $\bG$ 
is semisimple and ``$F$-simple'' (see \ref{MdirprodG3}). But, as 
already noted in \cite[p.~78]{St68}, the case of an arbitrary connected 
reductive $\bG$ can be easily recovered from this case. As an illustration
of the methods developed in the previous section, let us explicitly work 
out the reduction argument.
\end{rem}

\begin{lem} \label{Mreductord} Suppose that the formula in 
Theorem~\ref{Morderform} is known to hold when $\bG$ is simple of 
adjoint type. Then the formula holds in general. 
\end{lem}

\begin{proof} Since the formula for $|\bT_0^F|$ in \ref{Mgeneralord} is
already known to hold in general, it will be sufficient to consider the 
cardinality of $\bG^F/\bT_0^F$. It will be convenient to slightly rephrase 
this as follows. Let us consider the set of cosets $\bG/\bT_0=\{g\bT_0\mid 
g\in \bG\}$ (just as an abstract set, we don't need the notion of a quotient
variety here). Since $\bT_0$ is $F$-stable, we have an induced action 
of $F$ on $\bG/\bT_0$. Consequently, we have a natural injective map 
$\bG^F/\bT_0^F\rightarrow (\bG/\bT_0)^F$, $g\bT_0^F\mapsto g\bT_0$. Now 
the connected group $\bT_0$ acts transitively on $g\bT_0$ by right 
multiplication. Hence, if $g\bT_0$ is $F$-stable, then 
Proposition~\ref{Mfrobconj0} shows that $g\bT_0$ contains a representative 
fixed by~$F$. It follows that the above map is surjective and so 
$|\bG^F/\bT_0^F|=|(\bG/\bT_0)^F|$. Thus, it will now be sufficient to 
consider the identity:
\begin{equation*}
|(\bG/\bT_0)^F|=\bO(\bW,\sigma,q) \qquad \mbox{where} \qquad \bO(\bW,
\sigma,q):=q^{|R|/2} \sum_{w\in \bW^\sigma} q^{l(w)}.\tag{$*$}
\end{equation*}
Since the formula in Theorem~\ref{Morderform} is assumed to hold when 
$\bG$ is simple of adjoint type, the same is true of the formula ($*$).
We must deduce from this that ($*$) holds in general. We do this in two 
steps.

1) First assume that $\bG$ is semisimple of adjoint type. As in
\ref{MdirprodG3}(b), we have a direct product decomposition
$\bG=\bG_1\times \ldots \times \bG_n$ where each $\bG_i$ is simple
of adjoint type. Furthermore, there is a permutation $\rho$ of $\{1,
\ldots,n\}$ such that $F(\bG_i)=\bG_{\rho(i)}$ for $i=1,\ldots,n$. Now
note that $\bT_0=\bT_1\times \ldots \times \bT_n$ where each $\bT_i$
(for $i=1,\ldots,n$) is a maximal torus of $\bG_i$ such that $F(\bT_i)=
\bT_{\rho(i)}$. Hence, we can also identify $\bG/\bT_0=\bG_1/\bT_1\times 
\ldots \times \bG_n/\bT_n$. Furthermore, if $I\subseteq \{1,\ldots, n\}$ 
and $n_i$ ($i \in I$) are as in Corollary~\ref{Mdirprodfin}, then 
$F^{n_i}(\bG_i/\bT_i)=\bG_i/\bT_i$ for all $i\in I$ and 
\begin{equation*}
|(\bG/\bT_0)^F|=\prod_{i\in I}|(\bG_i/\bT_i)^{F^{n_i}}|.\tag{1a}
\end{equation*}
It remains to show that there is a similar factorisation of the right
hand side of ($*$). Recall from \ref{MdirprodG3} that we have a partition 
$R=R_1\sqcup\ldots\sqcup R_n$. Consequently, we also have a direct product 
decomposition 
\[\bW=\bW_1\times \ldots \times \bW_n\qquad \mbox{where} \qquad
\bW_i:=\langle w_\alpha \mid \alpha\in R_i\rangle;\]
furthermore, $\sigma(\bW_i)=\bW_{\rho(i)}$ for $i=1,\ldots,n$. Here, $\bW_i$
is the Weyl group of the factor $\bG_i$ (relative to $\bT_i\subseteq 
\bG_i$). Now, if $w\in \bW$ and $w=w_1 \cdots w_n$ with $w_i\in \bW_i$ for 
all $i$, then $l(w)=l(w_1)+\ldots + l(w_n)$. Using this formula, it is
straightforward to verify that the expression for $\bO(W,\sigma,q)$ is 
compatible with the above product decomposition, that is, we have 
$\sigma^{n_i}(\bW_i)=\bW_i$ for all $i\in I$ and
\begin{equation*}
\bO(\bW,\sigma,q)=\prod_{i\in I} \bO(\bW_i,\sigma^{n_i},q^{n_i}).\tag{1b}
\end{equation*}
By assumption and Lemma~\ref{Fsimpleorder}, we have $|(\bG_i/
\bT_i)^{F^{n_i}}|=\bO(\bW_i,\sigma^{n_i}, q^{n_i})$ for all $i\in I$. 
Hence, comparing (1a) and (1b), we see that ($*$) holds for $\bG$ as well.

2) Now let $\bG$ be arbitrary (connected and reductive). As in
Remark~\ref{Madjquot}, we consider an adjoint quotient $\pi_{\text{ad}}
\colon \bG \rightarrow\bGad$ with kernel $\bZ=\bZ(\bG)$. By 
Remark~\ref{Msemisimple}(a), we have $\bZ\subseteq \bT_0$; furthermore, 
$\bT':=\pi_{\text{ad}}(\bT_0)$ is an $F$-stable maximal torus of $\bGad$ 
(see \ref{Mconnredcomp}(a)). So we get a bijective map 
\[\bG/\bT_0\rightarrow \bGad/\bT',\qquad g\bT_0\mapsto \pi_{\text{ad}}
(g)\bT',\]
which is compatible with the action of $F$ on $\bG/\bT_0$ and on $\bGad/
\bT'$. In particular, $|(\bG/\bT_0)^F|=|(\bGad/\bT')^F|$ and so the left
hand side of ($*$) does not change when we pass from $\bG$ to $\bGad$. 
On the other hand, by \ref{Mconnredcomp}(d), $\pi$ induces an 
$F$-equivariant isomorphism from the Weyl group of $\bG$ (relative to 
$\bT_0$) onto the Weyl group of $\bGad$ (relative to $\bT'$). Hence, the 
right hand side of ($*$) does not change either when we pass from $\bG$ 
to $\bGad$. Thus, if ($*$) holds for $\bGad$, then ($*$) also holds for 
$\bG$.
\end{proof}

Following \cite{BrMa}, we now formally introduce ``series of finite groups 
of Lie type''. This relies on the following definition, which is a slight 
modification of that in \cite[\S 1]{BrMa}. (See Example~\ref{MexpBrMa}
below for further comments on this.)

\begin{defn} \label{MdefBrMa} Let $\cR=(X,R,Y,R^\vee)$ be a root datum, 
with Weyl group $\bW \subseteq \Aut(X)$. We set $X_\R:=\R \otimes_\Z X$. We
can canonically regard $X$ as a subset of $X_\R$; thus, we also have 
$\bW \subseteq \GL(X_\R)$. Let $\varphi_0\in \GL(X_\R)$ be an invertible 
linear map of finite order which normalises $\bW$, and assume that 
\[ \cP=\cP_\G:=\left\{q\in \R_{>0} \,\Big|\,\begin{array}{l} 
\mbox{$q\varphi_0(X) \subseteq X$ and the corresponding map}\\
\mbox{$q\varphi_0\colon X\rightarrow X$ is a $p$-isogeny for some 
prime $p$} \end{array}\right\}\]
is non-empty. We form the coset $\varphi_0\bW \subseteq \GL(X_\R)$. Then 
\[ \G=\bigl((X,R,Y,R^\vee),\varphi_0\bW\bigr)\]
is called a \nm{complete root datum} or a \nm{generic finite reductive 
group}. We define a corresponding rational function $|\G|\in \R(y)$ 
(where $y$ is an indeterminate) by 
\[\frac{y^{|R|}}{|\G|}=\frac{1}{|\bW|} \sum_{w\in \bW} \frac{1}{\det(y\,
\id_{X_\R}- w\varphi_0^{-1})}.\]
We call $|\G|$ the \nm{order polynomial} of $\G$; this will be justified 
in Remark~\ref{MdefBrMa3} below. Note that $\cP$ is an infinite set: 
If $q\in \cP$ and $q\varphi_0$ is a $p$-isogeny (where $p$ is a prime), 
then $p^mq\in\cP$ for all integers $m\geq 1$. 
\end{defn}

\begin{exmp} \label{MexpBrMa} 
Let $\bG$ be connected reductive and $F\colon \bG\rightarrow \bG$ be 
a Steinberg map. Then we obtain a corresponding complete root datum by
taking the root datum of $\bG$ (relative to an $F$-stable maximal torus
$\bT\subseteq \bG$) together with the linear map $\varphi_0$ defined in 
Proposition~\ref{Mdefnq}(b). In particular, this includes all the cases 
discussed in Examples~\ref{Mquasisplit}, \ref{Msuzuki2}, \ref{expdelu1}.
This shows that the above Definition~\ref{MdefBrMa} is somewhat more 
general than that in \cite{BrMa}, in which cases like those in 
Example~\ref{expdelu1} are not included. 
\end{exmp}

Let us now fix a complete root datum $\G=\bigl((X,R,Y,R^\vee),\varphi_0\bW
\bigr)$.

\begin{rem} \label{MdefBrMa0} Let $q\in \cP$ and set $\varphi:=
q\varphi_0$. Then $\varphi$ is a $p$-isogeny (for some prime $p$) and 
we have $\varphi^d=q^d\id_X$, where $d\geq 1$ is the order of $\varphi_0$. 
In particular, this implies that $q^d=p^m$ for some $m\geq 1$. Let $\bG$ be 
a connected reductive algebraic group over $k=\overline{\F}_p$ whose root 
datum (relative to a maximal torus $\bT\subseteq \bG$) is isomorphic to 
$(X,R,Y,R^\vee)$. Then $\varphi$ gives rise to an isogeny $F\colon \bG 
\rightarrow \bG$ which is a Steinberg map by Proposition~\ref{Mgenfrob}. We 
write $\G(q):=\bG^F$. Thus, we obtain a family of finite groups 
\[ \{\G(q) \mid q \in \cP\}\]
which we call the \nm{series of finite groups of Lie type} defined by $\G$.
There are some choices involved in the definition of $\G(q)$ but we shall 
see in the remarks below that different choices lead to isomorphic
finite groups.
\end{rem}

\begin{rem} \label{MdefBrMa1} Since $\varphi_0$ normalises $\bW$, we obtain
a group automorphism $\sigma \colon \bW \rightarrow \bW$ such that 
\[\sigma(w)=\varphi_0^{-1}w \varphi_0 \qquad \mbox{for all $w \in \bW$}.\]
Note that this is compatible with Remark~\ref{MdefisogR1}: For any 
$q\in \cP$, the automorphism of $\bW$ induced by the $p$-isogeny 
$q\varphi_0$ (where $p$ is a prime) is given by $\sigma$. Let us now see 
what happens when we replace $\varphi_0$ by another map in the coset 
$\varphi_0\bW$. First note that $\varphi_0w$ has finite order for any 
$w \in \bW$. Furthermore, if $q\in \cP$ and $q\varphi_0$ is a $p$-isogeny 
(where $p$ is a prime), then $(q\varphi_0)w$ also is a $p$-isogeny. 
Thus, if $\varphi_0$ satisfies the defining conditions for a complete root 
datum, then so does $\varphi_0 w$ for any $w \in \bW$. We also note
the following identity:
\[ (\varphi_0 w)^m=\varphi_0^m \cdot \bigl(\sigma^{m-1}(w)\cdots
\sigma^2(w)\sigma(w) w\bigr) \qquad \mbox{for all $m \geq 1$}.\] 
Now let $q \in \cP$ and let $\bG,\bT,F$ be as in Remark~\ref{MdefBrMa0}.
Then $\varphi:=q\varphi_0$ is the linear map induced on $X\cong X(\bT)$ 
by~$F$. Let $w \in \bW$ and $\dot{w}$ be a representative of $w$ in 
$N_{\bG}(\bT)$. We define $F'\colon \bG\rightarrow \bG$ by $F'(g):=
\dot{w}^{-1}F(g)\dot{w}$ for $g \in \bG$. By Lemma~\ref{canfrob0a}, $F'$ 
also is a Steinberg map and we have $\bG^{F'}\cong \bG^F$. Now $\bT$ is 
$F'$-stable and one easily sees that $\varphi w \colon X
\rightarrow X$ is the linear map induced by~$F'$. (See, for example, 
\cite[3.3.4]{Ca2}.) This shows that, if we replace $\varphi_0$ by 
$\varphi_0w$ for some $w \in \bW$, then $F$ changes to $F'$ but we obtain 
isomorphic finite groups.
\end{rem}

\begin{rem} \label{MdefBrMa2} Let $q\in\cP$ and set $\varphi:=q\varphi_0$. 
Then $\varphi$ is a $p$-isogeny (where $p$ is a prime) and so there is a 
corresponding permutation $\alpha\mapsto \alpha^\dagger$ of $R$; we have 
$\varphi(\alpha^\dagger)=q_\alpha \alpha$ for all $\alpha\in R$, where 
$\{q_\alpha\}$ are the root exponents of $\varphi$. Let $\sigma \colon
\bW \rightarrow \bW$ be the group automorphism in Remark~\ref{MdefBrMa1}. 
By Remark~\ref{MdefisogR1}, we have $\sigma(w_\alpha)=w_{\alpha^\dagger}$ 
for all $\alpha \in R$; also recall that the root exponents are positive.
Hence, we conclude that the permutation $\alpha\mapsto \alpha^\dagger$ 
only depends on $\varphi_0$, but not on $q$.
\end{rem}

\begin{rem} \label{MdefBrMa3} Let us fix a base $\Pi$ of $R$. Since any
two bases can be tranformed into each other by a unique element of $\bW$,  
there is a unique $w \in \bW$ such that, if we replace $\varphi_0$ by
$\varphi_0':= \varphi_0w$, then $\Pi^\dagger=\Pi$ where $\alpha \mapsto
\alpha^\dagger$ is the permutation induced by $\varphi_0'$; see 
Remarks~\ref{Misosimple} and \ref{MdefBrMa2}. Assume now that this is the 
case. Let $q \in \cP$ and $\bG,\bT,F$ be as in Remark~\ref{MdefBrMa0}.
Then $\bT$ lies in the $F$-stable Borel subgroup $\bB=\langle \bT, 
\bU_\alpha \mid \alpha \in R^+\rangle$ (where $R^+$ are the positive 
roots with respect to $\Pi$) and we are in the setting of \ref{MstdcalC}, 
where $\bT_0:= \bT$. So, by Remark~\ref{Morderform1}, we have $|\bG^F|=
|\G|(q)$ and this is also equal to the expression in 
Theorem~\ref{Morderform}. Since this holds for all
$q\in\cP$, we obtain an identity of rational functions in~$y$:
\begin{equation*}
|\G|=y^{|R|/2}\det(y\, \id_{X_\R}-\varphi_0^{-1})\sum_{w \in \bW^\sigma}
y^{l(w)}. \tag{a}
\end{equation*}
Thus, the rational function $|\G|$ actually is a polynomial in $y$ such 
that $|\bG^F|=|\G|(q)$. This provides the justification for calling $|\G|$
the {\em order polynomial} of~$\G$ (see also \cite[1.12]{BrMa}). Now
let $K\subseteq \R$ be a subfield such that $\det(y\, \id_{X_\R}-
\varphi_0^{-1})\in K[y]$. Since $\varphi_0$ has finite order, all 
eigenvalues of this polynomial are roots of unity. By \cite[2.1]{St68}, 
an analogous result is also true for the term $\sum_{w} y^{l(w)}$ in (a).
So there is a factorisation
\begin{equation*}
|\G|=y^{|R/2|}\times \mbox{ product of cyclotomic polynomials in 
$K[y]$}.\tag{b}
\end{equation*}
If $\bG$ is simple, then such factorisations can be seen explicitly in 
Table~\ref{Mordertab} (p.~\pageref{Mordertab}).
\end{rem}

\begin{rem} \label{MdefBrMa0a} Let $\Pi=\{\alpha_s\mid s\in S\}$ be a base
of $R$ and $C=(c_{st})_{s,t\in S}$ be the corresponding Cartan matrix; also 
choose a $\Z$-basis of $X$. Then $\cR$ is determined by a factorisation 
$C=\breve{A}\cdot A^{\text{tr}}$ as in Remark~\ref{Mcorbrulu1}. Assume that 
$\varphi_0$ is chosen such that the permutation of $R$ induced by $\varphi_0$
leaves $\Pi$ invariant (which is possible by Remark~\ref{MdefBrMa3}). Let 
$Q$ be the matrix of $\varphi_0\colon X_\R \rightarrow X_\R$ (with respect 
to the chosen basis of $X$). If $q\in \cP$, 
then $q\varphi_0$ is a $p$-isogeny (for some prime $p$) and the conditions 
(MI1), (MI2) in \ref{Mdefmatrixisog} show that $qQ A^{\text{tr}}=
A^{\text{tr}}P^\circ$ and $P^\circ \breve{A}= q \breve{A}Q$, where 
$P^\circ$ is a monomial matrix whose non-zero entries are all powers of~$p$. 
It follows that 
\[ Q A^{\text{tr}}=A^{\text{tr}}Q^\circ \qquad \mbox{and} \qquad Q^\circ 
\breve{A}= \breve{A}Q,\]
where $Q^\circ:=q^{-1}P^\circ$ is a monomial matrix; each non-zero entry 
of $Q^\circ$ is a positive real number such that some positive power of it 
is an integral power of~$p$. Now note that, although the pair $(P,P^\circ)$
is used in the construction, $Q^\circ$ is uniquely determined by $Q$ and, 
hence, independent of $(P,P^\circ)$. There are two cases:
\begin{itemize}
\item[(I)] All non-zero entries of $Q^\circ$ are equal to $1$. Then 
$\varphi_0$ is a $1$-isogeny, as in Example~\ref{Mquasisplit}. Consequently,
the set $\cP$ consists of {\em all} prime powers. This is what is called 
the ``{\em cas g\'en\'eral }'' in \cite[\S 1]{BrMa}.
\item[(II)] Otherwise, there is a unique prime number $p$ such that each
non-zero entry of $Q^\circ$ has the property that some positive power of it
is a positive integral power of $p$. In this case, $\cP$ will only consist of
positive real numbers $q$ such that $q\varphi_0$ is a $p$-isogeny for this
prime $p$.
\end{itemize}
For example, consider the root datum of Cartan type $A_1\times A_1$
in Example~\ref{Mstrangexp}(b), where $\varphi_0$ is determined by a
certain matrix of order $2$, denoted $P_0$. Then 
\[ Q=P_0=\left(\begin{array}{cc} 0 & 1 \\ 1  & 0  \end{array}\right)\qquad
\mbox{and}\qquad Q^\circ=\left(\begin{array}{cc} 0 & 1/2  \\ 2 & 0 
\end{array}\right).\]
So we are in case (II) where $p=2$ and $\cP=\{2^m \mid m\geq 1\}$. 
Similarly, the complete root data of the Suzuki and Ree groups 
in Example~\ref{Msuzuki2} are of type (II), where $\cP=\{\sqrt{2}^{2m+1} 
\mid m\geq 0\}$ (for ${^2\!B}_2$, ${^2\!F}_4$) or $\cP=
\{\sqrt{3}^{2m+1} \mid m\geq 0\}$ (for ${^2\!G}_2$). 
\end{rem}

\begin{defn}[See \protect{\cite[p.~250]{BrMa}, \cite[1.5]{BMM0}}] 
\label{Mennola} The \nm{Ennola dual} of a complete root datum
$\G=\bigl((X,R,Y,R^\vee),\varphi_0 \bW\bigr)$ is defined by 
\[\G^-:= \bigl((X,R,Y,R^\vee),-\varphi_0 \bW\bigr).\]
(Note that $\G^-$ is a complete root datum since, for any $p$-isogeny of
root data $\varphi\colon X\rightarrow X$, the map $-\varphi$ also is a 
$p$-isogeny of root data; in particular, $\cP_{\G^-}=\cP_{\G}$.)
In this situation, we write $\G(-q):=\G^-(q)$ for any $q\in \cP_{\G}$. 
We have
\[ |\G^-|(y)=(-1)^{\text{rank} X} |\G|(-y).\]
For the origin of the name ``Ennola dual'', see Example~\ref{Mennola1}
below.
\end{defn}

\begin{exmp} \label{Mennola1} (a) Assume that $-\id_X \in\bW$. Then, clearly,
we have $\G^-=\G$ and $|\G^-|=|\G|\in \R[y]$. 

(b) Let $\bG=\GL_n(k)$ and $\bT_0\subseteq \bG$ be the maximal torus 
consisting of the diagonal matrices in $\bG$. We have described the 
corresponding root datum in Example~\ref{rootdatGL}. If we set $\varphi_0=
\id_{X_\R}$ and denote by $\G$ the corresponding complete root datum, 
then $\G(q)\cong \GL_n(q)$ for all prime powers $q$. We claim that
\[\G(-q)\cong \GU_n(q)\qquad \mbox{for all $q\in \cP_\G$}.\]
This is seen as follows. Let $\tau\colon \bG\rightarrow \bG$ be the
automorphism which sends an invertible  matrix to its transpose inverse.
Then $\tau(\bT_0)=\bT_0$ and the induced map on $X$ is $-\id_X$. (Thus, 
$\tau$ is a concrete realisation of the isogeny in Example~\ref{Mopposite}.) 
Let $F_q\colon\bG\rightarrow\bG$ be the standard Frobenius map (raising
every matrix entry to its $q$-th power). Then $\tau$ commutes with $F_q$
and so $F'=\tau\circ F_q$ is a Frobenius map on $\bG$; see
Remark~\ref{MintrinsF}(c). We have $F'(\bT_0)=\bT_0$ and the induced map on 
$X$ is given by $-q\,\id_X$. Thus, $(\bG,F')$ gives rise to the complete
root datum $\G^-$; finally, note that $\bG^{F'}\cong \GU_n(q)$. (The 
difference between this realisation of $\GU_n(q)$ and the one in
Example~\ref{Mquasisplit} is that, here, $\bT_0$ is not a maximally split
torus for $F'$.) Now the identity $|\G^-|(y)=(-1)^{\text{rank} X} |\G|(-y)$
gives an {\em a priori} explanation for the fact that the order 
formula for $\GU_n(q)$ in Table~\ref{Mordertab} (p.~\pageref{Mordertab}) 
is obtained from that of $\GL_n(q)$ by simply changing $q$ to $-q$ (and 
fixing the total sign). Ennola \cite{En63} observed that a similar statement
should even be true for the irreducible characters of these groups; we will 
discuss this in further detail at a later stage.
\end{exmp}

\begin{exmp} \label{Mdefdualcr} The \nm{dual complete root datum} of
$\G=\bigl((X,R,Y,R^\vee),\varphi_0 \bW\bigr)$ is defined by 
\[\G^*:= \bigl((Y,R^\vee,X,R),\varphi_0^\trp \bW\bigr),\]
where $\varphi_0^\trp\colon Y_{\R}\rightarrow Y_{\R}$ is the transpose map
defined, as in \ref{Mhomrootdata}, through the canonical extension of the 
pairing $\langle \;,\; \rangle \colon X\times Y\rightarrow \Z$ to a pairing 
$X_\R\times Y_\R \rightarrow \R$. Here, we also use the identification
of $\bW^\vee\subseteq \Aut(Y)$ with $\bW$, as in Remark~\ref{MidentW}. We 
have $\cP_{\G^*}=\cP_\G$; furthermore, by 
Remark~\ref{Morderform1}(b), we also have $|\G^*|(y)=|\G|(y)$. 
Now, for each $q\in \cP_\G$, we obtain a finite group $\G(q)$ (arising 
from a pair $(\bG,F)$ as in Remark~\ref{MdefBrMa0}) and a finite group
$\G^*(q)$ (arising from an analogous pair $(\bG^*,F^*)$). We then see that 
$(\bG,F)$ and $(\bG^*,F^*)$ are in duality as in Definition~\ref{Mdefdual}.
\end{exmp}

\begin{defn}[See \protect{\cite[1.1]{BrMa}}] \label{MdefBrMa4} Let 
$\G=\bigl((X,R,Y,R^\vee),\varphi_0\bW\bigr)$ be a complete root datum. 
For any $w \in \bW$, the complete root datum 
\[\G_w:=\bigl((X,\varnothing, Y, \varnothing),\varphi_0w^{-1}
\bigr)\]
is called a \nm{maximal toric sub-datum} of $\G$. (We choose $w^{-1}$ here
in order to have consistency with the order formulae below and the 
notation in \cite[2.1]{LuB}.) In general, $\G$ is said to be a \nm{toric 
datum} if $R=\varnothing$; in this case, a corresponding connected reductive 
algebraic group is a torus.
\end{defn}

\begin{abs} \label{Mftori} Let $\G=\bigl((X,R,Y,R^\vee),\varphi_0\bW\bigr)$ 
be a complete root datum. We assume that $\varphi_0$ is chosen 
such that the permutation of $R$ induced by $\varphi_0$ leaves a base of 
$R$ invariant (which is possible by Remark~\ref{MdefBrMa3}). Thus, if
$q\in \cP$ and $\bG,\bT,F$ are as in Remark~\ref{MdefBrMa0}, then we 
are in the setting in \ref{MstdcalC}, where $\bT_0:=\bT$. 

Now let $w \in \bW$ and consider the maximal toric sub-datum $\G_w$ in 
Definition~\ref{MdefBrMa4}. The corresponding order polynomial is just 
given by 
\[ |\G_w|=\det(y\, \id_{X_\R}-w\varphi_0^{-1})\in \R[y].\]
Let $q \in \cP$ and $\bG,\bT_0,F$ be as above. Let $\dot{w}$ be a 
representative of $w$ in $N_{\bG}(\bT_0)$. By Theorem~\ref{langst} 
(Lang--Steinberg), we can write $\dot{w}=g^{-1}F(g)$ for some $g\in \bG$. 
Then $\bT':=g\bT_0g^{-1}$ is an $F$-stable maximal torus of $\bG$; we say 
that $\bT'$ is a \nm{torus of type $w$}. Now conjugation with $g$ defines 
an isomorphism of algebraic groups from $\bT'$ onto $\bT_0$. This 
isomorphism sends $\bT'^F$ onto the subgroup
\[ \bT_0[w]:=\{t\in \bT_0 \mid F(t)=\dot{w}^{-1}t\dot{w}\}\subseteq\bT_0.\]
(Another common notation for this subgroup is $\bT_0^{wF}$.)
Note that $\bT_0[w]$ only depends on $w$, but not on the choice of $\dot{w}$.
Furthermore, the isomorphism $\bT'\cong \bT_0$ induces an isomorphism of 
abelian groups $X(\bT') \cong X(\bT_0)$. One easily checks that, under
this isomorphism and the identification $X=X(\bT_0)$, the map induced
by $F$ on $X(\bT')$ corresponds to the map $\varphi_0 w^{-1}\colon X 
\rightarrow X$. (See \cite[3.3.4]{Ca2} for further details.)
By \cite[3.3.5]{Ca2}, this implies that 
\[ |\bT'^F|=|\bT_0[w]|=|\G_w|(q)=\det(q\,\id_{X_\R}-w\varphi_0^{-1}).\] 
Thus, $\bT'^F\cong \bT_0[w]$ is a member of the series of finite groups of
Lie type defined by the complete root datum $\G_w$. Since $|\bT'^F|$
divides $|\bG^F|$ and since this holds for all $q \in \cP$, we conclude that 
\[ |\G_w| \quad \mbox{divides} \quad |\G| \quad \mbox{in $\R[y]$}.\]
Conversely, since all maximal tori are conjugate in $\bG$, an arbitrary 
$F$-stable maximal torus is of the form $g\bT_0 g^{-1}$ for some $g\in \bG$ 
such that $g^{-1}F(g)\in N_{\bG}(\bT_0)$. If $w$ denotes the image of 
$g^{-1}F(g)$ in $\bW$, then $g\bT_0 g^{-1}$ is a torus of type $w$, as 
above. This discussion shows that the subgroups of $\bG^F$ arising from 
$F$-stable maximal tori can all be realized as subgroups of the form 
$\bT_0[w]\subseteq \bT_0$, for various $w\in \bW$. 
\end{abs}

In later chapters, we will see that various other classes of subgroups
of $\bG^F$ fit into the framework of complete root data. The general
formalism is further developed in \cite{BMM0}, \cite{BMM1}, \cite{BMM2}.
Note that, even with our slightly more general definition, any complete  
root datum as above defines a {\em reflection datum} as in 
\cite[Def.~2.6]{BMM2}, over a suitable subfield $K\subseteq \R$.

\section{Regular embeddings} \label{sec:regemb}

Lusztig's work \cite{LuB}, \cite{Lu5} (to be discussed in more detail in 
later chapters) shows that the character theory of finite groups of Lie 
type is considerably easier when the center of the underlying algebraic 
group is connected. Thus, when trying to prove a result about a general 
finite group of Lie type, it often happens that one first tries to 
establish this result in the case where the center is connected. The
concept of ``regular embedding'' provides a technical tool in order
to reduce a general statement to the connected center case. A major 
result on representations in this context will be stated in 
Theorem~\ref{multfree}. 

\begin{defn} \label{Mdefregemb}
Let $\bG$, $\bG'$ be connected reductive algebraic groups over 
$k=\overline{\F}_p$ and $F\colon \bG\rightarrow \bG$, $F'\colon \bG'
\rightarrow \bG'$ be Steinberg maps. Let $i\colon \bG \rightarrow \bG'$ 
be a homomorphism of algebraic groups such that $i\circ F=F'\circ i$.
Following \cite[\S 7]{Lu5}, we say that $i$ is a \nm{regular embedding} 
if $\bG'$ has connected center, $i$ is an isomorphism of $\bG$ with a 
closed subgroup of $\bG'$ and $i(\bG), \bG'$ have the same derived subgroup.

Note that $\bGder'\subseteq i(\bG)$ and so $i(\bG)$ is normal in
$\bG'$ with $\bG'/i(\bG)$ abelian. Then the finite group $i(\bG^F)=
i(\bG)^{F'}$ contains the derived subgroup of the finite group 
$\bG'^{F'}$ and so $i(\bG^F)$ is normal in $\bG'^{F'}$ with 
$\bG'^{F'}/i(\bG^F)$ abelian. Thus, as far as the representation theory
of $\bG^F$ and of $\bG'^{F'}$ is concerned, we are in a situation where 
Clifford theory (with abelian factor group) applies.
\end{defn}

\begin{exmp} \label{Mregembexp1}
(a) Let $\bG$ be a connected reductive algebraic group with a connected 
center and $F\colon \bG\rightarrow\bG$ be a Steinberg map. Then 
$\bGder$ is semisimple and, clearly, $\bGder \subseteq \bG$ is a 
regular embedding. A standard example is given by $\bG=\GL_n(k)$ where 
$\bGder=\SL_n(k)$; note that this works for both the Frobenius maps in 
Example~\ref{Mquasisplit}, where either $\bG^F=\GL_n(q)$ or $\bG^F=\GU_n(q)$.

(b) Let $\bG=\SL_n(k)$ and $F\colon\bG\rightarrow \bG$ be a Frobenius map.
Without using $\GL_n(k)$ directly, we can {\it construct} a regular embedding 
$i\colon\bG\rightarrow \bG'$ as follows. Let 
\[\bG':=\{(A,\xi)\in M_n(k) \times \bkm\mid \xi\det(A)=1\};\]
then $\bZ(\bG')=\{(\xi I_n,\xi^{-n})\mid \xi\in\bkm\}$ is connected and 
$\dim \bZ(\bG')=1$. For $A\in\bG$ we set $i(A):=(A,1)\in\bG'$; then $i$ is 
a closed embedding. A Frobenius map $F' \colon\bG'\rightarrow\bG'$ is 
defined by $F'(A,\xi)=(F(A),\xi^q)$ if $\bG^F=\SL_n(q)$, and by $F'(A,
\xi) =(F(A), \xi^{-q})$ if $\bG^F=\SU_n(q)$. In Lemma~\ref{Mregemex} below,
the basic idea of this construction is generalised to an arbitrary connected 
reductive group~$\bG$.

(c) Let $n\geq 2$ and $\bG\subseteq \GL_n(k)$ be one of the classical 
groups in \ref{subsecclassic}. Then one can apply a similar construction
as in (b). In each case, $\bZ(\bG)$ consists of the scalar matrices in 
$\bG$. If $\bG=\mbox{SO}_2(k)$, then $\bG=\bZ(\bG) \cong \bkm$; 
otherwise, we have $\bZ(\bG)=\{\pm I_n\}$ where $I_n$ is the identity 
matrix. So let us now assume that $\mbox{char}(k)\neq 2$ and $\bZ(\bG)=
\{\pm I_n\}$. Then $\bG=\Gamma(Q_n,k)$ where $Q_n^\trp=\pm Q_n$. We set 
\[\bG'=C\Gamma(Q_n,k):=\{(A,\xi)\in M_n(k)\times \bkm\mid  A^\trp Q_nA=
\xi Q_n\}; \]
this is called the \nm{conformal group} corresponding to $\bG$. Then
$\bG'$ is a linear algebraic group such that $\bZ(\bG')=\{(\xi I_n,\xi^2)
\mid \xi\in\bkm\}$ is connected and $\dim \bZ(\bG')=1$. Consider the closed 
subgroup $\bG_1=\{(A,1)\mid A\in \bG\} \subseteq \bG'$. Then we have an 
injective homomorphism of algebraic groups $i \colon \bG\rightarrow 
\bG_1$, $A\mapsto (A,1)$, with inverse given by $(A,1)\mapsto A$. Hence,
$i$ is a closed embedding. 

Let $F\colon\GL_n(k)\rightarrow \GL_n(k)$ be the standard Frobenius map 
(raising each entry of a matrix to its $q$th power). Then $F$ restricts 
to a Frobenius map on $\bG$. The map $F'\colon \bG'\rightarrow \bG'$, 
$(A,\xi)\mapsto (F(A),\xi^q)$, is easily seen to be a Frobenius map such 
that $i\circ F=F'\circ i$. Thus, $i$ is a regular embedding. 

If $n$ is even and $\bG=\mbox{SO}_{n}(k)$, then we also have a ``twisted'' 
Frobenius map $F_1\colon \bG\rightarrow \bG$, $A\mapsto t_n^{-1}F(A)t_n$, 
where 
\[\renewcommand{\arraystretch}{1.2} t_n:=
\left[\begin{array}{c|c|c} \; I_{m-1}\; &0&0 \\ \hline  0 & 
\begin{array}{cc}0 & 1 \\ 1 & 0\end{array} & 0  \\\hline 0&0 &
\;  I_{m-1}\;\end{array}\right]\in \mbox{GO}_{2m}(k) \qquad (m=n/2).\] 
This gives rise to the finite ``non-split'' orthogonal group $\bG^{F_1}=
\mbox{SO}_{n}^-(q)$; see \cite[4.1.10(d)]{mybook}. Again, we have a
Frobenius map $F_1'\colon \bG'\rightarrow \bG'$, $(A,\xi)\mapsto (F_1(A),
\xi^q)$, such that $i\circ F_1=F_1'\circ i$. Thus, $i$
also is a regular embedding with respect to $F_1$. (These examples
already appeared in \cite[\S 8.1]{Lu1}.)
\end{exmp}

We have the following general existence result.

\begin{lem}[Cf.\ \protect{\cite[1.21]{DeLu}}] \label{Mregemex} Let $\bG$ 
be connected reductive and $F\colon \bG\rightarrow \bG$ be a Steinberg 
map. Let $\bZ$ be the center of $\bG$ and $\bS\subseteq \bG$ be an
$F$-stable torus such that $\bZ \subseteq \bS$. (For example, one could 
take any $F$-stable maximal torus of $\bG$.) Let $\bG'$ be the quotient 
of $\bG \times \bS$ by the closed normal subgroup $\{(z,z^{-1}) \mid z 
\in \bZ\}$. Let $\bS'$ be the image of $\{1\}\times \bS\subseteq \bG
\times\bS$ in $\bG'$. Then the map $F'\colon \bG' \rightarrow \bG'$ 
induced by $F$ is a Steinberg map and the map $i \colon \bG \rightarrow \bG'$ 
induced by $\bG \rightarrow \bG\times \bS$, $g\mapsto (g,1)$, is a 
regular embedding, where $\bS'$ is the center of $\bG'$.
\end{lem}

\begin{proof} By Lemma~\ref{Mgenfrob3}, $\bG'$ is reductive and,
by Lemma~\ref{Mgenfrob2}, $F'\colon \bG'\rightarrow \bG'$ is a Steinberg 
map. Furthermore, one easily sees that $i$ is injective, 
that $i\circ F=F'\circ i$ and that $\bS'=\bZ(\bG')$. (Thus, the center 
of $\bG'$ indeed is connected.) Let $\bZ':=\{(z,z^{-1}) \mid z \in 
\bZ\}$ and $\bH:=i(\bG)=(\bG\times\bZ)/\bZ' \subseteq\bG'$. We have 
$\bG'=\bH.\bS'$. Since $\bS'=\bZ(\bG')$, it follows that $\bGder'=
\bH_{\text{der}}=i(\bGder)$. Now we claim that $i_1\colon\bG \rightarrow 
\bH$, $g\mapsto \iota(g)$, is an isomorphism of algebraic groups. To 
see this, consider the homomorphism $\pi\colon \bG\times \bZ \rightarrow 
\bG$, $(g,z)\mapsto gz$. Since $\bZ' \subseteq \ker(\pi)$, we have an 
induced homomorphism $\bar{\pi} \colon\bH \rightarrow \bG$, which is 
obviously inverse to $i_1\colon \bG\rightarrow \bH$. Thus, the claim 
is proved, and it follows that $i$ is a regular embedding.  
\end{proof}

\begin{exmp} \label{Mexpregemb} Let $\bG$ be simple of 
simply-connected type and $F\colon \bG\rightarrow \bG$ be a Steinberg
map. Assume that $\bZ=\bZ(\bG)$ is non-trivial. Let $\bT_0\subseteq \bG$ 
be a maximally split torus and $\bS\subseteq \bT_0$ be a subtorus as in 
Example~\ref{Mcentersc}. Then $\bZ\subseteq \bS$ and we would like to 
perform the construction in Lemma~\ref{Mregemex} using $\bS$. For this 
purpose, we need to check that $\bS$ is $F$-stable. 

To see this, let $\cR=(X,R,Y,R^\vee)$ be the root datum of $\bG$ with 
respect to $\bT_0$. Let $\Pi=\{\alpha_1,\ldots,\alpha_n\}$ be a base for 
$R$, with a labelling as in Table~\ref{Mdynkintbl}. Then $Y=\Z R^\vee$ 
and $\{\alpha_1^\vee, \ldots,\alpha_n^\vee\}$ is a $\Z$-basis of $Y$. 
Now $F$ induces a linear map $\varphi\colon X\rightarrow X$ which is a 
$p$-isogeny of $\cR$. Hence, there is a permutation $i\mapsto i'$ of
$\{1,\ldots,n\}$ and there are integers $q_i>0$ (each an integral power 
of~$p$) such that 
\[\varphi^{\trp}(\alpha_{i'}^\vee)=q_{i'}\alpha_i^\vee \qquad \mbox{for 
$1\leq i \leq n$}.\]
(Note: The permutation $\alpha_i\mapsto \alpha_{i'}$ is inverse to the
permutation $\alpha_i\mapsto \alpha_{i^\dagger}$ in \ref{MstdcalC}.) 
Now consider the isomorphism $Y\otimes_\Z \bkm\rightarrow \bT_0$, $\xi
\otimes \nu \mapsto\nu(\xi)$. Then, for each $\xi\in\bkm$ and $\nu\in Y$, 
the element $\xi\otimes \varphi^\trp(\nu)$ corresponds to $F(\nu(\xi))$.
(See \cite[\S 3.2]{Ca2}.) Thus, in terms of the notation $\bT_0=\{h(\xi_1,
\ldots,\xi_n) \mid \xi_1,\ldots, \xi_n\in\bkm\}$ in Example~\ref{Mcentersc}, 
the action of $F$ on $\bT_0$ is given by 
\[ F\bigl(h(\xi_1,\xi_2,\ldots,\xi_n)\bigr)=h(\xi_{1'}^{q_{1'}},
\xi_{2'}^{q_{2'}}, \ldots,\xi_{n'}^{q_{n'}}) \qquad \mbox{for all $\xi_1,
\ldots, \xi_n\in\bkm$}.\]
First note that we do not need to consider the case where $F$ is a Steinberg 
map but not a Frobenius map. For, this case only occurs in types $B_2$,
$G_2$, $F_4$ and, in all three cases, we have $\bZ=\{1\}$ (since 
$\mbox{char}(k)=2$ in type $B_2$). So let now $F$ be a Frobenius map. 
Then all $q_i$ are equal, to $q$ say, and $i\mapsto i'$ determines a 
symmetry of the Dynkin diagram of $\cR$. If $i\mapsto i'$ is the identity, 
then 
\[F\bigl(h(\xi_1,\xi_2,\ldots, \xi_n)\bigr)=h(\xi_1^q, \xi_2^q, \ldots,
\xi_n^q)\qquad \mbox{for all $\xi_i\in\bkm$}.\]
The cases where there exists a non-trivial permutation $i\mapsto i'$ are 
as follows. 
\begin{itemize}
\item[$A_n$:] $i'=n+1-i$ for $1\leq i\leq n$ and so
\[ F\bigl(h(\xi_1,\xi_2,\ldots,\xi_n)\bigr)=h(\xi_n^q, \xi_{n-1}^q, 
\ldots,\xi_1^q)\qquad \mbox{for all $\xi_i\in\bkm$}.\]
\item[$D_n$:] $1'=2$, $2'=1$, $i'=i$ for $3\leq i\leq n$ and so 
\[ F\bigl(h(\xi_1,\xi_2,\ldots,\xi_n)\bigr)=h(\xi_2^q, \xi_1^q, 
\xi_3^q,\ldots,\xi_n^q)\qquad \mbox{for all $\xi_i\in\bkm$}.\]
\item[$D_4$:] $1'=2$, $2'=4$, $3'=3$, $4'=1$ and so 
\[ F\bigl(h(\xi_1,\xi_2,\xi_3,\xi_4)\bigr)=h(\xi_2^q,\xi_4^q, 
\xi_3^q,\xi_1^q)\qquad \mbox{for all $\xi_i\in\bkm$}.\]
\item[$E_6$:] $1'=6$, $2'=2$, $3'=5$, $4'=4$, $5'=3$, $6'=1$ and so 
\[ F\bigl(h(\xi_1,\xi_2,\xi_3,\xi_4,\xi_5,\xi_6)\bigr)=h(\xi_6^q,\xi_2^q, 
\xi_5^q,\xi_4^q,\xi_3^q,\xi_1^q)\qquad \mbox{for all $\xi_i\in\bkm$}.\]
\end{itemize}
In each case, the description in Example~\ref{Mcentersc} 
immediately shows that $\bS$ is $F$-stable. 
\end{exmp}

\begin{prop}[Cf.\ \protect{\cite[\S 14.1]{LuB}, \cite[\S 10]{Lu5}, 
\cite[5.3]{Lu08a}}] \label{Mexpregemb1} Assume that $\bG$ is simple and 
let $F\colon \bG \rightarrow \bG$ be a Steinberg map. Then there exists 
a regular embedding $i \colon \bG\rightarrow \bG'$ with the following 
properties.
\begin{itemize}
\item[(a)] If $\bG$ is of simply-connected type $D_n$ with $n$ even, 
$\operatorname{char}(k)\neq 2$ and $\bG^F$ is ``untwisted'' (i.e., 
the above permutation $i\mapsto i'$ is the identity), then $\dim \bZ(\bG')=2$
and there is a surjective map $\bG'^F/i(\bG^F)\rightarrow \Z/2\Z\times
\Z/2\Z$.
\item[(b)] In all other cases, $\bG'^F/i(\bG^F)$ is cyclic.  
\end{itemize}
\end{prop}

\begin{proof} We begin by noting that, if $\bZ(\bG)=\{1\}$, then we can take 
$\bG'=\bG$ and $i$ the identity; we have $\bG'^F=i(\bG^F)$ in this 
case. So, for the remainder of the proof, we can assume that 
$\bZ(\bG)\neq\{1\}$. 

First we deal with the case where $\bG$ is simple of simply-connected
type. Let $\bS$ be the torus in Example~\ref{Mcentersc}. We have $\bZ(\bG)
\subseteq \bS$ and $\bS$ is $F$-stable by the discussion in 
Example~\ref{Mexpregemb}. So, applying Lemma~\ref{Mregemex} with $\bS$, we 
obtain a regular embedding $i \colon \bG\rightarrow \bG'$. Since 
$\bG=\bGder$, we have $\bGder'= i(\bG)$. Then $i(\bG^F)=
i(\bG)^{F'}=\bGder'^{F'}$ and so Proposition~\ref{Msameorder}(b) 
implies that
\[ \bG'^F/i(\bG^F)\cong \bK^{F'} \qquad \mbox{where} \qquad \bK:=
\bG'/\bGder'.\]
(We have an induced action of $F'$ on $\bK$ by Lemma~\ref{Mgenfrob2}.)
Since $\bG'=i(\bG).\bS'$ and $\bS'=\bZ(\bG')$, the inclusion $\bS'
\hookrightarrow \bG'$ induces an isogeny  $\bS'\rightarrow \bK$.
Composition with the map $\bS\rightarrow \bS'$ from Lemma~\ref{Mregemex} 
yields an isogeny $f\colon \bS\rightarrow \bK$ such that $f\circ F=
F'\circ f$ and $\ker(f)=\bZ(\bG)$. In particular, $\bK$ is a torus and
$\dim \bK=\dim \bS \leq 2$.  

If $\dim \bK=0$, then $i(\bG^F)=\bG'^F$. If $\dim \bK=1$, then $\bK\cong
\bkm$ and so $\bK^{F'}$ is isomorphic to a finite subgroup of $\bkm$; 
hence, $\bK^{F'}$ is cyclic in this case.

Finally, assume that $\dim \bK=2$. This case only occurs in type $D_n$
with $n$ even, where $\bG=\mbox{Spin}_{2n}(k)$ and $\bZ(\bG)=\{t\in\bS\mid 
t^2=1\}$. Since $\bZ(\bG)\neq \{1\}$, we have $\mbox{char}(k)\neq 2$ and 
$\bZ(\bG)\cong \Z/2\Z\times \Z/2\Z$. If $F$ is ``untwisted'', then 
Remark~\ref{Mremregemb1}(c) below will show that $\bG'^F/i(\bG^F)$ 
has a factor group isomorphic to $\Z/2\Z \times \Z/2\Z$. This completes
the proof of (a). 

It remains to consider the case where $F$ is ''twisted'', with all
root exponents equal to~$q$. By the description in 
Example~\ref{Mexpregemb}, there is an isomorphism of algebraic groups 
$\bS \cong \bkm\times \bkm$ such that the action of $F$ on $\bS$ 
corresponds to the map $(s_1,s_2)\mapsto (s_2^q,s_1^q)$ on $\bkm\times\bkm$. 
Consequently, $\bS^{F}\cong \F_{q^2}^\times$ is cyclic. We want to show
that a similar argument works for $\bK$. To see this, let 
$\{\varepsilon_1,\varepsilon_2\}$ be a $\Z$-basis of $X(\bS)$. Then $F$
induces the linear map $\varphi\colon X(\bS)\rightarrow X(\bS)$ such
that $\varphi(\varepsilon_1)=q\varepsilon_2$ and $\varphi(\varepsilon_2)=
q\varepsilon_1$. Now consider the isogeny $f\colon \bS\rightarrow\bK$ 
mentioned above. Since it has kernel $\bZ(\bG)$, and since $\mbox{char}(k)
\neq 2$, the correspondences in \ref{subsec17} show that 
\[X(\bS)/f^*(X(\bK))\cong X(\bZ(\bG)) \cong \Z/2\Z\times \Z/2\Z.\]
Hence, we must have $f^*(X(\bK))=2X(\bS)$. For $i=1,2$, let $\delta_i
\in X(\bK)$ be such that $f^*(\delta_i)=2\varepsilon_i$.
Then $\{\delta_1,\delta_2\}$ is a $\Z$-basis of $X(\bK)$. Let $\beta
\colon X(\bK)\rightarrow X(\bK)$ be the linear map induced by $F'$. Since 
$f\circ F=F'\circ f$, we also have $\varphi\circ f^*=f^*\circ \beta$.
Hence, we must have $\beta(\delta_1)=q\delta_2$ and $\beta(\delta_2)=
q\delta_1$. Then $\bK \rightarrow \bkm \times \bkm$, $t\mapsto 
(\delta_1(t), \delta_2(t))$, is an isomorphism of algebraic groups such 
that the action of $F'$ on $\bK$ corresponds to the map $(t_1,t_2)
\mapsto (t_2^q, t_1^q)$ on $\bkm\times \bkm$. Consequently, $\bK^{F'}
\cong \F_{q^2}^\times$ is also cyclic.

This settles all cases where $\bG$ is simple of simply-connected type.
Now let $\bG_1$ be simple and $F_1\colon \bG_1 \rightarrow \bG_1$ be a 
Steinberg map. We can assume that $\bZ(\bG_1)\neq \{1\}$ and that $\bG_1$ 
is not of simply-connected type. By Example~\ref{Mvartypes}, there
are only two cases to consider: $\bG_1$ of type $A_n$ or $D_n$. 

Let $\bG$ be simple of simply-connected type such that $\bG,\bG_1$ have 
the same Cartan type. By Proposition~\ref{st916}, we can find an isogeny
$f\colon\bG\rightarrow \bG_1$ and a Steinberg map $F\colon\bG\rightarrow
\bG$ such that $f\circ F=F_1\circ f$. We have $\ker(f)\subseteq\bZ(\bG)$ 
and $f(\bZ(\bG))=\bZ(\bG_1)$. If $\bG_1$ is not of type $D_n$ with $n$ 
even, then let $\bS$ be an $F$-stable torus in $\bG$ with $\bZ(\bG)\subseteq 
\bS$ and $\dim \bS\leq 1$, as above. Then $\bS_1:=f(\bS) \subseteq \bG_1$
is an $F_1$-stable maximal torus such that $\bZ(\bG_1)\subseteq \bS_1$ and
$\dim \bS_1\leq 1$. Performing the construction in Lemma~\ref{Mregemex} 
on $\bG_1$ using $\bS_1$, we obtain a regular embedding $i_1\colon 
\bG_1\rightarrow \bG_1'$ such that $\bG_1'^{F_1'}/ i_1(\bG_1^{F_1})$
is cyclic, by the same argument as above. 

It remains to consider the case where $\bG_1$ is of type $D_n$ with $n$ 
even. Since we are also in the case where $\bZ(\bG_1) \neq \{1\}$ and 
$\bG_1$ is not of simply-connected type, we must have $\mbox{char}(k)
\neq 2$ and $|\bZ(\bG_1)|=|\ker(f)|=2$. Assume first that 
$\bG_1^{F_1}$ is untwisted. Then we can argue as follows. Recall that 
\[ \ker(f)\subseteq \bZ(\bG)=\{h(\xi_1,\xi_2,1, \xi_1\xi_2,1, \xi_1 
\xi_2,\ldots )\mid \xi_1^2=\xi_2^2=1\}.\]
Let $\tilde{\bS}$ be one of the following $1$-dimensional subtori of $\bS$:
\[ \{h(\xi,1,1, \xi,1, \xi,\ldots )\mid \xi\in\bkm\} \quad \mbox{or}
\quad \{h(1,\xi,1,\xi,1, \xi,\ldots )\mid \xi\in\bkm\}.\]
Each of these is $F$-stable, and we can choose $\tilde{\bS}$ such that
$\bZ(\bG) \subseteq \ker(f).\tilde{\bS}$. But, in this case, $\bZ(\bG_1)
\subseteq f(\tilde{\bS})$ and so the same construction as above, using 
$\bS_1=f(\tilde{\bS})$, yields the desired conclusion. 
Finally, if $\bG_1^{F_1}$ is twisted, then $\bG_1\cong \mbox{SO}_{2n}(k)$ 
(see Example~\ref{Mvartypes}(c)) and so a regular embedding with the
desired properties is obtained as in Example~\ref{Mregembexp1}(c), using 
the corresponding conformal group.
\end{proof}

The following remark contains a number of useful, purely group-theoretical 
properties of a regular embedding.

\begin{rem} \label{Mremregemb1} Let $i\colon \bG\rightarrow\bG'$ be a 
regular embedding. To simplify notation, we identify $\bG$ with its image 
in $\bG'$ and use the symbol $F$ for both Steinberg maps; thus, $F'=F$, 
$\bG\subseteq \bG'$ and $\bGder=\bGder'$. Let $\bZ$ denote the center of 
$\bG$ and $\bZ'$ denote the center of $\bG'$. Then it is rather
straightforwad to prove the following results (see \cite[\S 1]{Le78} 
for details).
\begin{itemize}
\item[(a)] We have $\bZ=\bZ'\cap \bG$ and $\bZ^F=\bZ'^F\cap \bG^F$. 
Furthermore, the inclusion $\bG\subseteq \bG'$ induces isomorphisms
$\bG/\bZ\cong \bG'/\bZ'$ and $(\bG/\bZ)^F \cong \bG'^F/\bZ'^F$.
\item[(b)] Let $\bT$ be an $F$-stable maximal torus of $\bG$. Then $\bT'
:=\bT.\bZ'$ is an $F$-stable maximal torus of $\bG'$, and every $F$-stable 
maximal torus of $\bG'$ is of this form. In this situation, we have 
\[\bT=\bG\cap \bT', \qquad \bG'^F=\bG^F.\bT'^F, \qquad \bT^F=\bG^F\cap 
\bT'^F.\]
Furthermore, the inclusion $\bT\subseteq \bT'$ induces an isomorphism 
$N_{\bG}(\bT)/\bT\cong N_{\bG'}(\bT')/\bT'$ which is compatible with the 
action of $F$ on both sides. 
\item[(c)] Let $\bT,\bT'$ be as in (b). Then there are canonical exact 
sequences 
\[\renewcommand{\arraystretch}{1.3} \begin{array}{ccccccc}
\{1\} & \longrightarrow & \bG^F.\bZ'^F & \longrightarrow  &\bG'^F & 
\longrightarrow & (\bZ/\bZ^\circ)_F \\ \{1\} & \longrightarrow & \bT^F.
\bZ'^F & \longrightarrow & \bT'^F & \longrightarrow & (\bZ/\bZ^\circ)_F,
\end{array}\]
where $(\bZ/\bZ^\circ)_F$ is defined in Remark~\ref{Madjquot}; the map
$\bG'^F\rightarrow (\bZ/\bZ^\circ)_F$ is given by sending $g'\in\bG'^F$ 
to $g^{-1}F(g)$ where $g\in \bG$ is such that $g\in g'\bZ'$ (which exists
by (a)). The map $\bT'^F\rightarrow (\bZ/\bZ^\circ)_F$ is given similarly.
\end{itemize}
It follows from (c) that, if $\bZ$ is connected, then $\bG'^F=\bG^F.\bZ'^F$.
\end{rem}
 
\begin{lem}[Cf.\ \protect{\cite[p.~164]{Lu5}}] \label{Mluslem164} Let 
$\bG \subseteq \bG'$ be a regular embedding (notation as in 
Remark~\ref{Mremregemb1}) and $\pi_{\text{ad}} \colon \bG \rightarrow 
\bGad$ be an adjoint quotient (see Remark~\ref{Madjquot}). Then 
$\pi_{\operatorname{ad}}$ has a unique extension to an abstract group
homomorphism $\hat{\pi}_{\operatorname{ad}} \colon \bG' \rightarrow 
\bGad$, and this induces a surjective homomorphism
\[ \bG'^F/\bG^F \rightarrow \bGad^F/\pi_{\operatorname{ad}}(\bG^F).\]
\end{lem}

\begin{proof} Suppose that there exists an abstract homomorphism 
$\hat{\pi}_{\text{ad}}\colon \bG'\rightarrow \bGad$ extending 
$\pi_{\text{ad}}$. Then $\hat{\pi}_{\text{ad}} (\bZ') \subseteq 
\bZ(\bGad)=\{1\}$ and so $\bZ'\subseteq \ker(\hat{\pi}_{\text{ad}})$. Since 
$\bG'=\bG.\bZ'$, it follows that $\hat{\pi}_{\text{ad}}(gz')=\pi(g)$ for all 
$g\in\bG$ and $z'\in\bZ'$. Thus, $\hat{\pi}_{\text{ad}}$ is uniquely 
determined (if it exists). Conversely, it is straightforward to check that 
this formula defines an abstract homomorphism $\bG'\rightarrow \bGad$; 
note that $\ker(\pi_{\text{ad}})=\bZ=\bZ'\cap\bG$ (see 
Remark~\ref{Mremregemb1}(a) for the last equality). Furthermore, 
$\hat{\pi}_{\text{ad}}$ commutes with the action of $F$ on both sides 
and we have $\ker(\hat{\pi}_{\text{ad}})=\bZ'$. Hence, since 
$\hat{\pi}_{\text{ad}}$ extends $\pi_{\text{ad}}$, we obtain an induced 
homomorphism $\bG'^F/\bG^F\rightarrow \bGad^F/\pi_{\text{ad}}(\bG^F)$ and 
all that remains to show is that $\hat{\pi}_{\text{ad}}(\bG'^F)=\bGad^F$. 
This is seen as follows. Let $g\in\bGad^F$. Then 
$\hat{\pi}_{\text{ad}}^{-1}(g)$ is an $F$-stable coset of 
$\ker(\hat{\pi}_{\text{ad}})=\bZ'$. Now the connected group $\bZ'$ acts 
transitively on this coset by multiplication and 
Proposition~\ref{Mfrobconj0} shows that $\hat{\pi}_{\text{ad}}^{-1}(g)^F
\neq \varnothing$.
\end{proof}

In order to obtain further properties of regular embeddings, it will
be useful to characterise these maps entirely in terms of root data. In
particular, this will allow us to show how regular embeddings relate 
to dual groups. 

\begin{lem} \label{Mregasai1} Let $\bG$, $\bG'$ be connected reductive 
groups over $k$ and $f\colon \bG \rightarrow \bG'$ be an isotypy (see 
{\rm \ref{Mdefcentralhom}}). Let $\bT\subseteq \bG$ and $\bT'\subseteq\bG'$ 
be maximal tori such that $f(\bT)\subseteq \bT'$. Let $\varphi\colon X(\bT') 
\rightarrow X(\bT)$, $\chi'\mapsto \chi\circ f|_{\bT}$, be the induced 
homomorphism. Then the following two conditions are equivalent.
\begin{itemize}
\item[(i)] $f$ is an isomorphism of $\bG$ onto a closed subgroup of 
$\bG'$.
\item[(ii)] $f$ is a central isotypy and $\varphi$ is surjective.
\end{itemize}
\end{lem}

\begin{proof} First note that the assumptions imply that $\bG'=
f(\bG).\bZ(\bG')$ and $f(\bGder)=\bGder'$. Let $\bG_1:=f(\bG)\subseteq 
\bG'$; this is a closed subgroup which is 
connected and reductive (see Lemma~\ref{Mgenfrob2}); furthermore, 
$\bGder'=f(\bGder)=(\bG_1)_{\text{der}}$. Let $\bT_1:=f(\bT) \subseteq 
\bT'$; then $\bT_1$ is a maximal torus of $\bG_1$ (see 
\ref{Mconnredcomp}(a)) and we have $\bT_1=\bG_1\cap \bT'$. (We have $\bG_1
\cap \bT'\subseteq C_{\bG_1}(\bT_1)=\bT_1$ where the last equality holds 
since $\bG_1$ is connected reductive; the inclusion ``$\subseteq$'' is 
clear.) Thus, we have $f=i \circ f_1$ where $f_1 \colon \bG \rightarrow 
\bG_1$ is the restricted map and $i\colon \bG_1 \hookrightarrow \bG'$ is 
the inclusion; it is clear that $i$ is a central isotypy. (Note 
that $(\bG_1)_{\text{der}}$ contains all the root subgroups of $\bG_1$; see 
Remark~\ref{MdirprodG1}.) Correspondingly, we have a factorisation 
$\varphi=\varphi_1\circ \varepsilon$ where $\varphi_1 \colon X(\bT_1)
\rightarrow X(\bT)$ is induced by $f_1$ and $\varepsilon \colon X(\bT')
\rightarrow X(\bT_1)$ is given by restriction. Note that $\varepsilon$ is 
surjective; see \ref{subsec17}.

Now suppose that (i) holds, that is, $f_1\colon \bG\rightarrow \bG_1$ is an 
isomorphism of algebraic groups. Then the composition $f=i\circ f_1$ will 
be a central isotypy. Furthermore, $\varphi_1\colon X(\bT_1)
\rightarrow X(\bT)$ is an isomorphism of abelian groups. Since 
$\varepsilon$ is surjective, it follows that $\varphi=\varphi_1 \circ 
\varepsilon$ must be surjective. Thus, (ii) holds.

Conversely, assume that (ii) holds. Then $\varphi_1$ is also surjective. 
So the correspondences in \ref{subsec17} show that $f_1\colon \bT 
\rightarrow \bT_1$ is a closed embedding. But, $\dim \bT_1=\dim f_1(\bT)=
\dim \bT - \dim \ker(f_1|_{\bT})=\dim \bT$ and so $f_1\colon \bT 
\rightarrow \bT_1$ is an isomorphism. Then Theorem~\ref{thmiso1}(b) shows 
that $f_1\colon \bG\rightarrow\bG_1$ also is an isomorphism. 
\end{proof}

\begin{cor} \label{Mregasai2}
Let $\bG$, $\bG'$ be connected reductive and $F\colon \bG\rightarrow \bG$, 
$F'\colon \bG' \rightarrow \bG'$ be Steinberg maps. Let $i\colon \bG 
\rightarrow \bG'$ be a homomorphism of algebraic groups such that $i\circ 
F=F'\circ i$ and $i(\bT)\subseteq \bT'$, where $\bT$ is an $F$-stable maximal
torus of $\bG$ and $\bT'$ is an $F'$-stable maximal torus of $\bG'$. Then
$i$ is a regular embedding if and only if the following three conditions hold.
\begin{itemize}
\item[(1)] $i$ is a central isotoy, i.e., the induced map $\varphi
\colon X(\bT')\rightarrow X(\bT)$ is a homomorphism of root data;
\item[(2)] the map $\varphi\colon X(\bT')\rightarrow X(\bT)$ is surjective; 
and
\item[(3)] $X(\bT')/\Z R'$ has no $p'$-torsion, where $R'$ are
the roots relative to~$\bT'$. 
\end{itemize}
\end{cor}

\begin{proof} Suppose that $i$ is a regular embedding. Since 
$i(\bGder)= \bGder'$, we have $\bG'=i(\bG).\bZ(\bG')$. So the general 
assumptions of Lemma~\ref{Mregasai1} plus condition (i) are satisfied. Hence,
the first two conditions hold; the third one holds because $\bZ(\bG')$ is
connected (see Lemma~\ref{Mconncent}). Conversely, if the 
above three conditions are satisfied, then $\bZ(\bG')$ is connected and 
Lemma~\ref{Mregasai1} shows that $i$ is an isomorphism of $\bG$ onto a 
closed subgroup of $\bG'$. Since $i$ is central, we have $\bG'=
i(\bG).\bZ(\bG')$ which implies that $\bGder'=i(\bGder)$. Hence, $i$ is 
a regular embedding. 
\end{proof}

\begin{abs} \label{Mdefdual2} Assume that $(\bG,F)$ and $(\bG^*,F^*)$ are 
in duality (see Definition~\ref{Mdefdual}), with respect to maximally 
split tori $\bT_0\subseteq \bG$ and $\bT_0^*\subseteq \bG^*$. Furthermore, 
assume that $(\bG',F')$ and $(\bG'^*, F'^*)$ are in duality, with respect 
to maximally split tori $\bT_0'\subseteq \bG'$ and $\bT_0'^*\subseteq 
\bG'^*$.  Let $f \colon \bG\rightarrow \bG'$ be a central isotypy such 
that $f\circ F= F'\circ f$ and $f(\bT_0)\subseteq \bT_0'$. Thus, the induced 
map $\varphi\colon X(\bT_0')\rightarrow X(\bT_0)$ is a homomorphism of 
root data as in \ref{Mhomrootdata}. But then the transpose map $\varphi^\trp
\colon Y(\bT_0) \rightarrow Y(\bT_0')$ defines a homomorphism of the dual 
root data. Using the isomorphisms $\delta^\trp\colon Y(\bT_0) \rightarrow
X(\bT_0^*)$ and $\delta'^\trp\colon Y(\bT_0') \rightarrow X(\bT_0'^*)$ 
from Definition~\ref{Mdefdual}, we obtain a map $\hat{\varphi}\colon 
X(\bT_0^*) \rightarrow X(\bT_0'^*)$ which is a homomorphism between the 
root data of $\bG'^*$ and $\bG^*$. Hence, by Theorem~\ref{thmiso1} 
(extended isogeny theorem), there exists a central isotypy 
$f^*\colon \bG'^*\rightarrow \bG^*$ which maps $\bT_0'^*$ 
into $\bT_0^*$ and induces $\hat{\varphi}$. Arguing as in 
Lemma~\ref{lemcompatiso}, one shows that $f^*$ can be chosen such that 
$f^*\circ F'^*=F^*\circ f^*$. In this situation, we say that 
the two central isotypies
\begin{center}
{\it $f\colon \bG\rightarrow \bG'$ and $f^*\colon \bG'^*\rightarrow \bG^*$ 
correspond to each other by duality.}
\end{center}
(This relation is symmetric.) With this notation, we can now state:
\end{abs}

\begin{lem} \label{Mregasai4} Let $f\colon \bG \rightarrow \bG'$ and
$f^*\colon \bG'^*\rightarrow \bG^*$ correspond to each other by duality,
as above. Assume that $f\colon \bG\rightarrow \bG'$ is an isomorphism
with a closed subgroup of $\bG'$. Then the following hold.
\begin{itemize}
\item[{\rm (a)}] $f^*\colon \bG'^*\rightarrow \bG^*$ is surjective and 
$\ker(f^*)$ is a central torus. 
\item[{\rm (b)}] $\bG'/f(\bG)$ is a torus and the pairs $(\ker(f^*),F'^*)$,
$(\bG'/f(\bG),\bar{F}')$ are in duality, where $\bar{F}'\colon 
\bG'/f(\bG)\rightarrow \bG'/f(\bG)$ is induced by $F'$.
\item[{\rm (c)}] The restricted map $f^*\colon (\bG'^*)^{F'^*}\rightarrow 
(\bG^*)^{F^*}$ is surjective.
\end{itemize}
\end{lem}

\begin{proof} (a) We follow \cite[2.5]{Bo3}. 
Let $\bT_0,\bT_0',\bT_0^*,\bT_0'^*$ be as in \ref{Mdefdual2}. By 
restriction, $f$ yields a closed embedding $f\colon \bT_0\rightarrow
\bT_0'$; let $\varphi\colon X(\bT_0')\rightarrow X(\bT_0)$ be the
induced map. Then \ref{subsec17}(a) implies that $\varphi$ is surjective
and $\ker(\varphi)\cong X(\bT_0'/f(\bT_0))$ is torsion-free. Hence, 
taking transpose maps and using the isomorphisms $\delta\colon X(\bT_0) 
\rightarrow Y(\bT_0^*)$ and $\delta'\colon X(\bT_0') \rightarrow 
Y(\bT_0'^*)$ from Definition~\ref{Mdefdual}, we obtain an exact 
sequence
\[ \{0\} \;\longrightarrow \; X(\bT_0^*)\;
\stackrel{\psi}{\longrightarrow} \;X(\bT_0'^*)\;
\longrightarrow \;\Hom(\ker(\varphi), \Z)\;\longrightarrow \;\{0\},\]
where $\psi$ is induced by $f^*\colon \bT_0'^*\rightarrow \bT_0^*$ and 
we have $\psi^\trp\circ \delta'=\delta\circ \varphi$. Now, by
\ref{subsec17}(b), we deduce that $f^*\colon \bT_0'^*\rightarrow \bT_0^*$ 
is surjective and that 
\[ X(\ker(f^*))\cong X(\bT_0'^*)/\psi(X(\bT_0^*))\cong \Hom(\ker(\varphi), 
\Z)\cong \Hom(X(\bT_0'/f(\bT_0)),\Z).\]
Hence, $X(\ker(f^*)$ is torsion-free and so $\ker(f^*)$ must be a torus. 
Since $\bG^*= f^*(\bG'^*).\bT_0^*$ and $f^*(\bT_0'^*)=\bT_0^*$, we also 
have $f^*(\bG'^*)=\bG^*$. 

(b) We follow \cite[2.6]{Bo3}. First note that
the inclusion $\bT_0'\subseteq \bG'$ induces an isomorphism of 
algebraic groups $\bT_0'/f(\bT_0)\rightarrow \bG'/f(\bG)$; in particular,
$\bG'/f(\bG)$ is a torus. Furthermore, from the above proof of (a), we 
deduce that there is an isomorphism $X(\ker(f^*))\rightarrow Y(\bT_0'/
f(\bT_0))$, and one easily checks that this is compatible with the actions 
of $F'$ and $F'^*$.

(c) Since $\ker(f^*)$ is connected, this follows from 
Proposition~\ref{Msameorder}(b).
\end{proof}

\begin{cor} \label{Mregasai4a} Let $i\colon \bG\rightarrow \bG'$ be a
regular embedding and $i^*\colon \bG'^*\rightarrow \bG^*$ be a 
corresponding dual homomorphism, as in \ref{Mdefdual2}. Then there is
an induced (non-canonical) isomorphism $\ker(i^*)^{F^*} 
\stackrel{\sim}{\rightarrow} \Irr(\bG'^F/i(\bG^F))$, $z\mapsto \theta_z$. 
(This isomorphism depends on the same choices as in Remark~\ref{Mdualtori}.)
\end{cor}

\begin{proof} As in the proof of Lemma~\ref{Mregasai4}, the inclusion
$\bT_0'\subseteq \bG'$ induces an isomorphism $\bT_0'/i(\bT_0)\rightarrow 
\bG'/i(\bG)$. Since $i(\bT_0)$ and $i(\bG)$ are connected, we have
an induced isomorphism $\bT_0'^F/i(\bT_0^F)\cong \bG'^F/i(\bG^F)$ and,
hence, an isomorphism $\Irr(\bT_0'^F/i(\bT_0^F))\cong \Irr(\bG'^F/
i(\bG^F))$. By Lemma~\ref{Mregasai4}, $\ker(i^*)$ and $\bG'/i(\bG)$ 
are tori in duality, so it remains to use the isomorphism in 
Corollary~\ref{Mdualtori}.
\end{proof}

The above results will play a certain role in the study of ``Lusztig
series'' of characters of finite groups of Lie type; see 
a later section. The following result is cited in 
\cite[8.8]{LuB}, \cite[8.1]{Lu5}, \cite[0.1]{LuRem} in relation to 
certain reduction arguments; it appears in the unpublished manuscript 
\cite{Asai}.

\begin{lem}[Asai \protect{\cite[\S 2.3]{Asai}}] \label{Mregasai5} Let 
$\bG$ be connected reductive and $F\colon \bG\rightarrow \bG$ a
Steinberg map. Then there exists a connected reductive group 
$\tilde{\bG}$, a Steinberg map $\tilde{F}\colon \tilde{\bG} 
\rightarrow \tilde{\bG}$ and a homomorphism of algebraic groups
$f\colon \tilde{\bG}\rightarrow \bG$, such that:
\begin{itemize}
\item[(a)] $\tilde{\bG}_{\operatorname{der}}$ is semisimple of
simply-connected type,
\item[(b)] $f$ is surjective and $F \circ f=\tilde{F}\circ f$,
\item[(c)] $\ker(f)$ is a central torus of
$\tilde{\bG}$.
\end{itemize}
In particular, $f$ induces a surjective homomorphism of finite groups
$\tilde{\bG}^{\tilde{F}}\rightarrow \bG^F$. Furthermore, if $\bG$ has a
connected center, then $\tilde{\bG}$ has a connected center, too.
\end{lem}

\begin{proof} Asai \cite{Asai} shows this by explicitly constructing
the appropriate root datum for $\tilde{\bG}$ and then using 
Theorem~\ref{thmiso1} (extended isogeny theorem). Here is
a more direct argument. Let $\pi_{\text{sc}}\colon (\bGder)_{\text{sc}} 
\rightarrow \bGder$ be a simply-connected covering of the derived group of 
$\bG$, as in Remark~\ref{Msccover}. Assume first that $\pi_{\text{sc}}$ 
is bijective. Let $\bZ$ be the center of $\bG$. We have an isogeny
\[ f_1\colon (\bGder)_{\text{sc}}\times \bZ^\circ\rightarrow \bG,
\qquad (g,z)\mapsto \pi_{\text{sc}}(g)z.\]
Let $\tilde{\bG}:=((\bGder)_{\text{sc}}\times \bZ^\circ)/\ker(f_1)$. 
Then $f_1$ induces a bijective morphism of algebraic groups 
$f\colon \tilde{\bG}\rightarrow \bG$ and one easily verifies that (a), 
(b), (c) hold. Furthermore, since $f$ is bijective, the center of 
$\tilde{\bG}$ is connected if and only if $\bZ$ is connected.

Now consider the general case, where $\pi_{\text{sc}}$ may not be bijective.
By Lemma~\ref{Mregemex}, there exists a regular embedding $i \colon 
\bG^*\rightarrow \bH$. By duality, we obtain a homomorphism of 
algebraic groups $i^*\colon \bH^*\rightarrow \bG$; note that, as
remarked in Definition~\ref{Mdefdual}, we can identify $(\bG^*)^*$ with 
$\bG$. By Lemma~\ref{Mregasai4}, $i^*$ is surjective and $\ker(i^*)$
is a central torus of $\bH^*$; furthermore, by Lemma~\ref{Mconncent}, the 
simply-connected covering $(\bH_{\text{der}}^*)_{\text{sc}}\rightarrow 
\bH_{\text{der}}^*$ is bijective. By the previous argument, there exists 
a bijective homomorphism of algebraic groups $f_1\colon \tilde{\bG}
\rightarrow \bH^*$ such that (a), (b), (c) hold. Then (a), (b), (c) hold 
for the composition $f=i^*\circ f_1\colon \tilde{\bG}\rightarrow \bG$. 
Finally, assume that $\bZ$ is connected. Now the derived subgroup of $\bH$ 
is isomorphic to that of $\bG^*$. Hence, Lemma~\ref{Mconncent} implies 
that the center of $\bH^*$ is connected as well. Since $f_1$ is bijective, 
it follows that $\tilde{\bG}$ also has a connected center.
\end{proof} 

\begin{exmp} \label{Mregasai6} Assume that $\bG$ is semisimple and let 
$i\colon \bG\rightarrow \bG'$ be a regular embedding. Applying 
Lemma~\ref{Mregasai5} to $\bG'$, we obtain a homomorphism of algebraic 
groups $f\colon \tilde{\bG}'\rightarrow\bG'$ satisfying the above three 
conditions. Furthermore, since $\bZ(\bG')$ is connected, we have that 
$\bZ(\tilde{\bG}')$ is connected, too. Now $f(\tilde{\bG}_{\text{der}}')=
\bGder=\bG$ and so, by restriction, we obtain an isogeny $\tilde{f}
\colon \tilde{\bG}_{\text{der}}' \rightarrow \bG$ which is a 
simply-connected covering of $\bG$. We have a commutative diagram:
\[\renewcommand{\arraystretch}{1.3} \begin{array}{ccc}  
\bG & \stackrel{i}{\longrightarrow} & \bG'\\
\uparrow & &  \uparrow \\ \tilde{\bG}_{\text{der}}' & \hookrightarrow &
\tilde{\bG}'\end{array}\]
In this sense, a simply-connected covering of $\bG$ can always be chosen 
to be compatible with the given regular embedding $i\colon \bG\rightarrow 
\bG'$. (This remark appears in \cite[8.1(d)]{Lu5}; it will be relevant in 
the discussion of Lusztig series in a later section.)
\end{exmp}

The following result was first stated (for $\K$ of characteristic~$0$) by 
Lusztig \cite[Prop.~10]{Lu5}, together with a reduction argument which 
reduces the proof to the case where $\bG$ is simple of simply-connected type. 
As far as such groups are concerned, one can use 
Proposition~\ref{Mexpregemb1}, which shows that type $D_n$ with $n$ even 
is the most complicated case to deal with. A full proof for this case 
appeared only much later, first in \cite[Chap.~16]{CaEn}, and then in 
\cite[\S 5]{Lu08a}.

\begin{thm}[Multiplicity-Freeness Theorem] \label{multfree} Let $i 
\colon \bG\rightarrow \bG'$ be a regular embedding and $\K$ be any
algebraically closed field. Then the restriction of every simple 
$\K\bG'^{F'}$-module to $\bG^F$ (via $i$) is multiplicity-free.
\end{thm}

\begin{proof} We can only sketch the general strategy here, and highlight
where the principal difficulty of the proof lies. First we note 
that the reduction argument described in the proof of \cite[Prop.~10]{Lu5} 
works for simple modules over any algebraically closed field $\K$, not just 
for $\mbox{char}(\K)=0$. (Some adjustments of a different kind are 
required, since Lusztig considers Frobenius maps, not Steinberg maps in
general.) Hence, it suffices to prove the theorem in the case where $\bG$ 
is simple of simply-connected type. Furthermore, the reduction argument
shows that it is sufficient to consider only one particular regular 
embedding $i\colon\bG \rightarrow \bG'$, namely, one satisfying the 
conditions in Proposition~\ref{Mexpregemb1}. So let us now assume that 
these conditions are satisfied. 

If $\bG'^F/i(\bG^F)$ is cyclic, then a 
standard result on representations of finite groups shows that the 
desired assertion holds; see, e.g., \cite[Theorem~III.2.14]{Feit}. (This 
uses that $\K$ is algebraically closed, but works without any assumption 
on $\mbox{char}(\K)$.) 

It remains to consider the case where $\bG$ is of type $D_n$ with $n$ 
even, $\mbox{char}(k) \neq 2$ and $F$ is ``untwisted''. Let us identify $\bG$
with $i(\bG)$ and use the notational conventions in Remark~\ref{Mremregemb1}.
Writing $G=\bG^F$, $G'=\bG'^F$, $H:=G.\bZ'^F$, we have  
\[H\trianglelefteq G'\quad\mbox{and}\quad G'/H \cong (\bZ/\bZ^\circ)_F=
\bZ \cong \Z/2\Z\times \Z/2\Z.\]
Let $V$ be a simple $\K G'$-module and  denote by $V_H$ its restriction to
$H$. Since $H=G.\bZ'^F$ and $\bZ'^F$ is contained in the center of $G'$, 
it is sufficient to show that $V_H$ is multiplicity-free. (To see this, 
one only needs to show that non-isomorphic simple $H$-submodules of 
$V_H$ remain simple and non-isomorphic upon restriction to $G$. And this 
easily follows, for example, by the argument in \cite[p.~265]{bs2}.) Now, 
if $\mbox{char} (\K)=2$, then $V_H$ is multiplicity-free by some general 
results on representations of finite groups; see, e.g., 
\cite[Lemma~3.14]{KlTi}. If $\mbox{char}(\K)= \mbox{char}(k)=p$, then 
$V_H$ is even simple by \cite[Lemma~3.4]{BruLu}. 

So, finally, assume that $\mbox{char}(\K)\neq \mbox{char}(k)$ and
$\mbox{char}(\K)\neq 2$. In particular, $\mbox{char}(\K)$ is either $0$ or
a prime not dividing the index of $H$ in $G'$. By Clifford's Theorem 
(see \cite[Theorem~VII.9.18]{Huppert2}), $V_H$ is semisimple and there 
are two possibilities: either $V_H$ is multiplicity-free (with $1$, $2$ or
$4$ irreducible constituents) or the direct sum of $2$ copies of a simple 
$\K H$-module. In the case where $\mbox{char}(\K)=0$, it is shown by an 
elaborate counting argument (first published in \cite{CaEn}; see also
\cite{Lu08a}) that the second type does not occur. This argument involves:
\begin{itemize}
\item knowledge of the action (by tensor product) of the four 
$1$-dimensional representations of $G'/H$ on the simple $\K G'$-modules;
\item counting conjugucay classes and simple modules for $\mbox{Spin}_{2n}
(q)$. (As noted in \cite[\S 13]{Lu5}, this is ``very long and unpleasant''.) 
\end{itemize}
Finally, it is shown in \cite[\S 3]{bs2}, using the results on basic sets 
of Brauer characters in \cite{bs1}, that Lusztig's argument can be adapted 
to work as well when $\mbox{char}(\K)>0$ (but still $\mbox{char}(\K)\neq 
\mbox{char}(k)$ and $\mbox{char}(\K)\neq 2$). 
\end{proof}

It would be highly desirable to find a more conceptual proof of this 
result which does not rely on a case-by-case analysis and the counting 
arguments for $\mbox{Spin}_{2n}(q)$.


\begin{theindex}

  \item adjoint quotient, 47
  \item adjoint representation, 9
  \item affine algebraic group, 3
  \item affine variety, 2
  \item algebraic $BN$-pair, 11

  \indexspace

  \item base, 14
  \item $BN$-pair, 10
  \item Bruhat cells, 11
  \item Bruhat decomposition, 11

  \indexspace

  \item Cartan matrix, 14, 15
  \item Cartan type, 14
  \item central isogeny, 30
  \item central isotypy, 30
  \item character group, 8
  \item characteristic exponent, 28
  \item co-character group, 8
  \item complete root datum, 55
  \item conformal group, 60
  \item connected, 3

  \indexspace

  \item defined over ${\mathbb  {F}}_q$, 31
  \item diagonal automomorphisms, 47
  \item dual complete root datum, 58
  \item dual root datum, 12
  \item duality, 49
  \item Dynkin diagram, 15

  \indexspace

  \item Ennola dual, 58

  \indexspace

  \item $F$-simple, 48
  \item finite group of Lie type, 33
  \item finite reductive group, 33
  \item Frobenius map, 31
  \item fundamental group, 15, 48

  \indexspace

  \item general orthogonal group, 5
  \item general unitary group, 37
  \item generic finite reductive group, 55

  \indexspace

  \item homomorphism of root data, 12

  \indexspace

  \item indecomposable Cartan matrix, 15
  \item isogeny, 27
  \item isogeny theorem, 27
  \item isotypy, 30

  \indexspace

  \item Jordan decomposition of elements, 6

  \indexspace

  \item length function, 10
  \item linear algebraic group, 3

  \indexspace

  \item maximal toric sub-datum, 58
  \item maximally split, 33, 49, 51

  \indexspace

  \item order polynomial, 56

  \indexspace

  \item $p$-isogeny of root data, 16
  \item perfect pairing, 8
  \item positive roots, 14
  \item preservation results, 27
  \item projective linear group, 26

  \indexspace

  \item rank, 14
  \item rational structure, 31
  \item reduced expression, 10
  \item reductive, 8
  \item reductive $BN$-pair, 23
  \item reductive $BN$-pair, 11
  \item regular embedding, 60
  \item regular functions, 2
  \item representation, 9
  \item root datum, 12
  \item root datum of adjoint type, 19
  \item root datum of simply-connected type, 19
  \item root exponents, 16, 28
  \item root space decomposition, 9
  \item root subgroup, 22
  \item roots, 9

  \indexspace

  \item semidirect product (of algebraic groups), 7
  \item semisimple, 6, 10
  \item semisimple group of adjoint type, 42
  \item semisimple group of simply-connected type, 42
  \item series of finite groups of Lie type, 56
  \item simple algebraic group, 8
  \item simply-connected covering, 47
  \item special linear group, 26
  \item special orthogonal group, 5
  \item standard Frobenius map, 31
  \item Steinberg map, 33
  \item symplectic group, 4

  \indexspace

  \item tangent space, 5
  \item Tits system, 10
  \item toric datum, 58
  \item torus, 6
  \item torus of type $w$, 59

  \indexspace

  \item unipotent, 6
  \item unipotent radical, 7

  \indexspace

  \item weight spaces, 9
  \item weights, 9
  \item Weyl group, 10, 12

  \indexspace

  \item Zariski topology, 2

\end{theindex}

\end{document}